\documentclass[a4paper,12pt]{article}
\usepackage{ucs}
\usepackage{amsfonts, amssymb, amsmath, amsthm, amscd}
\usepackage{graphicx}
\usepackage{cite}
\ifx\undefined \pdfgentounicode \else
\input{glyphtounicode} \pdfgentounicode=1
\fi

\author{A.A. Vasil'eva\footnote{Faculty of Mechanics and Mathematics, Lomonosov Moscow State University; Moscow Center for Fundamental and Applied Mathematics.}}
\title{Kolmogorov widths of an intersection of anisotropic finite-dimensional balls: case $2\le q_j<\infty$}
\date{}
\begin{document}

\maketitle

\newenvironment{Biblio}{%
                  \renewcommand{\refname}{\footnotesize REFERENCES}%
                  \begin{thebibliography}{99}%
                  \itemsep -0.2em}%
                  {\end{thebibliography}}

\def\inff{\mathop{\smash\inf\vphantom\sup}}
\renewcommand{\le}{\leqslant}
\renewcommand{\ge}{\geqslant}
\newcommand{\sgn}{\mathrm {sgn}\,}
\newcommand{\inter}{\mathrm {int}\,}
\newcommand{\dist}{\mathrm {dist}}
\newcommand{\supp}{\mathrm {supp}\,}
\newcommand{\R}{\mathbb{R}}
\newcommand{\Z}{\mathbb{Z}}
\newcommand{\N}{\mathbb{N}}
\newcommand{\Q}{\mathbb{Q}}
\theoremstyle{plain}
\newtheorem{Trm}{Theorem}
\newtheorem{trma}{Theorem}
\newtheorem{Def}{Definition}
\newtheorem{Cor}{Corollary}
\newtheorem{Lem}{Lemma}
\newtheorem{Rem}{Remark}
\newtheorem{Sta}{Proposition}
\renewcommand{\proofname}{\bf Proof}
\renewcommand{\thetrma}{\Alph{trma}}

\begin{abstract}
Order estimates for the Kolmogorov $n$-widths of $\cap _{\alpha\in A}\nu_\alpha B^{\overline{k}}_{\overline{p}}$ in $l^{\overline{k}} _{\overline{q}}$ are obtained; here $\overline{q}=(q_1, \, \dots, \, q_d)$, $2\le q_j<\infty$, $j=1, \, \dots, \, d$.
\end{abstract}


\section{Introduction}

In this paper we study the problem of estimating the Kolmogorov $n$-widths of an intersection of anisotropic balls $\nu_\alpha B^{\overline{k}}_{\overline{p}}$ in $l^{\overline{k}} _{\overline{q}}$ when each coordinate of the vector $\overline{q}$ lies in $[2, \, \infty)$ (see the notation below). The particular cases, when the vectors $\overline{p}$, $\overline{q}$, $\overline{k}$ were one-dimensional or two-dimensional, were considered in \cite{galeev1, vas_ball_inters, vas_mix_sev}. For one ball $B^{\overline{k}}_{\overline{p}}$, order estimates of the widths are obtained in \cite{vas_anisotr}. So, the present paper continues the previous studies. The estimates we obtained can be applied in the Kolmogorov width problem for
intersection of anisotropic function classes.

Let $N\in \N$, $1\le p\le \infty$, $(x_i)_{i=1}^N\in \mathbb{R}^N$, $\|(x_i)_{i=1}^N\|_{l_p^N} = \left(\sum \limits _{i=1}^N |x_i|^p\right)^{1/p}$ for $p<\infty$, $\|(x_i)_{i=1}^N\|_{l_p^N} = \max _{1\le i\le N}|x_i|$ for $p=\infty$. By $l_p^N$ we denote the space $\mathbb{R}^N$ with the norm $\|\cdot\|_{l_p^N}$, and by $B_p^N$, the unit ball in $l_p^N$.

Let $k_1, \, \dots, \, k_d\in \N$, $1\le p_1, \, \dots, \, p_d\le \infty$, $\overline{k}=(k_1, \, \dots, \, k_d)$, $\overline{p}=(p_1, \, \dots, \, p_d)$. By $l_{\overline{p}}^{\overline{k}}$ we denote the space $\mathbb{R}^{k_1\dots k_d}= \{(x_{j_1,\dots,j_d})_{1\le j_s\le k_s, \, 1\le s\le d}:\; x_{j_1,\dots,j_d}\in \mathbb{R}\}$ with the norm defined by induction on $d$: for $d=1$ it is $\|\cdot\|_{l_{p_1}^{k_1}}$, and for $d\ge 2$,
$$
\|(x_{j_1,\dots,j_d})_{1\le j_s\le k_s, \, 1\le s\le d}\|_{l_{\overline{p}}^{\overline{k}}} = \left\|\bigl(\|(x_{j_1,\dots, \, j_{d-1}, \, j_d})_{1\le j_s\le k_s, \, 1\le s\le d-1}\|_{l_{(p_1,\dots, \, p_{d-1})}^{(k_1,\dots, \, k_{d-1})}}\bigr)_{j_d=1}^{k_d}\right\|_{l_{p_d}^{k_d}}.
$$
The unit ball in $l_{\overline{p}}^{\overline{k}}$ is denoted by $B^{\overline{k}}_{\overline{p}}$. Notice that the spaces $l_{\overline{p}}^{\overline{k}}$ are the particular case of anisotropic spaces $L_{\overline{p}}$ on a product of measure spaces; for details, see \cite{bed_pan, gal_pan}.

\begin{Def}
Let $X$ be a normed space, and let $M\subset X$, $n\in
\Z_+$. The Kolmogorov $n$-width of the set $M$ in the space $X$ is defined as follows:
$$
d_n(M, \, X) = \inf _{L\in {\cal L}_n(X)} \sup _{x\in M} \inf
_{y\in L} \|x-y\|;
$$
here ${\cal L}_n(X)$ is the family of all subspaces in $X$
of dimension at most $n$.
\end{Def}
The properties of the widths are described in \cite{itogi_nt, kniga_pinkusa, teml_book, alimov_tsarkov}.

Let $A$ be a nonempty set, and let, for each $\alpha\in A$, a number $\nu_\alpha>0$ and a vector $\overline{p}_\alpha=(p_{\alpha,1}, \, \dots, \, p_{\alpha,d})$ be given, where $1\le p_{\alpha,j}\le \infty$, $1\le j\le d$. Let $\overline{k}\in \N^d$. Consider the set
\begin{align}
\label{m_def} M = \cap _{\alpha\in A} \nu_\alpha B^{\overline{k}}_{\overline{p}_\alpha}.
\end{align}
In this paper we obtain order estimates for the Kolmogorov $n$-widths of the set $M$ in $l_{\overline{q}}^{\overline{k}}$ for $2\le q_j<\infty$, $n\le \frac{k_1\dots k_d}{2}$; here $\overline{q}=(q_1, \, \dots, \, q_d)$.

The problem of the widths $d_n(B^k_p, \, l^k_q)$ was studied in \cite{k_p_s, stech_poper, pietsch1, stesin, bib_ismag, kashin_oct, bib_kashin, kashin_matr, gluskin1, bib_gluskin, garn_glus}. Order estimates of these quantities are known for all parameters except the case $q=\infty$, $p\in [1, \, 2)$; for $p\ge q$ and for $p=1$, $q=2$ the sharp values are known.

For $d=2$ the problem of estimates for the widths of the ball $B^{\overline{k}}_{\overline{p}}$ in $l^{\overline{k}}_{\overline{q}}$ was investigated in \cite{galeev2, galeev5, izaak1, izaak2, mal_rjut, vas_besov, dir_ull, vas_mix_sev, mal_rjut1}. The historical survey is written in \cite{vas_mix2} (except the recent results). In \cite{vas_mix_sev} the result from \cite{vas_besov} is generalized to arbitrary $\overline{p}=(p_1, \, p_2)$, under the condition $2\le q_1, \, q_2<\infty$. In \cite{mal_rjut1} order estimates of the widths of the ball $B^{\overline{k}}_{\overline{p}}$ in $l^{\overline{k}}_{\overline{q}}$ are obtained for $1\le q_1=q_2\le 2$ and arbitrary $p_1$, $p_2$.

For arbitrary $d\in \N$, $n\le k_1\dots k_d/2$, $\overline{q}\in [2, \, \infty)^d$ the order estimates of $d_n(B^{\overline{k}}_{\overline{p}}, \, l^{\overline{k}}_{\overline{q}})$ are obtained in \cite{vas_anisotr}. In addition, the main result of the paper \cite{mal_rjut1} (Theorem 1) immediately implies the order estimates in the case $1\le q_1=\dots = q_d\le 2$, and the main result of \cite{mal_rjut} yields the order estimates in the case $1\le p_i\le q_i\le 2$ ($1\le i\le \nu$), $1\le q_i\le p_i\le \infty$ ($\nu+1\le i\le d$), where $\nu\in \{0, \, 1, \, \dots, \, d\}$.

Order estimates of $d_n(\cap _{\alpha\in A}\nu_\alpha B^{\overline{k}}_{\overline{p}_\alpha}, \, l^{\overline{k}}_{\overline{q}})$ for $d=1$ are obtained in \cite{galeev1} for $n=k_1/2$ and in \cite{vas_ball_inters} for $n\le k_1/2$ (for $1\le q<\infty$ and arbitrary $p_\alpha\in [1, \, \infty]$, as well as for $q=\infty$ and $p_\alpha\ge 2$, $\alpha\in A$). The case $d=2$, $\overline{q}\in [2, \, \infty)^2$, $n\le k_1k_2/2$ was studied in \cite{vas_mix_sev}.

Also notice that in \cite{hinr_mic} the problem of estimating the Kolmogorov and Gelfand widths of balls in finite-dimensional symmetric spaces was studied (for example, in Lorentz abd Orlicz spaces). In \cite{schatten1, schatten2}  order estimates for the Kolmogorov, Gelfand and approximation numbers of embeddings of Schatten classes of $N\times N$-matrices were obtained.

Before formulating the main result of the present paper we give some notation.

Let $X$, $Y$ be sets, $f_1$, $f_2:\ X\times Y\rightarrow \mathbb{R}_+$.
We write $f_1(x, \, y)\underset{y}{\lesssim} f_2(x, \, y)$ (or $f_2(x, \, y)\underset{y}{\gtrsim} f_1(x, \, y)$) if, for any $y\in Y$, there is $c(y)>0$ such that $f_1(x, \, y)\le
c(y)f_2(x, \, y)$ for all $x\in X$; $f_1(x, \,
y)\underset{y}{\asymp} f_2(x, \, y)$, if $f_1(x, \, y)
\underset{y}{\lesssim} f_2(x, \, y)$ and $f_2(x, \,
y)\underset{y}{\lesssim} f_1(x, \, y)$.

Let $2\le q<\infty$, $1\le p\le \infty$. We set 
\begin{align}
\label{om_pq}
\omega_{p,q} = \begin{cases} 0 & \text{if }p> q, \\ \frac{1/p-1/q}{1/2-1/q} & \text{if }2<p\le q, \\ 1 & \text{if }1\le p\le 2.\end{cases} 
\end{align}

Let $\sigma$ be a permutation of $d$ elements such that \begin{align}
\label{upor} \omega_{p_{\sigma(1)},q_{\sigma(1)}} \le \omega_{p_{\sigma(2)},q_{\sigma(2)}}\le \dots \le \omega_{p_{\sigma(d)},q_{\sigma(d)}}.
\end{align}
The numbers $\mu$, $\nu\in \{0, \, \dots, \, d\}$ are defined as follows:
\begin{align}
\label{mu_nu_def} \{1, \, \dots, \, \mu\} = \{j:\; \omega_{p_{\sigma(j)},q_{\sigma(j)}} = 0\}, \; \{1, \, \dots, \, \nu\} = \{j:\; \omega_{p_{\sigma(j)},q_{\sigma(j)}}<1\}.
\end{align}
By \eqref{om_pq}, \eqref{upor}, the condition $j\in \{1, \, \dots, \, \mu\}$ is equivalent to $p_{\sigma(j)}\ge q_{\sigma(j)}$ for $q_{\sigma(j)}>2$, and to $p_{\sigma(j)}> q_{\sigma(j)}$ for $q_{\sigma(j)}=2$; the condition $j\in \{1, \, \dots, \, \nu\}$ is equivalent to $p_{\sigma(j)}> 2$.

We set $p_j^* = \max\{p_j, \, 2\}$, $1\le j\le d$,
\begin{align}
\label{phi_def}
\begin{array}{c} \Phi(\overline{p},\, \overline{q}, \, \overline{k}, \, n) = \prod _{j=1}^\mu k_{\sigma(j)}^{1/q_{\sigma(j)} - 1/p_{\sigma(j)}} \cdot \min \Bigl\{ 1, \, \min _{\mu+1\le t\le d} \prod _{j=\mu+1}^{t-1} k_{\sigma(j)}^{1/q_{\sigma(j)}-1/p^*_{\sigma(j)}} \\ \times(n^{-1/2}k_{\sigma(1)}^{1/2}\dots k_{\sigma(t-1)}^{1/2} k_{\sigma(t)}^{1/q_{\sigma(t)}}\dots k_{\sigma(d)}^{1/q_{\sigma(d)}})^{\omega _{p_{\sigma(t)},q_{\sigma(t)}}}\Bigl\}.
\end{array}
\end{align}

In \cite{vas_anisotr} it was shown that for $2\le q_j<\infty$, $n\le \frac{k_1\dots k_d}{2}$ the following order estimate holds:
$$
d_n(B_{\overline{p}}^{\overline{k}}, \, l_{\overline{q}}^{\overline{k}}) \underset{\overline{q}}{\asymp} \Phi(\overline{p},\overline{q}, \, \overline{k}, \, n).
$$

Given $\overline{p} = (p_1, \, \dots, \, p_d)$, $\lambda\in \R$, we denote by $\frac{\lambda}{\overline{p}}$ the vector with coordinates $\left(\frac{\lambda}{p_1}, \, \dots,\, \frac{\lambda}{p_d}\right)$.

For $p\in [1, \, \infty]^d$, $2<q<\infty$ we set
\begin{align}
\label{om_prime} \omega'_{p,q} = \frac{1/p-1/q}{1/2-1/q}.
\end{align}

Given $E\subset \R^d$, we denote by ${\rm conv}\, E$ and ${\rm aff}\, E$, respectively, the convex and the affine hull of the set $E$.

\begin{Def}
\label{zm_def}
Let $Z\subset \R^d$, $1\le m\le d+1$. We say that $Z\in {\cal Z}_m$ if there is a set $I\subset \{1, \, \dots, \, d\}$ such that one of the following conditions holds:
\begin{enumerate}
\item $\# I = m-1$,
$$
Z=\{(x_1, \, \dots, \, x_d)\in \R^d:\; x_i=1/q_i, \; i\in I\};
$$
\item $\# I = m-1$,
$$
Z=\{(x_1, \, \dots, \, x_d)\in \R^d:\; x_i=1/2, \; i\in I\};
$$
\item $m\ge 2$, $\# I=m$,
$$
Z=\{(x_1, \, \dots, \, x_d)\in \R^d:\; \omega'_{1/x_i,q_i} = \omega'_{1/x_l,q_l}\in (0, \, 1)\; \forall i, \, l\in I:\; q_i>2, \, q_l>2;$$$$ x_i=1/2 \; \; \forall i\in I:\; q_i=2\},
$$
and $\# \{i\in I:\; q_i>2\}\ge 2$.
\end{enumerate}
\end{Def}

The plane ${\rm aff}\, Z$ has codimension $m-1$. It suffices to check this fact in case 3. Let
$$
I'=\{i\in I:\; q_i>2\} = \{i_1, \, \dots, \, i_k\}, \; k\ge 2,
$$
$$
I''=\{i\in I:\; q_i=2\} = \{j_1, \, \dots, \, j_{m-k}\}.
$$
Then ${\rm aff}\, Z$ is given by the equations 
$$
\frac{x_{i_1}-1/q_{i_1}}{1/2-1/q_{i_1}} = \dots = \frac{x_{i_k}-1/q_{i_k}}{1/2-1/q_{i_k}}, \; x_{j_1} = \dots = x_{j_{m-k}} =1/2;
$$
it is the system of $m-1$ linearly independent equations.

Let $L \subset \R^d$, $Z\subset \R^d$ be affine planes intersecting at point $\hat x$. We say that they are complementary if $(L-\hat x)\oplus (Z-\hat x) = \R^d$ (i.e., the linear subspaces $L-\hat x$ and $Z-\hat x$ intersect only at zero and their sum is $\R^d$).

\begin{Def}
\label{nm_def}
Let $1\le m\le d+1$, $\overline{\alpha} = (\alpha_1, \, \dots, \, \alpha_m) \in A^m$. We say that $\overline{\alpha}\in {\cal N}_m$ if there are numbers $\lambda_j>0$, $j=1, \, \dots, \, m$, $\sum \limits _{j=1}^m \lambda_j=1$, and a set $Z\in {\cal Z}_m$ such that $\sum \limits _{j=1}^m \frac{\lambda_j}{\overline{p}_{\alpha_j}} \in Z$, and the planes ${\rm aff}\, \{1/\overline{p}_{\alpha_j}\}_{j=1}^m$ and ${\rm aff}\, Z$ are complementary.
\end{Def}

Notice that if $\overline{\alpha}\in {\cal N}_m$, then the points $\{1/\overline{p}_{\alpha_j}\}_{j=1}^m$ are affinely independent; in particular, the points $\alpha_1, \, \dots, \, \alpha_m$ are different. Indeed, from the complementarity of the planes and the equality $\dim {\rm aff}\, Z=d - m+1$ for $Z\in {\cal Z}_m$ it follows that ${\rm aff}\, \{1/\overline{p}_{\alpha_j}\}_{j=1}^m$ is $m-1$-dimensional.

Let $\overline{\alpha}\in {\cal N}_m$ and let $Z \in {\cal Z}_m$ be such as in Definition \ref{nm_def}. We denote the corresponding numbers $\lambda_j$ by $\lambda_j(\overline{\alpha}, \, Z)$. We also define the vector $$\overline{\theta}(\overline{\alpha}, \, Z) = (\theta_1(\overline{\alpha}, \, Z), \, \dots, \, \theta_d(\overline{\alpha}, \, Z))$$ by the equation
\begin{align}
\label{theta_a_def}
\frac{1}{\overline{\theta}(\overline{\alpha}, \, Z)} = \sum \limits _{j=1}^m \frac{\lambda_j(\overline{\alpha}, \, Z)}{\overline{p}_{\alpha_j}}.
\end{align}

\begin{Trm}
\label{main} Let the set $M$ be defined by formula \eqref{m_def}, and let $2\le q_i<\infty$, $1\le i\le d$, $n\in \Z_+$, $n\le \frac{k_1\dots k_d}{2}$. Then
\begin{align}
\label{main_eq}
d_n(M, \, l_{\overline{q}}^{\overline{k}}) \underset{\overline{q},d}{\asymp} \min _{1\le m\le d+1} \inf _{\overline{\alpha} \in {\cal N}_m, \, Z \in {\cal Z}_m} \nu_{\alpha_1}^{\lambda_1(\overline{\alpha}, \, Z)} \dots \nu_{\alpha_m}^{\lambda_m(\overline{\alpha}, \, Z)} \Phi(\overline{\theta}(\overline{\alpha}, \, Z), \, \overline{q}, \, \overline{k}, \, n)
\end{align}
(the infimum is taken over $\overline{\alpha} \in {\cal N}_m$, $Z\in {\cal Z}_m$ for which the corresponding numbers $\lambda_j(\overline{\alpha}, \, Z)$ are well-defined).
\end{Trm}

The paper is organized as follows. In \S 2 we formulate the well-known results which will be further applied; also we prove the upper estimate for the widths. In \S 3 we prove the lower estimate for the finite set $A$, under assumption that $q_i>2$, $i=1, \, \dots, \, d$, and the points $\{(\nu_\alpha, \, \overline{p}_\alpha)\}_{\alpha \in A}$ are in the general position. In \S 4 we prove the lower estimate for arbitrary $A$, $\overline{q}\in [2, \, \infty)^d$ and $\{(\nu_\alpha, \, \overline{p}_\alpha)\}_{\alpha \in A}$.

\section{Auxiliary assertions and the upper estimate}

Let $$G = \{(\sigma_1, \, \dots, \, \sigma_d, \, \varepsilon_1, \, \dots, \, \varepsilon_d): \sigma _j \in S_{k_j}, \, \varepsilon_j \in \{-1, \, 1\}^{k_j}, \, 1\le j\le d\},$$
where $S_{k_j}$ is the group of permutations of $k_j$ elements.
For $g = (\sigma_1, \, \dots, \, \sigma_d, \, \varepsilon_1, \, \dots, \, \varepsilon_d) \in G$, $\varepsilon_j = (\varepsilon_{j,1}, \, \dots, \, \varepsilon_{j,k_j})$, $x = (x_{i_1,\dots,i_d})_{1\le i_1\le k_1,\dots, 1\le i_d\le k_d}$ we set $$g(x) = (\varepsilon_{1,i_1}\dots \varepsilon _{d,i_d}x_{\sigma_1(i_1),\dots,\sigma_d(i_d)})_{1\le i_1\le k_1,\dots, 1\le i_d\le k_d}.$$

Let $s_j \in \{1, \, \dots, \, k_j\}$, $1\le j\le d$, $\overline{s}=(s_1, \, \dots, \, s_d)$. We set $\hat x(\overline{s}) = (\hat x_{i_1,\dots,i_d})_{1\le i_1\le k_1,\dots, 1\le i_d\le k_d}$, where
$$
\hat x_{i_1,\dots,i_d} = \begin{cases} 1 & \text{for }1\le i_1\le s_1, \, \dots, \, 1\le i_d\le s_d, \\ 0 & \text{otherwise},\end{cases}
$$
\begin{align}
\label{vks_defin}
V^{\overline{k}}_{\overline{s}} = {\rm conv}\, \{g(\hat x(\overline{s})):\; g\in G\}.
\end{align}

\begin{trma}
\label{vks_low_est} {\rm (see \cite[Corollary 1]{vas_anisotr}).} Let $2\le q_j<\infty$, $k_j\in \N$, $s_j\in \{1, \, \dots, \, k_j\}$, $1\le j\le d$, $n\in \Z_+$, $n\le \frac{k_1\dots k_d}{2}$. Then 
\begin{align} \label{vk1kddn}
d_n(V^{\overline{k}}_{\overline{s}}, \, l_{\overline{q}}^{\overline{k}}) \gtrsim \begin{cases} s_1^{1/q_1} \dots s_d^{1/q_d}, & n \le k_1^{\frac{2}{q_1}}\dots k_d^{\frac{2}{q_d}} s_1^{1-\frac{2}{q_1}}\dots s_d^{1-\frac{2}{q_d}}, \\ n^{-1/2}k_1^{1/q_1}\dots k_d^{1/q_d}s_1^{1/2}\dots s_d^{1/2}, & n \ge k_1^{\frac{2}{q_1}}\dots k_d^{\frac{2}{q_d}} s_1^{1-\frac{2}{q_1}}\dots s_d^{1-\frac{2}{q_d}}. \end{cases}
\end{align}
\end{trma}

\begin{trma}
\label{one_ball_est} {\rm (see \cite{vas_anisotr}).} Let $2\le q_j<\infty$, $j=1, \, \dots, \, d$, $n\in \Z_+$, $n\le \frac{k_1\dots k_d}{2}$, and let the function $\Phi$ be defined by formula \eqref{phi_def}. Then
$$
d_n(B^{\overline{k}}_{\overline{p}}, \, l^{\overline{k}} _{\overline{q}}) \underset{\overline{q}}{\asymp} \Phi(\overline{p}, \, \overline{q}, \, \overline{k}, \, n).
$$
\end{trma}

Let $\sigma$ be a permutation satisfying \eqref{upor}, and let the numbers $\mu$ and $\nu$ be defined by \eqref{mu_nu_def}. From \eqref{phi_def} it follows that
\begin{align}
\label{phi_case10}
\Phi(\overline{p},\, \overline{q}, \, \overline{k}, \, n) = \prod _{i=1}^\mu k_{\sigma(i)} ^{1/q_{\sigma(i)} - 1/p_{\sigma(i)}} \quad \text{for } n \le \prod _{i=1}^\mu k_{\sigma(i)} \prod _{i=\mu+1}^d k_{\sigma(i)}^{2/q_{\sigma(i)}};
\end{align}
if $\mu+1\le t\le \nu$ and $\prod _{i=1}^{t-1} k_{\sigma(i)} \prod _{i=t}^d k_{\sigma(i)}^{2/q_{\sigma(i)}}\le n \le \prod _{i=1}^t k_{\sigma(i)} \prod _{i=t+1}^d k_{\sigma(i)}^{2/q_{\sigma(i)}}$, then
\begin{align}
\label{phi_case20}
\Phi(\overline{p},\, \overline{q}, \, \overline{k}, \, n) = \prod _{i=1}^{t-1} k_{\sigma(i)} ^{1/q_{\sigma(i)} - 1/p_{\sigma(i)}}\Bigl(n^{-1/2}\prod_{i=1}^{t-1} k_{\sigma(i)} ^{1/2}\prod _{i=t}^d k_{\sigma(i)}^{1/q_{\sigma(i)}}\Bigr)^{\frac{1/p_{\sigma(t)}- 1/q_{\sigma(t)}} {1/2 - 1/q_{\sigma(t)}}};
\end{align}
if $n \ge \prod _{i=1}^\nu k_{\sigma(i)} \prod _{i=\nu+1}^d k_{\sigma(i)} ^{2/q_{\sigma(i)}}$, then
\begin{align}
\label{phi_case30}
\Phi(\overline{p},\, \overline{q}, \, \overline{k}, \, n) = \prod _{i=1}^\nu k_{\sigma(i)} ^{1/q_{\sigma(i)} - 1/p_{\sigma(i)}} \cdot n^{-1/2}\prod_{i=1}^\nu k_{\sigma(i)} ^{1/2}\prod _{i=\nu+1}^d k_{\sigma(i)}^{1/q_{\sigma(i)}}.
\end{align}

\begin{Rem}
\label{t1t1_pr}
If $\mu+1\le t_1\le t \le t_1'\le \nu$, 
\begin{align}
\label{p_sig_t_eq}
\omega _{p_{\sigma(t_1)}, q_{\sigma(t_1)}} = \omega _{p_{\sigma(t_1+1)}, q_{\sigma(t_1+1)}} = \dots = \omega _{p_{\sigma(t'_1)}, q_{\sigma(t'_1)}}, 
\end{align}
$\prod _{i=1}^{t-1} k_{\sigma(i)} \prod _{i=t}^d k_{\sigma(i)}^{2/q_{\sigma(i)}}\le n \le \prod _{i=1}^t k_{\sigma(i)} \prod _{i=t+1}^d k_{\sigma(i)}^{2/q_{\sigma(i)}}$, then
$$
\Phi(\overline{p},\, \overline{q}, \, \overline{k}, \, n) = \prod _{i=1}^{t_1-1} k_{\sigma(i)} ^{1/q_{\sigma(i)} - 1/p_{\sigma(i)}}\Bigl(n^{-1/2}\prod_{i=1}^{t_1-1} k_{\sigma(i)} ^{1/2}\prod _{i=t_1}^d k_{\sigma(i)}^{1/q_{\sigma(i)}}\Bigr)^{\frac{1/p_{\sigma(t_1)}- 1/q_{\sigma(t_1)}} {1/2 - 1/q_{\sigma(t_1)}}}.
$$
\end{Rem}

Indeed,
$$
\Phi(\overline{p},\, \overline{q}, \, \overline{k}, \, n) \stackrel{\eqref{phi_case20}}{=} \prod _{i=1}^{t-1} k_{\sigma(i)} ^{1/q_{\sigma(i)} - 1/p_{\sigma(i)}}\Bigl(n^{-1/2}\prod_{i=1}^{t-1} k_{\sigma(i)} ^{1/2}\prod _{i=t}^d k_{\sigma(i)}^{1/q_{\sigma(i)}}\Bigr)^{\frac{1/p_{\sigma(t)}- 1/q_{\sigma(t)}} {1/2 - 1/q_{\sigma(t)}}} = 
$$
$$
= \prod _{i=1}^{t_1-1} k_{\sigma(i)} ^{1/q_{\sigma(i)} - 1/p_{\sigma(i)}} \prod _{i=t_1}^{t-1} k_{\sigma(i)} ^{(1/q_{\sigma(i)}-1/2) \omega _{p_{\sigma(i)}, q_{\sigma(i)}}} \Bigl(n^{-1/2}\prod_{i=1}^{t-1} k_{\sigma(i)} ^{1/2}\prod _{i=t}^d k_{\sigma(i)}^{1/q_{\sigma(i)}}\Bigr)^{\omega _{p_{\sigma(t)}, q_{\sigma(t)}}} \stackrel{\eqref{p_sig_t_eq}}{=} 
$$
$$
= \prod _{i=1}^{t_1-1} k_{\sigma(i)} ^{1/q_{\sigma(i)} - 1/p_{\sigma(i)}} \Bigl( \prod _{i=t_1}^{t-1} k_{\sigma(i)}^{1/q_{\sigma(i)}-1/2} \cdot n^{-1/2}\prod_{i=1}^{t-1} k_{\sigma(i)} ^{1/2}\prod _{i=t}^d k_{\sigma(i)}^{1/q_{\sigma(i)}}\Bigr)^{\omega _{p_{\sigma(t_1)}, q_{\sigma(t_1)}}} = 
$$
$$
= \prod _{i=1}^{t_1-1} k_{\sigma(i)} ^{1/q_{\sigma(i)} - 1/p_{\sigma(i)}}\Bigl(n^{-1/2}\prod_{i=1}^{t_1-1} k_{\sigma(i)} ^{1/2}\prod _{i=t_1}^d k_{\sigma(i)}^{1/q_{\sigma(i)}}\Bigr)^{\frac{1/p_{\sigma(t_1)}- 1/q_{\sigma(t_1)}} {1/2 - 1/q_{\sigma(t_1)}}}.
$$

The function $\Phi$ for $q_j>2$ ($j=1, \, \dots, \, d$) can be written as follows:
\begin{align}
\label{phi_case1} \Phi(\overline{p},\, \overline{q}, \, \overline{k}, \, n) =
 \prod _{p_j\ge q_j} k_j^{1/q_j - 1/p_j} \quad \text{for } n\le \prod _{p_j\ge q_j} k_j \prod _{p_j< q_j} k_j^{2/q_j};
\end{align}
\begin{align}
\label{phi_case2}
\begin{array}{c}
\Phi(\overline{p},\, \overline{q}, \, \overline{k}, \, n) = \prod_{\omega_{p_j,q_j}<\omega_{p_t,q_t}} k_j^{1/q_j - 1/p_j} (n^{-1/2} \prod _{\omega_{p_j,q_j}< \omega_{p_t,q_t}} k_j^{1/2}\prod _{\omega_{p_j,q_j}\ge \omega_{p_t,q_t}} k_j^{1/q_j})^{\omega _{p_t,q_t}} \\
\text{ for } \omega_{p_t,q_t}\in (0, \, 1) \text{ and} \\ \prod _{\omega_{p_j,q_j}< \omega_{p_t,q_t}} k_j \prod _{\omega_{p_j,q_j}\ge \omega_{p_t,q_t}} k_j^{2/q_j} \le n \le \prod _{\omega_{p_j,q_j}\le \omega_{p_t,q_t}} k_j \prod _{\omega_{p_j,q_j}> \omega_{p_t,q_t}} k_j^{2/q_j},
\end{array}
\end{align}
\begin{align}
\label{phi_case3}
\Phi(\overline{p},\, \overline{q}, \, \overline{k}, \, n) = \prod _{p_j>2} k_j^{1/q_j-1/p_j} n^{-1/2} \prod _{p_j>2} k_j^{1/2} \prod _{p_j\le 2} k_j^{1/q_j} \quad \text{for } n \ge \prod _{p_j>2} k_j \prod _{p_j\le 2} k_j^{2/q_j}.
\end{align}

The following lemmas are proved in \cite{vas_plq}.
\begin{Lem}
\label{conv_hull} {\rm (see \cite{vas_plq}).}
Let $\xi_1, \, \dots, \, \xi_m \in \R^{m-1}$ be affinely independent points, let $\eta \in \R^{m-1}$, and let the points $\{\xi_1, \, \dots, \, \xi_{j-1}, \, \xi_{j+1}, \, \dots, \, \xi_m, \, \eta\}$ be affinely independent for each $j\in \{1, \, \dots, \, m\}$. We set $\Delta = {\rm conv}\, \{\xi_1, \, \dots, \, \xi_m\}$. Let $a\in {\rm int}\, \Delta$. Then there is $i\in \{1, \, \dots, \, m\}$ such that $a\in {\rm conv}\, \{\xi_1, \, \dots, \, \xi_{i-1}, \, \xi_{i+1}, \, \dots, \, \xi_m, \, \eta\}$.
\end{Lem}

\begin{Lem}
\label{inter_incl} {\rm (see \cite{vas_plq}).} Let $\overline{k}=(k_1, \, \dots, \, k_d)\in \N^d$, $m\in \N$, $\nu_j>0$, $\overline{p}^j\in [1, \, \infty]^d$, $\lambda_j\ge 0$, $1\le j\le m$, $\sum \limits _{j=1}^m \lambda_j=1$, and let the vector $\overline{p}$ be defined by the equation
$\frac{1}{\overline{p}} = \sum \limits _{j=1}^m \frac{\lambda_j}{\overline{p}^j}$.
Then
$$
\cap _{j=1}^m \nu_j B^{\overline{k}}_{\overline{p}^j} \subset \nu_1^{\lambda_1} \dots \nu_m^{\lambda_m}B^{\overline{k}}_{\overline{p}}.
$$
\end{Lem}

From Lemma \ref{inter_incl} and Theorem \ref{one_ball_est} we obtain the upper estimate in Theorem \ref{main}:
\begin{align}
\label{up_est_main} d_n(M, \, l_{\overline{q}}^{\overline{k}}) \underset{\overline{q}}{\lesssim} \min _{1\le m\le d+1} \inf _{\overline{\alpha} \in {\cal N}_m,\, Z \in {\cal Z}_m} \nu_{\alpha_1}^{\lambda_1(\overline{\alpha},\, Z)} \dots \nu_{\alpha_m}^{\lambda_m(\overline{\alpha},\, Z)} \Phi(\overline{\theta}(\overline{\alpha},\, Z), \, \overline{q}, \, \overline{k}, \, n).
\end{align}

\section{Proof of the lower estimate in the case of finite set $A$, $q_j>2$ and general position}

Throughout this section we assume that the set $A$ is finite and $q_j>2$, $1\le j\le d$.

\begin{Def}
\label{zm_def1} Let $Z \subset \R^d$, $1\le m\le d+1$. We say that $Z\in \tilde {\cal Z}_m$ if there are a set $I\subset \{1, \, \dots, \, d\}$ and a number $L\in \Z_+$ with the following properties:
\begin{enumerate}
\item $I=I_2\sqcup I_q \sqcup I_{\omega_1}\sqcup \dots \sqcup I_{\omega_L}$, $I_{\omega_t} = \{i_{1,t}, \, \dots, \, i_{l_t,t}\}$,
\begin{align}
\label{card_i_zm} \# I_2 + \# I_q +\sum \limits _{t=1}^L (\# I_{\omega_t}-1) = m-1, \quad \# I_{\omega_t}\ge 2, \; t=1, \, \dots, \, L;
\end{align}
\item the set $Z$ is defined by the conditions
\begin{align}
\label{pl_tr_simpl00}
\begin{array}{c}
x_i=1/q_i, \; i\in I_q; \quad x_i = 1/2, \; i\in I_2;\\ \omega' _{1/x_{i_{1,t}}, q_{i_{1,t}}} = \dots = \omega' _{1/x_{i_{l_t,t}}, q_{i_{l_t,t}}} \in (0, \, 1), \; t=1, \, \dots, \, L.
\end{array}
\end{align}
\end{enumerate}
\end{Def}
From \eqref{card_i_zm}, \eqref{pl_tr_simpl00} it follows that, for $Z\in \tilde {\cal Z}_m$, the plane ${\rm aff}\, Z$ has codimension $m-1$.

\begin{Def}
\label{nm_def1}
Let $1\le m\le d+1$, $\overline{\alpha} = (\alpha_1, \, \dots, \, \alpha_m) \in A^m$. We say that $\overline{\alpha}\in \tilde{\cal N}_m$ if there are numbers $\lambda_j>0$, $j=1, \, \dots, \, m$, $\sum \limits _{j=1}^m \lambda_j=1$, and a set $Z\in \tilde{\cal Z}_m$ such that $\sum \limits _{j=1}^m \frac{\lambda_j}{\overline{p}_{\alpha_j}} \in Z$, and, in addition, the planes ${\rm aff}\, \{1/\overline{p}_{\alpha_j}\} _{j=1}^m$ and ${\rm aff}\, Z$ are complementary.
\end{Def}

For $\overline{\alpha} \in \tilde {\cal N}_m$, $Z\in \tilde {\cal Z}_m$, we denote by $\lambda_j(\overline{\alpha}, \, Z)$ the corresponding $\lambda_j$ from Definition \ref{nm_def1} (if such $\lambda_j$ are well-defined for the set $Z$). We also define the vector $\overline{\theta}(\overline{\alpha}, \, Z)$ by the equation $\frac{1}{\overline{\theta}(\overline{\alpha}, \, Z)} = \sum \limits _{j=1}^m \frac{\lambda_j(\overline{\alpha}, \, Z)}{\overline{p}_{\alpha_j}}$.

We write $\overline{r} \in \Theta$ if $\overline{r} = \overline{\theta}(\overline{\alpha}, \, Z)$ for some $m\in \{1, \, \dots, \, d+1\}$, $\overline{\alpha} \in \tilde {\cal N}_m$, $Z \in \tilde {\cal Z}_m$.

We will prove that in the case of general position (see Definition \ref{gen_pos} below) the following estimate holds:
\begin{align}
\label{dn_low_est_mod}
d_n(M, \, l_{\overline{q}}^{\overline{k}}) \underset{d}{\gtrsim} \min _{1\le m\le d+1} \min _{\overline{\alpha} \in \tilde{\cal N}_m, \, Z \in \tilde{\cal Z}_m} \nu_{\alpha_1}^{\lambda_1(\overline{\alpha}, \, Z)} \dots \nu_{\alpha_m}^{\lambda_m(\overline{\alpha}, \, Z)} \Phi(\overline{\theta}(\overline{\alpha}, \, Z), \, \overline{q}, \, \overline{k}, \, n),
\end{align}
and the minimum in the right-hand side can be attained only if one of the following conditions holds: a) $I=I_2$, b) $I=I_q$, c) $I=I_{\omega_1}$. This together with \eqref{up_est_main} implies \eqref{main_eq}.

In what follows, we write $\log x = \log_2 x$. 

Let $2\le m\le d+1$, $I\subset \{1, \, \dots, \, m\}$, $\# I = m$. We say that ${\cal A}=(A_{j,i})_{1\le j\le m, \, i\in I} \in {\cal M}_{m,I}$ if, for each $j\in \{1, \, \dots, \, m\}$, there are an index $i_*(j)\in I$ and a partition $\{1, \, \dots, \, d\} \backslash i_*(j) = T_{j,1}\sqcup T_{j,2}$, such that
\begin{align}
\label{a_mm_i} A_{j,i} = \begin{cases} \Bigl(n^{1/2} \prod _{i\in T_{j,1}}k_i^{-1/2} \cdot k_{i_*(j)}^{-1/q_{i_*(j)}}\prod _{i\in T_{j,2}}k_i^{-1/q_i}\Bigr) ^{\frac{1}{1/2-1/q_{i_*(j)}}}, & i=i_*(j), \\ k_i, & i\in T_{j,1}\cap I, \\ 1, & i\in T_{j,2}\cap I.\end{cases}
\end{align}
We say that ${\cal B}=(B_{j,i})_{2\le j\le m,\, i\in I}\in \hat {\cal M}_{m,I}$ if there is a matrix ${\cal A}=(A_{j,i})_{1\le j\le m, \, i\in I}\in {\cal M}_{m,I}$ such that
\begin{align}
\label{b_mm_i} B_{j,i} = \log A_{j,i} - \log A_{1,i}, \quad 2\le j\le m, \; i\in I
\end{align}
and the range of ${\cal B}$ is $m-1$.

Let $I=\{i_1, \, \dots, \, i_l\}\subset \{1, \, \dots, \, d\}$ be a nonempty subset, $i_1<\dots<i_l$. Given $x=(x_1, \, \dots, \, x_d)\in \R^d$, we write $x_I=(x_{i_1}, \, \dots, \, x_{i_l})\in \R^l$.

\begin{Def}
\label{gen_pos} We say that the pairs $\{(\nu_\alpha, \, \overline{p}_\alpha)\}_{\alpha \in A}$ are in general position if the following conditions hold:
\begin{enumerate}
\item For any number $m\in \{2, \, \dots, \, d+1\}$, any set $I\subset \{1, \, \dots, \, d\}$ such that $\# I = m-1$, and for any different $\alpha_1, \, \dots, \, \alpha_m\in A$ the points $\{(1/\overline{p}_{\alpha_j})_I\}_{j=1}^m$ are affinely independent.

\item Let $L\in \Z_+$, let $I=I_2\sqcup I_q\sqcup I_{\omega_1}\sqcup \dots \sqcup I_{\omega_L}$ satisfy \eqref{card_i_zm}, $i_t\in I_{\omega_t}$ $(1\le t\le L)$, and let $Z\subset \R^d$ be an affine plane defined by the equations
\begin{align}
\label{pl_tr_simpl}
\begin{array}{c}
x_i=1/2, \quad i\in I_2; \quad x_i=1/q_i, \quad i\in I_q; \\ \omega'_{1/x_i,q_i} = \omega'_{1/x_{i_t},q_{i_t}}, \quad \forall i\in I_{\omega_t}, \; 1\le t\le L.
\end{array}
\end{align}
Then for any different $\beta_1, \, \dots, \, \beta_m\in A$ the affine hull of $\{1/\overline{p}_{\beta_j}\}_{j=1}^m$ and $Z$ are complementary. If $k<m$, then for any $\beta_1, \, \dots, \, \beta_k\in A$ the sets ${\rm conv}\, \{1/\overline{p}_{\beta_j}\}_{j=1}^k$ and $Z$ do not intersect.

\item For each $m\in \{2, \, \dots, \, d+1\}$, each matrix ${\cal B}=(B_{j,i})_{2\le j\le m, \, i\in I}\in \hat {\cal M}_{m,I}$ and any different $\beta_1, \, \dots, \, \beta_m\in A$ the matrix $\Bigl(\sum \limits _{i\in I}(1/p_{\beta_k,i}-1/p_{\beta_1,i})B_{j,i}\Bigr)_{2\le k\le m, \, 2\le j\le m}$ is nondegenerated.

\item In the right-hand side of \eqref{dn_low_est_mod} the minimum is attained only at one element $(\overline{\alpha}, \, Z)$.
\end{enumerate}
\end{Def}

In what follows in this section we suppose that $\{(\nu_\alpha, \, \overline{p}_\alpha)\}_{\alpha \in A}$ are in  general position.

Let 
\begin{align}
\label{psi_def} \Psi = \min _{1\le m\le d+1} \min _{\overline{\alpha} \in \tilde{\cal N}_m, \, Z \in \tilde{\cal Z}_m} \nu_{\alpha_1}^{\lambda_1(\overline{\alpha}, \, Z)} \dots \nu_{\alpha_m}^{\lambda_m(\overline{\alpha}, \, Z)} \Phi(\overline{\theta}(\overline{\alpha}, \, Z), \, \overline{q}, \, \overline{k}, \, n).
\end{align}
We prove that if 
\begin{align}
\label{psi_eq_nu_phi}
\Psi = \nu_{\alpha_1}^{\lambda_1(\overline{\alpha}, \, Z)} \dots \nu_{\alpha_m}^{\lambda_m(\overline{\alpha}, \, Z)} \Phi(\overline{\theta}(\overline{\alpha}, \, Z), \, \overline{q}, \, \overline{k}, \, n)
\end{align}
for some $\overline{\alpha} \in \tilde{\cal N}_m$, $Z \in \tilde {\cal Z}_m$, then
\begin{align}
\label{dn_est_psi} d_n(M, \, l_{\overline{q}}^{\overline{k}}) \underset{d}{\gtrsim} \nu_{\alpha_1}^{\lambda_1(\overline{\alpha}, \, Z)} \dots \nu_{\alpha_m}^{\lambda_m(\overline{\alpha}, \, Z)} \Phi(\overline{\theta}(\overline{\alpha}, \, Z), \, \overline{q}, \, \overline{k}, \, n),
\end{align}
and, in addition, only the case $\alpha \in {\cal N}_m$, $Z \in {\cal Z}_m$ is possible.

\subsection{The case $m=1$}

Suppose that \eqref{psi_eq_nu_phi} holds with $m=1$; i.e.,
\begin{align}
\label{psi_eq_nu_phi1} \Psi = \nu_{\alpha} \Phi(\overline{p}_{\alpha}, \, \overline{q}, \, \overline{k}, \, n)
\end{align}
for some $\alpha\in A$. Then $I=\varnothing$. We show that
\begin{align}
\label{dn_est_psi1} d_n(M, \, l_{\overline{q}}^{\overline{k}}) \gtrsim \nu_{\alpha} \Phi(\overline{p}_{\alpha}, \, \overline{q}, \, \overline{k}, \, n).
\end{align}
To this end, we consider the set
\begin{align}
\label{w_def}
W = \nu_\alpha u_1^{-1/p_{\alpha_1}}\dots u_d^{-1/p_{\alpha_d}}V^{\overline{k}}_{\overline{u}}
\end{align}
(see \eqref{vks_defin}), where $\overline{u}=(u_1, \, \dots, \, u_d)$, $u_i=\lceil s_i \rceil$ or $u_i = \lfloor s_i\rfloor$, and the numbers $s_i$, such that $1\le s_i\le k_i$ and $W\subset 2M$, will be defined below. After that we use Theorem \ref{vks_low_est}.

Let the permutation $\sigma \in S_d$ and the numbers $\mu\in \{0, \, \dots, \, d\}$, $\nu\in \{0, \, \dots, \, d\}$ be defined by conditions \eqref{upor} and \eqref{mu_nu_def}. By the generality of position (see assertion 2 of Definition \ref{gen_pos}), $p_{\alpha, \sigma(i)}>q_{\sigma(i)}$ for $1\le i\le \mu$, $2<p_{\alpha, \sigma(i)}<q_{\sigma(i)}$ for $\mu+1\le i\le \nu$, $p_{\alpha, \sigma(i)}<2$ for $\nu+1\le i\le d$; in addition, $\omega_{p_{\alpha, \sigma(\mu+1)},q_{\sigma(\mu+1)}}< \dots < \omega_{p_{\alpha, \sigma(\nu)},q_{\sigma(\nu)}}$.

The numbers $s_i$ are defined as in \cite[\S 2, cases after Remark 3]{vas_anisotr}:
\begin{enumerate}
\item for $n\le k_{\sigma(1)} \dots k_{\sigma(\mu)}k_{\sigma(\mu+1)}^{2/q_{\sigma(\mu+1)}} \dots k_{\sigma(d)}^{2/q_{\sigma(d)}}$
\begin{align}
\label{s_sigma_i1}
s_{\sigma(i)} = \begin{cases} k_{\sigma(i)}, & i\le \mu, \\ 1, & i>\mu; \end{cases}
\end{align}
\item for $t\in \{\mu+1, \, \dots, \, \nu\}$, $$k_{\sigma(1)} \dots k_{\sigma(t-1)}k_{\sigma(t)}^{2/q_{\sigma(t)}} \dots k_{\sigma(d)}^{2/q_{\sigma(d)}}< n \le k_{\sigma(1)} \dots k_{\sigma(t)}k_{\sigma(t+1)}^{2/q_{\sigma(t+1)}} \dots k_{\sigma(d)}^{2/q_{\sigma(d)}}$$
\begin{align}
\label{s_sigma_i2}
s_{\sigma(i)} = \begin{cases} k_{\sigma(i)}, & i\le t-1, \\ (n^{1/2}k_{\sigma(1)}^{-1/2}\dots k_{\sigma(t-1)}^{-1/2} k_{\sigma(t)}^{-1/q_{\sigma(t)}}\dots k_{\sigma(d)} ^{-1/q_{\sigma(d)}})^{\frac{1}{1/2-1/q_{\sigma(t)}}}, & i=t, \\ 1, & i\ge t+1; \end{cases}
\end{align}
\item for $n> k_{\sigma(1)} \dots k_{\sigma(\nu)} k_{\sigma(\nu+1)} ^{2/q_{\sigma(\nu+1)}} \dots k_{\sigma(d)} ^{2/q_{\sigma(d)}}$
\begin{align}
\label{s_sigma_i3}
s_{\sigma(i)} = \begin{cases} k_{\sigma(i)}, & i\le \nu, \\ 1,  & i>\nu. \end{cases}
\end{align}
\end{enumerate}

We can see from the definition that $1\le s_i\le k_i$, the numbers $s_i$ are integer, except, possibly, the number $s_{\sigma(t)}$ in case 2; then we put $u_{\sigma(t)} = \lceil s_{\sigma(t)}\rceil$. For other indices $i$ we take $u_i=s_i$ and define the set $W$ by formula \eqref{w_def}. In \cite[\S 2]{vas_anisotr} it was proved that
$$
d_n(W, \, l_{\overline{q}}^{\overline{k}}) \gtrsim \nu_{\alpha} \Phi(\overline{p}_{\alpha}, \, \overline{q}, \, \overline{k}, \, n).
$$
Hence, in order to prove \eqref{dn_est_psi1}, it remains to check the inclusion $W\subset 2M$; to this end, it suffices to show that
\begin{align}
\label{incl_ineq} \nu_{\alpha} s_1^{1/p_{\beta,1}-1/p_{\alpha,1}} \dots s_d^{1/p_{\beta,d}-1/p_{\alpha,d}} \le \nu_{\beta}, \quad \beta \in A.
\end{align}

For $\beta=\alpha$, this inequality is trivial.

Let $\beta \ne \alpha$. We set
$$
D = \{(1/p_1, \, \dots, \, 1/p_d)\in [0, \, 1]^d:\; p_{\sigma(i)}>q_{\sigma(i)}, \; 1\le i\le \mu; $$$$2< p_{\sigma(i)}<q_{\sigma(i)}, \; \mu+1\le i\le \nu;\; p_{\sigma(i)}<2, \; \nu+1\le i\le d;$$$$ \omega_{p_{\sigma(\mu+1)},q_{\sigma(\mu+1)}} < \dots < \omega_{p_{\sigma(\nu)},q_{\sigma(\nu)}}\}.
$$
The set $D$ is open in $[0, \, 1]^d$ and contains the point $1/\overline{p}_{\alpha}$. Let
$$
\lambda = \sup \Bigl\{\tau \in [0, \, 1]:\; \frac{1-\tau}{\overline{p}_{\alpha}} + \frac{\tau}{\overline{p}_{\beta}}\in D\Bigr\};
$$
the vector $\overline{\theta}=(\theta_1, \, \dots, \, \theta_d)$ is defined by the equation
\begin{align}
\label{1_theta_1lpa}
\frac{1}{\overline{\theta}} = \frac{1-\lambda}{\overline{p}_{\alpha}} + \frac{\lambda} {\overline{p}_{\beta}}.
\end{align}
Then $\lambda >0$.

We prove that $\overline{\theta} \in \Theta$.

For $\lambda = 1$, we have $\frac{1}{\overline{\theta}} = \frac{1}{\overline{p}_\beta}$; it remains to notice that $\beta\in \tilde{\cal N}_1$, $\R^d \in \tilde {\cal Z}_1$.
For $\lambda<1$, from the definition of $D$ and by the generality of position (see assertion 2 of Definition \ref{gen_pos}) there is the unique $i_*\in \{1, \, \dots, \, d\}$ such that $\theta_{\sigma(i_*)} = q_{\sigma(i_*)}$, $\theta_{\sigma(i_*)} = 2$ or $0< \omega _{\theta_{\sigma(i_*)},q_{\sigma(i_*)}} = \omega _{\theta_{\sigma(i_*+1)},q_{\sigma(i_*+1)}}<1$; for other indices $i$ the inequalities are the same as in the definition of $D$. Therefore, $(\alpha, \, \beta) \in \tilde{\cal N}_2$, $\overline{\theta}\in \Theta$.

From \eqref{psi_def} and \eqref{psi_eq_nu_phi1} we get
\begin{align}
\label{nu_a_nu_b}
\nu_\alpha \Phi(\overline{p}_\alpha, \, \overline{q}, \, \overline{k}, \, n) \le \nu_\alpha^{1-\lambda} \nu_\beta ^{\lambda} \Phi(\overline{\theta}, \, \overline{q}, \, \overline{k}, \, n).
\end{align}
From the definition of $D$, $\lambda$ and $\overline{\theta}$ it follows that
$$
\omega_{\theta_{\sigma(1)},q_{\sigma(1)}} \le \omega_{\theta_{\sigma(2)},q_{\sigma(2)}} \le \dots \le \omega_{\theta_{\sigma(d)},q_{\sigma(d)}}, \quad \theta_{\sigma(i)} \ge q_{\sigma(i)} \text{ for }i\le \mu,
$$
$$
2\le \theta_{\sigma(i)}\le q_{\sigma(i)} \text{ for }\mu+1\le i\le \nu, \quad \theta_{\sigma(i)}\le 2 \text{ for } i\ge \nu+1.
$$

\begin{enumerate}
\item Let $n\le k_{\sigma(1)} \dots k_{\sigma(\mu)} k_{\sigma(\mu+1)} ^{2/q_{\sigma(\mu+1)}} \dots k_{\sigma(d)}^{2/q_{\sigma(d)}}$. Applying \eqref{phi_case10}, \eqref{1_theta_1lpa}, we get that \eqref{nu_a_nu_b} can be written as follows: 
$$
\left(\frac{\nu_\alpha}{\nu_\beta}\right) ^\lambda \prod_{i=1}^\mu k_{\sigma(i)}^{1/q_{\sigma(i)} - 1/p_{\alpha,\sigma(i)}} \le \prod_{i=1}^\mu k_{\sigma(i)}^{1/q_{\sigma(i)} - (1-\lambda)/p_{\alpha,\sigma(i)} -\lambda/ p_{\beta, \sigma(i)}}; 
$$
hence,
$$
\nu_\alpha \prod_{i=1}^\mu k_{\sigma(i)}^{1/p_{\beta,\sigma(i)} - 1/p_{\alpha,\sigma(i)}} \le \nu_\beta.
$$
This together with \eqref{s_sigma_i1} yields \eqref{incl_ineq}.

\item Let $\mu+1\le t\le \nu$,
$$
k_{\sigma(1)} \dots k_{\sigma(t-1)} k_{\sigma(t)} ^{2/q_{\sigma(t)}} \dots k_{\sigma(d)} ^{2/q_{\sigma(d)}} < n \le k_{\sigma(1)} \dots k_{\sigma(t)} k_{\sigma(t+1)} ^{2/q_{\sigma(t+1)}} \dots k_{\sigma(d)} ^{2/q_{\sigma(d)}}.
$$
From \eqref{phi_case10}, \eqref{phi_case20}, \eqref{phi_case30}, \eqref{1_theta_1lpa} it follows that \eqref{nu_a_nu_b} can be written as follows: 
$$
\left(\frac{\nu_\alpha}{\nu_\beta}\right) ^\lambda \prod _{i=1}^{t-1} k_{\sigma(i)} ^{1/q_{\sigma(i)} - 1/p_{\alpha, \sigma(i)}} \Bigl( n^{1/2} \prod _{i=1}^{t-1} k_{\sigma(i)} ^{-1/2} \prod _{i=t}^d k_{\sigma(i)} ^{-1/q_{\sigma(i)}}\Bigr) ^{\frac{1/q_{\sigma(t)} - 1/p_{\alpha, \sigma(t)}}{1/2-1/q_{\sigma(t)}}} \le 
$$
$$
\le \prod _{i=1}^{t-1} k_{\sigma(i)} ^{1/q_{\sigma(i)} - (1-\lambda)/p_{\alpha, \sigma(i)}-\lambda/p_{\beta, \sigma(i)}} \Bigl( n^{1/2} \prod _{i=1}^{t-1} k_{\sigma(i)} ^{-1/2} \prod _{i=t}^d k_{\sigma(i)} ^{-1/q_{\sigma(i)}}\Bigr) ^{\frac{1/q_{\sigma(t)} - (1-\lambda)/p_{\alpha, \sigma(t)}-\lambda/p_{\beta,\sigma(t)}}{1/2-1/q_{\sigma(t)}}};
$$
therefore,
$$
\nu_\alpha \prod _{i=1}^{t-1} k_{\sigma(i)} ^{1/p_{\beta,\sigma(i)} - 1/p_{\alpha, \sigma(i)}} \Bigl( n^{1/2} \prod _{i=1}^{t-1} k_{\sigma(i)} ^{-1/2} \prod _{i=t}^d k_{\sigma(i)} ^{-1/q_{\sigma(i)}}\Bigr) ^{\frac{1/p_{\beta,\sigma(t)} - 1/p_{\alpha, \sigma(t)}}{1/2-1/q_{\sigma(t)}}} \le \nu_\beta.
$$
This together with \eqref{s_sigma_i2} implies \eqref{incl_ineq}.

\item Let $n>k_{\sigma(1)} \dots k_{\sigma(\nu)} k_{\sigma(\nu+1)} ^{2/q _{\sigma(\nu+1)}} \dots k_{\sigma(d)} ^{2/q _{\sigma(d)}}$. Then, by \eqref{phi_case30}, \eqref{1_theta_1lpa}, the inequality \eqref{nu_a_nu_b} can be written as follows:
$$
\left(\frac{\nu_\alpha}{\nu_\beta}\right) ^\lambda \prod _{i=1}^{\nu} k_{\sigma(i)} ^{1/q_{\sigma(i)} - 1/p_{\alpha, \sigma(i)}} n^{-1/2} \prod _{i=1}^{\nu} k_{\sigma(i)} ^{1/2} \prod _{i=\nu+1}^{d} k_{\sigma(i)} ^{1/q_{\sigma(i)}} \le 
$$
$$
\le \prod _{i=1}^{\nu} k_{\sigma(i)} ^{1/q_{\sigma(i)} - (1-\lambda)/p_{\alpha, \sigma(i)} -\lambda/p_{\beta, \sigma(i)}} n^{-1/2} \prod _{i=1}^{\nu} k_{\sigma(i)} ^{1/2} \prod _{i=\nu+1}^{d} k_{\sigma(i)} ^{1/q_{\sigma(i)}};
$$
hence,
$$
\nu_\alpha \prod _{i=1}^{\nu} k_{\sigma(i)} ^{1/p_{\beta,\sigma(i)} - 1/p_{\alpha, \sigma(i)}} \le \nu_\beta.
$$
This together with \eqref{s_sigma_i3} implies \eqref{incl_ineq}.
\end{enumerate}
This completes the proof of \eqref{dn_est_psi1}.

\subsection{The case $m\ge 2$}

For brevity, we write $\overline{\theta}(\overline{\alpha}):=\overline{\theta}(\overline{\alpha}, \, Z)$, $\lambda_j(\overline{\alpha}) = \lambda_j(\overline{\alpha}, \, Z)$, $j=1, \, \dots, \, m$.

We define the permutation $\sigma\in S_d$ according to \eqref{upor} for $\overline{p}:=\overline{\theta}(\overline{\alpha})$; it is chosen so that
\begin{align}
\label{sigma_choice}
\begin{array}{c}
\theta_{\sigma(i)}(\overline{\alpha})> q_{\sigma(i)}, \; 1\le i\le \mu_1; \quad \theta_{\sigma(i)}(\overline{\alpha}) = q_{\sigma(i)}, \; \mu_1+1\le i\le \mu_2;
\\
2< \theta_{\sigma(i)}(\overline{\alpha}) < q_{\sigma(i)}, \; \mu_2+1\le i\le \nu_1; \quad \theta_{\sigma(i)}(\overline{\alpha}) = 2, \; \nu_1+1\le i\le \nu_2;
\\
\theta_{\sigma(i)}(\overline{\alpha})< 2, \; \nu_2+1\le i \le d.
\end{array}
\end{align}

Let $I=I_2\sqcup I_q\sqcup I_{\omega_1} \sqcup \dots \sqcup I_{\omega_L}$ be the set from Definition \ref{zm_def1}. Then, by \eqref{sigma_choice} and generality of position (see assertion 2 of Definition \ref{gen_pos}), 
\begin{align}
\label{iq2_nu_nu}
I_q=\{\sigma(\mu_1+1), \, \dots, \, \sigma(\mu_2)\}, \quad I_2=\{\sigma(\nu_1+1), \, \dots, \, \sigma(\nu_2)\}.
\end{align}
Changing if necessary the numeration, we assume that $\omega_{\theta_i(\overline{\alpha}), q_i} \le \omega_{\theta_j(\overline{\alpha}), q_j}$ for $i\in I_{\omega_t}$, $j\in I_{\omega_l}$, $1\le t<l\le L$. Then from \eqref{upor} and generality of position (see assertion 2 of Definition \ref{gen_pos}) it follows that there are $t_1<t_1'<t_2<t_2'<\dots < t_L<t_L'$ such that
\begin{align}
\label{i_om_j_tj}
I_{\omega_j} = \sigma(\{t_j, \, t_j+1, \, \dots, \, t'_j\}), \quad 1\le j\le L;
\end{align}
in addition,
\begin{align}
\label{i_om_str} \omega_{\theta_{\sigma(i)}(\overline{\alpha}),q_{\sigma(i)}} < \omega_{\theta_{\sigma(i+1)}(\overline{\alpha}),q_{\sigma(i+1)}}, \quad \mu_2+1\le i < \nu_1, \; i\notin \{t_j, \, t_{j+1}, \, \dots, \, t'_{j-1}\}_{j=1}^L.
\end{align}

As in the case $m=1$, we define the set
\begin{align}
\label{w_mg2_def} W= \nu_{\alpha_1}^{\lambda_1(\overline{\alpha})} \dots \nu_{\alpha_m} ^{\lambda_m(\overline{\alpha})} u_1^{-1/\theta_1(\overline{\alpha})}\dots u_d^{-1/\theta_d(\overline{\alpha})}V^{\overline{k}} _{\overline{u}},
\end{align}
where $\overline{u} = (u_1, \, \dots, \, u_d)$, $u_i = \lceil s_i\rceil$ or $u_i=\lfloor s_i\rfloor$, and the numbers $s_i\in [1, \, k_i]$, such that $W \subset 2^d M$, will be defined below.

First we define $s_{\sigma(i)}$ for $\sigma(i)\notin I$.
\begin{enumerate}
\item If $n \le k_{\sigma(1)} \dots k_{\sigma(\mu_2)} k_{\sigma(\mu_2+1)} ^{2/q_{\sigma(\mu_2+1)}} \dots k_{\sigma(d)} ^{2/q_{\sigma(d)}}$, we set
\begin{align}
\label{s_def1} s_{\sigma(i)} = \begin{cases} k_{\sigma(i)}, & 1\le i \le \mu_1, \\ 1, & \mu_2+1\le i\le d, \; \sigma(i)\notin I.\end{cases}
\end{align}
\item Let $\mu_2+1\le t\le \nu_1$,
\begin{align}
\label{ks1tm1n}
k_{\sigma(1)} \dots k_{\sigma(t-1)} k_{\sigma(t)} ^{2/q_{\sigma(t)}} \dots k_{\sigma(d)} ^{2/q_{\sigma(d)}} < n \le k_{\sigma(1)} \dots k_{\sigma(t)} k_{\sigma(t+1)} ^{2/q_{\sigma(t)}} \dots k_{\sigma(d)} ^{2/q_{\sigma(d)}}.
\end{align}
Then we set
\begin{align}
\label{s_def2}
 s_{\sigma(i)} = 
 \begin{cases}
 k_{\sigma(i)}, & 1\le i \le t-1, \; \sigma(i)\notin I,\\
 (n^{1/2} k_{\sigma(1)}^{-1/2} \dots k_{\sigma(t-1)} ^{-1/2} k_{\sigma(t)}^{-1/q_{\sigma(t)}} \dots k_{\sigma(d)} ^{-1/q_{\sigma(d)}}) ^{\frac{1}{1/2-1/q_{\sigma(t)}}}, & i= t, \; \sigma(i)\notin I, \\ 1, & t+1\le i\le d, \; \sigma(i)\notin I.
\end{cases} 
\end{align}

\item Let $n> k_{\sigma(1)} \dots k_{\sigma(\nu_1)} k_{\sigma(\nu_1+1)} ^{2/q_{\sigma(\nu_1+1)}} \dots k_{\sigma(d)} ^{2/q_{\sigma(d)}}$. Then we set
\begin{align}
\label{s_def3} s_{\sigma(i)} = \begin{cases} k_{\sigma(i)}, & 1\le i \le \nu_1, \; \sigma(i)\notin I, \\ 1, & \nu_2+1\le i\le d.\end{cases}
\end{align}
\end{enumerate}
From \eqref{s_def1}, \eqref{ks1tm1n}, \eqref{s_def2}, \eqref{s_def3} we get that $1\le s_l\le k_l$ for $l\notin I$.

For $i\in I$, the numbers $s_i>0$ will be defined later; now we say that they will satisfy the equations
\begin{align}
\label{si_def_ini} \frac{\nu_{\alpha_j}}{\nu_{\alpha_1}} = \prod _{i=1}^d s_i ^{1/p_{\alpha_j,i} - 1/p_{\alpha_1,i}}, \quad 2\le j\le m,
\end{align}
under the conditions \eqref{s_def1}, \eqref{s_def2}, \eqref{s_def3}. If we take the logarithm of both sides of the equations \eqref{si_def_ini}, we obtain the linear system of equations on $\log s_i$, $i\in I$. Since the vectors $\{(1/\overline{p}_{\alpha_j})_I\}_{j=1}^m$ are affinely independent and $\# I\ge m-1$ (see assertion 1 of Definition \ref{gen_pos} and \eqref{card_i_zm}), this system has a solution. 

The following assertion can be checked directly.
\begin{Sta}
\label{sta_nu12}
Let the numbers $s_i>0$ satisfy \eqref{si_def_ini}, $\mu_{j,k}\ge 0$, $1\le j\le m$, $k=1, \, 2$, $\sum \limits _{j=1}^m \mu_{j,k}=1$. Let the vectors $\overline{\theta}_k$ and the numbers $\nu_{(k)}$ be defined by the equations
\begin{align}
\label{1theta_k}
\frac{1}{\overline{\theta}_k} =\sum \limits _{j=1}^m \frac{\mu_{j,k}}{\overline{p}_{\alpha_j}}, \quad \nu_{(k)} = \prod _{j=1}^m \nu_{\alpha_j}^{\mu_{j,k}}, \quad k=1, \, 2.
\end{align}
Then
\begin{align}
\label{nu1_nu2} \frac{\nu_{(1)}}{\nu_{(2)}} =  \prod _{l=1}^d s_l^{1/\theta_{1,l}-1/\theta_{2,l}}.
\end{align}
\end{Sta}

First we take for an arbitrary solution $\{s_i\}_{i\in I}$ of the equation \eqref{si_def_ini} and prove that, under the condition of generality of position, we have $I=I_2$, $I=I_q$ or $I=I_{\omega_1}$, and, in the cases $I=I_q$ and $I=I_2$, the inequalities $1\le s_l\le k_l$ hold for all $l\in I$. Then we consider the case $I=I_{\omega_1}$, construct the special solution of \eqref{si_def_ini} and prove that $1\le s_l\le k_l$, $l\in I$.

In what follows, we denote
\begin{gather}
\label{i_om_razl}
I_\omega = I_{\omega_1} \sqcup \dots \sqcup I_{\omega_L},
\\
\label{delta_simpl_def}
\Delta = {\rm conv}\, \{1/\overline{p}_{\alpha_j}\} _{j=1}^m.
\end{gather}

{\bf The case $l\in I_q$.} Then $\sigma^{-1}(l) \in \{\mu_1+1, \, \dots, \, \mu_2\}$ (see \eqref{iq2_nu_nu}). Consider the set of points $(1/p_1, \, \dots, \, 1/p_d)\in \Delta$ defined by the conditions
\begin{align}
\label{segm_def_delta}
\begin{array}{c}
p_{\sigma(i)} = q_{\sigma(i)}, \; \mu_1+1\le i\le \mu_2, \; \sigma(i)\ne l; \\
p_{\sigma(i)} = 2, \; \nu_1+1\le i\le \nu_2; \\
\omega' _{p_{\sigma(t_j)}, q_{\sigma(t_j)}} = \omega' _{p_{\sigma(t_j+1)}, q_{\sigma(t_j+1)}} = \dots =\omega' _{p_{\sigma(t'_j)}, q_{\sigma(t'_j)}}, \; 1\le j\le L,
\end{array}
\end{align}
where $\omega'_{p,q}$ is from formula \eqref{om_prime}, and the numbers $t_j$, $t'_j$ ($1\le j\le L$) are from formula \eqref{i_om_j_tj}.

This set is a segment $[1/\overline{\rho}_{(1)}, \, 1/\overline{\rho}_{(2)}]$; in addition, $1/\overline{\theta}(\overline{\alpha})= 1/\overline{\theta}(\overline{\alpha}, \, Z)$ is its interior point. Indeed, the affine hull of the set $Z \in \tilde{\cal Z}_m$ has codimension $m-1$, and the affine hull of $\Delta$ has dimension $m-1$. These planes are complementary, intersect at the point $1/\overline{\theta}(\overline{\alpha})$, and $1/\overline{\theta}(\overline{\alpha})$ lies in the relative interior of $\Delta$ (see Definition \ref{nm_def1}). The set $Z$ is defined by \eqref{pl_tr_simpl00}. In the conditions \eqref{segm_def_delta} the equations are almost the same as those in \eqref{pl_tr_simpl00}; we only exclude the equation $x_l = 1/q_l$ and the inequalities $0<\omega'_{p_{\sigma(t_j)}, q_{\sigma(t_j)}} < 1$ ($1\le j\le L$); hence the intersection of the corresponding plane with $\Delta$ is a segment with the interior point $1/\overline{\theta}(\overline{\alpha})$.

We write $\overline{\rho}_{(k)} = (\rho_{(k),1}, \, \dots, \, \rho_{(k),d})$, $k=1, \, 2$.

Without loss of generality,
\begin{align}
\label{rho_i} \rho_{(1),l} > q_l>\rho _{(2),l}.
\end{align}
The number $\lambda \in (0, \, 1)$ is defined by the equation
\begin{align}
\label{lam_def1} \frac{1}{\overline{\theta}(\overline{\alpha})} = \frac{1-\lambda}{\overline{\rho}_{(1)}} + \frac{\lambda}{\overline{\rho}_{(2)}}.
\end{align}

Let
$$
D = \{(1/p_1, \, \dots, \, 1/p_d):\; p_{\sigma(i)}> q_{\sigma(i)}, \; i\in \overline{1, \, \mu_1},
$$
$$
2<p_{\sigma(i)} < q_{\sigma(i)}, \; i\in \overline{\mu_2+1, \, \nu_1}, \;p_{\sigma(i)}<2, \; i\in \overline{\nu_2+1, \, d},
$$
$$
p_l>2, \; \omega _{p_l,q_l} < \omega_{p_{\sigma(\mu_2+1)}, q_{\sigma(\mu_2+1)}}
< \dots < \omega _{p_{\sigma(t_1-1)},q_{\sigma(t_1-1)}} < \omega _{p_{\sigma(t_1)},q_{\sigma(t_1)}}<$$$$<
\omega _{p_{\sigma(t'_1+1)},q_{\sigma(t'_1+1)}}<\dots< \omega _{p_{\sigma(t_2-1)},q_{\sigma(t_2-1)}}<\omega _{p_{\sigma(t_2)},q_{\sigma(t_2)}} < \omega _{p_{\sigma(t'_2+1)},q_{\sigma(t'_2+1)}} < $$$$<\dots < \omega _{p_{\sigma(t'_L+1)},q_{\sigma(t'_L+1)}}< \omega _{p_{\sigma(t'_L+2)},q_{\sigma(t'_L+2)}}<\dots < \omega _{p_{\sigma(\nu_1)}, q_{\sigma(\nu_1)}} \}.
$$
Then the set $D$ is open in $[0, \, 1]^d$ and $\overline{\theta}(\overline{\alpha})\in D$ (the last inclusion follows from \eqref{sigma_choice}, \eqref{i_om_j_tj} and \eqref{i_om_str}). We set
\begin{align}
\label{tau_st}
\tau_* = \inf \Bigl\{\tau \in [0, \, \lambda]:\; \frac{1-\tau}{\rho_{(1)}} + \frac{\tau}{\rho_{(2)}} \in D\Bigr\},
\end{align}
\begin{align}
\label{tau_stst}
\tau_{**} = \sup \Bigl\{\tau \in [\lambda, \, 1]:\; \frac{1-\tau}{\rho_{(1)}} + \frac{\tau}{\rho_{(2)}} \in D\Bigr\}.
\end{align}
Then $0\le \tau_*<\lambda<\tau_{**}\le 1$. The vectors $\overline{\theta}_{(1)}$ and $\overline{\theta}_{(2)}$ are defined by the equations
\begin{align}
\label{th1th2_def}
\frac{1}{\overline{\theta}_{(1)}} = \frac{1-\tau_*}{\overline{\rho}_{(1)}} + \frac{\tau_*}{\overline{\rho}_{(2)}}, \quad \frac{1}{\overline{\theta}_{(2)}} = \frac{1-\tau_{**}}{\overline{\rho}_{(1)}} + \frac{\tau_{**}}{\overline{\rho}_{(2)}}.
\end{align}
Then there is a number $\tilde \lambda \in (0, \, 1)$ such that
\begin{align}
\label{1th11l2l}
\frac{1}{\overline{\theta}(\overline{\alpha})} = \frac{1-\tilde \lambda}{\overline{\theta}_{(1)}} + \frac{\tilde \lambda}{\overline{\theta}_{(2)}}.
\end{align}

The points $1/\overline{\rho}_{(1)}$, $1/\overline{\rho}_{(2)}$ by definition lie in $\Delta$; hence there exist the numbers $\tilde \mu_{1,j}\ge 0$, $\tilde \mu_{2,j}\ge 0$ ($j=1, \, \dots, \, m$) such that $\sum \limits _{j=1}^m \tilde \mu_{1,j} = \sum \limits _{j=1}^m \tilde \mu_{2,j} = 1$ and 
\begin{align}
\label{theta_mukj} \frac{1}{\overline{\theta}_{(1)}} = \sum \limits _{j=1}^m \frac{\tilde \mu_{1,j}}{\overline{p} _{\alpha_j}}, \quad \frac{1}{\overline{\theta}_{(2)}} = \sum \limits _{j=1}^m \frac{\tilde \mu_{2,j}}{\overline{p} _{\alpha_j}}.
\end{align}
We set
\begin{align}
\label{nu12_def}
\nu_{(1)} = \prod _{j=1}^m\nu_{\alpha_j}^{\tilde\mu_{1,j}}, \quad \nu_{(2)} = \prod _{j=1}^m\nu_{\alpha_j}^{\tilde\mu_{2,j}}.
\end{align}
By Proposition \ref{sta_nu12}, \eqref{nu1_nu2} holds and 
\begin{align}
\label{nu_a_12}
\nu_{\alpha_1}^{\lambda_1(\overline{\alpha})} \dots \nu_{\alpha_m}^{\lambda_m(\overline{\alpha})} = \nu_{(1)}^{1-\tilde \lambda} \nu_{(2)}^{\tilde \lambda};
\end{align}
the last equality follows from \eqref{1th11l2l} and affine independence of the vectors $\{(1/\overline{p}_{\alpha_j})_{I'}\}_{j=1}^m$ for any set $I'\subset \{1, \, \dots, \, d\}$ of cardinality $m-1$ (see assertion 1 of Definition \ref{gen_pos}).

We show that $\overline{\theta}_{(k)}\in \Theta$, $k=1, \, 2$. Consider the case $k=1$ (for $k=2$, the proof is similar).

Let $\tau_*=0$. We prove that 
\begin{enumerate}
\item the point $1/\overline{\theta}_{(1)}$ lies in the relative interior of $m-2$-dimensional face $\Delta'$ of the simplex $\Delta$;
\item $\theta_{(1),\sigma(i)}= q_{\sigma(i)}$ for $i \in \{\mu_1+1, \, \dots, \, \mu_2\} \backslash \{\sigma^{-1}(l)\}$, $\theta_{(1),\sigma(i)}=2$ for $i \in \{\nu_1+1, \, \dots, \, \nu_2\}$, 
\begin{align}
\label{om_ttj_in_01}
\omega'_{\theta_{(1),\sigma(t_j)},q_{\sigma(t_j)}} = \omega'_{\theta_{(1),\sigma(t_j+1)},q_{\sigma(t_j+1)}} = \dots = \omega'_{\theta_{(1),\sigma(t_j')},q_{\sigma(t_j')}}\in (0, \, 1)
\end{align}
for $1\le j\le L$.
\end{enumerate}
Indeed, $1/\overline{\theta}_{(1)}=1/\overline{\rho}_{(1)}$; by definition, this point belongs to the boundary of $\Delta$; hence, it belongs to its $m-2$-dimensional face. The plane defined by equations \eqref{segm_def_delta}, by generality of position (see assertion 2 of Definition \ref{gen_pos}), cannot intersect with $k$-dimensional face of $\Delta$ for $k<m-2$; hence, $1/\overline{\rho}_{(1)}$ can be only the relative interior point of $m-2$-dimensional face. The second assertion follows from the definition of $D$, \eqref{segm_def_delta} and generality of position (if $\omega'_{\theta_{(1),\sigma(t_j)},q_{\sigma(t_j)}}=0$ or $\omega'_{\theta_{(1),\sigma(t_j)},q_{\sigma(t_j)}}=1$, then another equality of the form $\theta_{(1),\sigma(t_j)} = q_{\sigma(t_j)}$ or $\theta_{(1), \sigma(t_j)}=2$ will appear, which contradicts to assertion 2 of Definition \ref{gen_pos}).

Finally, ${\rm aff}\, \Delta'$ and the plane defined by \eqref{segm_def_delta} are complementary (again by assertion 2 of Definition \ref{gen_pos}).

Hence, for $\tau_*=0$ we have $\overline{\theta}_{(1)}\in \Theta$.

Let $\tau_*>0$. Then the point $1/\overline{\theta}_{(1)}$ belongs to the relative interior of the simplex $\Delta$; here the following cases are possible:
\begin{itemize}
\item Besides \eqref{segm_def_delta}, the following equation holds:
\begin{align}
\label{t1iqi2om}
\theta_{(1),i}=q_i, \quad \theta_{(1),i} = 2 \quad \text{or} \quad \omega'_{\theta _{(1),i},q_i} = \omega'_{\theta_{(1),j},q_j} \in [0, \, 1].
\end{align}
Taking into account assertion 2 of Definition \ref{gen_pos}, we get that in the last case there is the inclusion into the interval $(0, \, 1)$; in addition, as for $\tau_*=0$, \eqref{om_ttj_in_01} holds. The set defined by \eqref{segm_def_delta}, \eqref{om_ttj_in_01}, \eqref{t1iqi2om} belongs to $\tilde{\cal Z}_m$; by generality of position, its affine hull is complementary to ${\rm aff}\, \Delta$. Hence $\overline{\theta}_{(1)}\in \Theta$.

\item There is the unique $j\in \{1, \, \dots, \, L\}$ for which the condition \eqref{om_ttj_in_01} is changed by the equations $\theta_{(1), \sigma(t_j)} = q _{\sigma(t_j)}, \, \dots, \, \theta_{(1), \sigma(t'_j)} = q _{\sigma(t'_j)}$ or $\theta_{(1), \sigma(t_j)} =\dots =\theta_{(1), \sigma(t'_j)} = 2$. Again by generality of position we get a set from $\tilde{\cal Z}_m$, whose affine hull is complementary to ${\rm aff}\, \Delta$; as in the previous cases, $\overline{\theta}_{(1)}\in \Theta$.
\end{itemize}

From \eqref{psi_def}, \eqref{psi_eq_nu_phi}, \eqref{nu_a_12} and inclusions $\overline{\theta}_{(1)}$, $\overline{\theta}_{(2)} \in \Theta$, we get
\begin{gather}
\label{nu1lnu_pth} \nu_{(1)}^{1-\tilde \lambda} \nu _{(2)} ^{\tilde \lambda} \Phi(\overline{\theta}(\overline{\alpha}), \, \overline{q}, \, \overline{k}, \, n) \le \nu_{(1)} \Phi(\overline{\theta}_{(1)}, \, \overline{q}, \, \overline{k}, \, n),
\\
\label{nu2lnu_pth} \nu_{(1)}^{1-\tilde \lambda} \nu _{(2)} ^{\tilde \lambda} \Phi(\overline{\theta}(\overline{\alpha}), \, \overline{q}, \, \overline{k}, \, n) \le \nu_{(2)} \Phi(\overline{\theta}_{(2)}, \, \overline{q}, \, \overline{k}, \, n).
\end{gather}

By \eqref{rho_i}, \eqref{lam_def1}, \eqref{tau_st}, \eqref{tau_stst}, \eqref{th1th2_def}, we have
\begin{align}
\label{theta_1gql} \theta_{(1),l}> q_l> \theta_{(2),l}.
\end{align}

Let
\begin{align}
\label{nllll}
n\le \prod _{1\le i\le \mu_2, \, \sigma(i)\ne l} k_{\sigma(i)} \cdot k_l ^{2/q_l} \prod _{i=\mu_2+1}^d k_{\sigma(i)} ^{2/q_{\sigma(i)}}. 
\end{align}
Since $1/\overline{\theta}_{(1)}\in \overline{D}$, by the definition of $D$, \eqref{phi_case1}, \eqref{sigma_choice}, \eqref{segm_def_delta} and \eqref{theta_1gql} we get that the inequality \eqref{nu1lnu_pth} can be written as follows:
$$
\nu_{(1)}^{1-\tilde \lambda} \nu _{(2)} ^{\tilde \lambda} \prod _{i=1}^{\mu_1} k_{\sigma(i)}^{1/q_{\sigma(i)}-1/\theta_{\sigma(i)}(\overline{\alpha})} \le \nu_{(1)} \prod _{i=1}^{\mu_1} k_{\sigma(i)}^{1/q_{\sigma(i)}-1/\theta_{(1),\sigma(i)}} \cdot k_l^{1/q_l-1/\theta_{(1),l}}.
$$
Taking into account \eqref{1th11l2l} and the equality $\theta_l(\overline{\alpha})=q_l$, we get that
\begin{align}
\label{n12ge_prod_kl}
\frac{\nu_{(1)}}{\nu_{(2)}} \ge \prod _{i=1}^{\mu_1} k_{\sigma(i)}^{1/\theta_{(1),\sigma(i)} - 1/\theta_{(2),\sigma(i)}} \cdot k_l^{1/\theta_{(1),l}-1/\theta_{(2),l}}.
\end{align}
Now we apply \eqref{nu1_nu2}, take into account that the points $1/\overline{\theta}(\overline{\alpha})$ and $1/\overline{\theta} _{(1)}$ satisfy \eqref{segm_def_delta}, and obtain that
$$
\frac{\nu_{(1)}}{\nu_{(2)}} = \prod _{i=1}^d s_{\sigma(i)} ^{1/\theta_{(1),\sigma(i)} - 1/\theta_{(2),\sigma(i)}} \stackrel{\eqref{s_def1}}{=} \prod _{i=1}^{\mu_1} k_{\sigma(i)}^{1/\theta_{(1),\sigma(i)} - 1/\theta_{(2),\sigma(i)}} \cdot s_l^{1/\theta_{(1),l}-1/\theta_{(2),l}} \times$$$$\times\prod _{\mu_1+1\le i\le \mu_2, \, \sigma(i)\ne l} s_{\sigma(i)} ^{1/q_{\sigma(i)} -1/q_{\sigma(i)}} \prod _{\mu_2+1\le i\le \nu_1, \, \sigma(i) \notin I_\omega} 1^{1/\theta_{(1),\sigma(i)} - 1/\theta_{(2),\sigma(i)}} \times$$$$\times\prod _{\sigma(i)\in I_\omega} s_{\sigma(i)} ^{1/\theta_{(1),\sigma(i)} - 1/\theta_{(2),\sigma(i)}} \prod _{\nu_1+1\le i\le \nu_2} s_{\sigma(i)} ^{1/2-1/2} \prod _{\nu_2+1\le i\le d} 1^{1/\theta_{(1),\sigma(i)} - 1/\theta_{(2),\sigma(i)}};
$$
this together with \eqref{n12ge_prod_kl} yields $s_l^{1/\theta_{(1),l}-1/\theta_{(2),l}} \prod _{\sigma(i)\in I_\omega} s_{\sigma(i)} ^{1/\theta_{(1),\sigma(i)} - 1/\theta_{(2),\sigma(i)}} \ge k_l^{1/\theta_{(1),l}-1/\theta_{(2),l}}$;
hence, by \eqref{theta_1gql},
\begin{align}
\label{est_predv1} s_l \le k_l \prod _{\sigma(i)\in I_\omega} s_{\sigma(i)} ^{-\frac{1/\theta_{(2),\sigma(i)} - 1/\theta_{(1),\sigma(i)}}{1/\theta_{(2),l}-1/\theta_{(1),l}}}.
\end{align}

Since $\theta_{(2),\sigma(i)}\ge q_{\sigma(i)}$ for $1\le i\le \mu_2$, $\sigma(i)\ne l$, $\theta_{(2),\sigma(i)}\le q_{\sigma(i)}$ for $\mu_2+1\le i\le d$ (by \eqref{segm_def_delta} and the definition of $D$), $\theta_{(2),l} \stackrel{\eqref{theta_1gql}}{<} q_l$, then, under condition \eqref{nllll}, by \eqref{phi_case1}, the inequality \eqref{nu2lnu_pth} similarly can be written as follows:
$$
\nu_{(1)}^{1-\tilde \lambda} \nu _{(2)} ^{\tilde \lambda} \prod _{i=1}^{\mu_1} k_{\sigma(i)}^{1/q_{\sigma(i)}-1/\theta_{\sigma(i)}(\overline{\alpha})} \le \nu_{(2)} \prod _{i=1}^{\mu_1} k_{\sigma(i)}^{1/q_{\sigma(i)}-1/\theta_{(2),\sigma(i)}};
$$
therefore,
\begin{align}
\label{est_predv2} s_l \ge \prod _{\sigma(i)\in I_\omega} s_{\sigma(i)} ^{-\frac{1/\theta_{(2),\sigma(i)} - 1/\theta_{(1),\sigma(i)}}{1/\theta_{(2),l}-1/\theta_{(1),l}}}.
\end{align}

If $\prod _{1\le i\le \mu_2, \, \sigma(i)\ne l} k_{\sigma(i)} \cdot k_l ^{2/q_l} \prod _{i=\mu_2+1}^d k_{\sigma(i)} ^{2/q_{\sigma(i)}}< n \le \prod _{1\le i\le \mu_2} k_{\sigma(i)} \prod _{i=\mu_2+1}^d k_{\sigma(i)} ^{2/q_{\sigma(i)}}$, the inequality \eqref{nu1lnu_pth} can be written as in the previous case, and we again obtain \eqref{est_predv1}. From the definition of $D$ and \eqref{theta_1gql} it follows that if $\mu_2+1\le \nu_1$, then $0<\omega_{\theta_{(2),l},q_l}\le \omega _{\theta _{(2), \sigma(\mu_2+1)},q_{\sigma(\mu_2+1)}}$, and if $\mu_2+1>\nu_1$, then $\theta _{(2),l}\in [2, \, q_l)$. Hence, by \eqref{phi_case1}, \eqref{phi_case2}, \eqref{phi_case3} and the conditions $\theta_{(1),\sigma(i)} = q_{\sigma(i)}$ ($\mu_1+1\le i\le \mu_2$, $\sigma(i)\ne l$), inequality \eqref{nu2lnu_pth} can be written as follows:
$$
\nu_{(1)}^{1-\tilde \lambda} \nu _{(2)} ^{\tilde \lambda} \prod _{i=1}^{\mu_1} k_{\sigma(i)}^{1/q_{\sigma(i)}-1/\theta_{\sigma(i)}(\overline{\alpha})} \le \nu_{(2)} \prod _{i=1}^{\mu_1} k_{\sigma(i)}^{1/q_{\sigma(i)}-1/\theta_{(2),\sigma(i)}}\times$$$$ \times \Bigl(n^{1/2} \prod _{1\le i\le \mu_2, \, \sigma(i)\ne l} k_{\sigma(i)}^{-1/2} \cdot k_l ^{-1/q_l} \prod _{i=\mu_2+1}^d k_{\sigma(i)} ^{-1/q_{\sigma(i)}}\Bigr)^{\frac{1/q_l-1/\theta_{(2),l}}{1/2-1/q_l}}.
$$
This together with \eqref{s_def1}, \eqref{nu1_nu2}, \eqref{segm_def_delta} implies that
\begin{align}
\label{est_predv3} s_l \ge \Bigl(n^{1/2} \prod _{1\le i\le \mu_2, \, \sigma(i)\ne l} k_{\sigma(i)}^{-1/2} \cdot k_l ^{-1/q_l} \prod _{i=\mu_2+1}^d k_{\sigma(i)} ^{-1/q_{\sigma(i)}}\Bigr)^{\frac{1}{1/2-1/q_l}}\prod _{\sigma(i)\in I_\omega} s_{\sigma(i)} ^{-\frac{1/\theta_{(2),\sigma(i)} - 1/\theta_{(1),\sigma(i)}}{1/\theta_{(2),l}-1/\theta_{(1),l}}}.
\end{align}

Let $\mu_2+1\le t\le \nu_1$, $\prod _{i=1}^{t-1} k_{\sigma(i)} \prod _{i=t}^d k_{\sigma(i)}^{2/q_{\sigma(i)}}< n\le \prod _{i=1}^t k_{\sigma(i)} \prod _{i=t+1}^d k_{\sigma(i)}^{2/q_{\sigma(i)}}$. By \eqref{phi_case20}, \eqref{phi_case1}, \eqref{phi_case3}, \eqref{segm_def_delta} and the definition of the set $D$, inequality \eqref{nu1lnu_pth} can be written as follows:
\begin{align}
\label{nu_t_eq111}
\begin{array}{c}
\nu_{(1)}^{1-\tilde \lambda} \nu _{(2)} ^{\tilde \lambda} \prod _{i=1}^{t-1} k_{\sigma(i)}^{1/q_{\sigma(i)}-1/\theta_{\sigma(i)}(\overline{\alpha})} \Bigl(n^{-1/2} \prod _{i=1}^{t-1} k_{\sigma(i)}^{1/2} \prod _{i=t}^d k_{\sigma(i)}^{1/q_{\sigma(i)}}\Bigr)^{\frac{1/\theta_{\sigma(t)}(\overline{\alpha})-1/q_{\sigma(t)}}{1/2-1/q_{\sigma(t)}}}\le \\ \le\nu_{(1)} \prod _{i=1}^{t-1} k_{\sigma(i)}^{1/q_{\sigma(i)}-1/\theta_{(1),\sigma(i)}} \Bigl(n^{-1/2} \prod _{i=1}^{t-1} k_{\sigma(i)}^{1/2} \prod _{i=t}^d k_{\sigma(i)}^{1/q_{\sigma(i)}}\Bigr)^{\frac{1/\theta_{(1),\sigma(t)}-1/q_{\sigma(t)}}{1/2-1/q_{\sigma(t)}}};
\end{array}
\end{align}
this is equivalent to the condition
$$
\frac{\nu_{(1)}}{\nu_{(2)}} \ge \prod _{i=1}^{t-1} k_{\sigma(i)}^{1/\theta_{(1),\sigma(i)}-1/\theta_{(2),\sigma(i)}} \Bigl(n^{1/2} \prod _{i=1}^{t-1} k_{\sigma(i)}^{-1/2} \prod _{i=t}^d k_{\sigma(i)}^{-1/q_{\sigma(i)}}\Bigr)^{\frac{1/\theta_{(1),\sigma(t)}-1/\theta_{(2), \sigma(t)}}{1/2-1/q_{\sigma(t)}}}.
$$
Inequality \eqref{nu2lnu_pth} has the form
\begin{align}
\label{nu_t_eq222}
\begin{array}{c}
\nu_{(1)}^{1-\tilde \lambda} \nu _{(2)} ^{\tilde \lambda} \prod _{i=1}^{t-1} k_{\sigma(i)}^{1/q_{\sigma(i)}-1/\theta_{\sigma(i)}(\overline{\alpha})} \Bigl(n^{-1/2} \prod _{i=1}^{t-1} k_{\sigma(i)}^{1/2} \prod _{i=t}^d k_{\sigma(i)}^{1/q_{\sigma(i)}}\Bigr)^{\frac{1/\theta_{\sigma(t)}(\overline{\alpha})-1/q_{\sigma(t)}}{1/2-1/q_{\sigma(t)}}}\le \\ \le\nu_{(2)} \prod _{i=1}^{t-1} k_{\sigma(i)}^{1/q_{\sigma(i)}-1/\theta_{(2),\sigma(i)}} \Bigl(n^{-1/2} \prod _{i=1}^{t-1} k_{\sigma(i)}^{1/2} \prod _{i=t}^d k_{\sigma(i)}^{1/q_{\sigma(i)}}\Bigr)^{\frac{1/\theta_{(2),\sigma(t)}-1/q_{\sigma(t)}}{1/2-1/q_{\sigma(t)}}};
\end{array}
\end{align}
it is equivalent to
$$
\frac{\nu_{(1)}}{\nu_{(2)}} \le \prod _{i=1}^{t-1} k_{\sigma(i)}^{1/\theta_{(1),\sigma(i)}-1/\theta_{(2),\sigma(i)}} \Bigl(n^{1/2} \prod _{i=1}^{t-1} k_{\sigma(i)}^{-1/2} \prod _{i=t}^d k_{\sigma(i)}^{-1/q_{\sigma(i)}}\Bigr)^{\frac{1/\theta_{(1),\sigma(t)}-1/\theta_{(2), \sigma(t)}}{1/2-1/q_{\sigma(t)}}}.
$$
Therefore,
$$
\frac{\nu_{(1)}}{\nu_{(2)}} = \prod _{i=1}^{t-1} k_{\sigma(i)}^{1/\theta_{(1),\sigma(i)}-1/\theta_{(2),\sigma(i)}} \Bigl(n^{1/2} \prod _{i=1}^{t-1} k_{\sigma(i)}^{-1/2} \prod _{i=t}^d k_{\sigma(i)}^{-1/q_{\sigma(i)}}\Bigr)^{\frac{1/\theta_{(1),\sigma(t)}-1/\theta_{(2), \sigma(t)}}{1/2-1/q_{\sigma(t)}}},
$$
and in \eqref{nu_t_eq111}, \eqref{nu_t_eq222} the inequalities turn to equalities; hence
\begin{align}
\label{phi12_eq}
\nu_{(1)}^{1-\tilde \lambda} \nu _{(2)} ^{\tilde \lambda} \Phi(\overline{\theta}(\overline{\alpha}), \, \overline{q}, \, \overline{k}, \, n) = \nu_{(1)} \Phi(\overline{\theta}_{(1)}, \, \overline{q}, \, \overline{k}, \, n)=\nu_{(2)} \Phi(\overline{\theta}_{(2)}, \, \overline{q}, \, \overline{k}, \, n),
\end{align}
which contradicts to generality of position (see assertion 4 of Definition \ref{gen_pos}).

For $n> \prod _{i=1}^{\nu_1} k_{\sigma(i)} \prod _{i=\nu_1+1}^d k_{\sigma(i)} ^{2/q_{\sigma(i)}}$, we apply \eqref{phi_case3} and get
\begin{align}
\label{nu_t_eq1118}
\begin{array}{c}
\nu_{(1)}^{1-\tilde \lambda} \nu _{(2)} ^{\tilde \lambda} \prod _{i=1}^{\nu_1} k_{\sigma(i)}^{1/q_{\sigma(i)}-1/\theta_{\sigma(i)}(\overline{\alpha})} \cdot n^{-1/2} \prod _{i=1}^{\nu_1} k_{\sigma(i)}^{1/2} \prod _{i=\nu_1+1}^d k_{\sigma(i)}^{1/q_{\sigma(i)}}\le \\ \le\nu_{(1)} \prod _{i=1}^{\nu_1} k_{\sigma(i)}^{1/q_{\sigma(i)}-1/\theta_{(1),\sigma(i)}} \cdot n^{-1/2} \prod _{i=1}^{\nu_1} k_{\sigma(i)}^{1/2} \prod _{i=\nu_1+1}^d k_{\sigma(i)}^{1/q_{\sigma(i)}},
\end{array}
\end{align}
\begin{align}
\label{nu_t_eq2228}
\begin{array}{c}
\nu_{(1)}^{1-\tilde \lambda} \nu _{(2)} ^{\tilde \lambda} \prod _{i=1}^{\nu_1} k_{\sigma(i)}^{1/q_{\sigma(i)}-1/\theta_{\sigma(i)}(\overline{\alpha})} \cdot n^{-1/2} \prod _{i=1}^{\nu_1} k_{\sigma(i)}^{1/2} \prod _{i=\nu_1+1}^d k_{\sigma(i)}^{1/q_{\sigma(i)}} \le \\ \le\nu_{(2)} \prod _{i=1}^{\nu_1} k_{\sigma(i)}^{1/q_{\sigma(i)}-1/\theta_{(2),\sigma(i)}} \cdot n^{-1/2} \prod _{i=1}^{\nu_1} k_{\sigma(i)}^{1/2} \prod _{i=\nu_1+1}^d k_{\sigma(i)}^{1/q_{\sigma(i)}};
\end{array}
\end{align}
this implies that
\begin{align}
\label{nu_prod_eq}
\frac{\nu_{(1)}}{\nu_{(2)}} = \prod _{i=1}^{\nu_1} k_{\sigma(i)}^{1/\theta_{(1),\sigma(i)}-1/\theta_{(2),\sigma(i)}}.
\end{align}
We again get \eqref{phi12_eq}; this contradicts to generality of position.

Thus, we proved
\begin{Sta}
\label{iq_sta} Let \eqref{psi_eq_nu_phi} hold. Then
\begin{enumerate}
\item For $n\le \prod _{1\le i\le \mu_2, \, \sigma(i)\ne l} k_{\sigma(i)} \cdot k_l ^{2/q_l} \prod _{i=\mu_2+1}^d k_{\sigma(i)} ^{2/q_{\sigma(i)}}$ $(l\in I_q)$, inequalities \eqref{est_predv1}, \eqref{est_predv2} hold.

\item For $\prod _{1\le i\le \mu_2, \, \sigma(i)\ne l} k_{\sigma(i)} \cdot k_l ^{2/q_l} \prod _{i=\mu_2+1}^d k_{\sigma(i)} ^{2/q_{\sigma(i)}}< n\le \prod _{i=1}^{\mu_2} k_{\sigma(i)}\prod _{i=\mu_2+1}^d k_{\sigma(i)} ^{2/q_{\sigma(i)}}$ $(l\in I_q)$, inequalities \eqref{est_predv1}, \eqref{est_predv3}.

\item If $n>\prod _{i=1}^{\mu_2} k_{\sigma(i)}\prod _{i=\mu_2+1}^d k_{\sigma(i)} ^{2/q_{\sigma(i)}}$, then $I_q = \varnothing$.
\end{enumerate}
\end{Sta}

{\bf The case $l\in I_2$.} Then $\sigma^{-1}(l) \in \{\nu_1+1, \, \dots, \, \nu_2\}$ (see \eqref{iq2_nu_nu}). Consider the set of points $(1/p_1, \, \dots, \, 1/p_d)\in \Delta$ defined by the conditions
\begin{align}
\label{segm_def_delta1}
\begin{array}{c}
p_{\sigma(i)} = q_{\sigma(i)}, \; \mu_1+1\le i\le \mu_2; \\
p_{\sigma(i)} = 2, \; \nu_1+1\le i\le \nu_2, \; \sigma(i)\ne l; \\
\omega' _{p_{\sigma(t_j)}, q_{\sigma(t_j)}} = \omega' _{p_{\sigma(t_j+1)}, q_{\sigma(t_j+1)}} = \dots =\omega' _{p_{\sigma(t'_j)}, q_{\sigma(t'_j)}}, \; 1\le j\le L,
\end{array}
\end{align}
where $\omega'_{p,q}$ is from formula \eqref{om_prime}.

As in the case $l\in I_q$, we get that, by generality of position, this set is a segment $[1/\overline{\rho}_{(1)}, \, 1/\overline{\rho}_{(2)}]$ with the interior point $1/\overline{\theta}(\overline{\alpha})$. We may assume that
\begin{align}
\label{rho_i1} \rho_{(1),l} > 2>\rho _{(2),l}.
\end{align}
The number $\lambda \in (0, \, 1)$ is defined by the equation \eqref{lam_def1}.

Let
$$
D = \{(1/p_1, \, \dots, \, 1/p_d):\; p_{\sigma(i)}> q_{\sigma(i)}, \; i\in \overline{1, \, \mu_1},
$$
$$
2<p_{\sigma(i)} < q_{\sigma(i)}, \; i\in \overline{\mu_2+1, \, \nu_1}, \;p_{\sigma(i)}<2, \; i\in \overline{\nu_2+1, \, d},
$$
$$
p_l<q_l, \; \omega_{p_{\sigma(\mu_2+1)}, q_{\sigma(\mu_2+1)}}
< \dots < \omega _{p_{\sigma(t_1-1)},q_{\sigma(t_1-1)}} < \omega _{p_{\sigma(t_1)},q_{\sigma(t_1)}}<$$$$<
\omega _{p_{\sigma(t'_1+1)},q_{\sigma(t'_1+1)}}<\dots < \omega _{p_{\sigma(t_2-1)},q_{\sigma(t_2-1)}}<\omega _{p_{\sigma(t_2)},q_{\sigma(t_2)}} < \omega _{p_{\sigma(t'_2+1)},q_{\sigma(t'_2+1)}} < $$$$<\dots < \omega _{p_{\sigma(t'_L+1)},q_{\sigma(t'_L+1)}}< \omega _{p_{\sigma(t'_L+2)},q_{\sigma(t'_L+2)}} < \dots<\omega _{p_{\sigma(\nu_1)}, q_{\sigma(\nu_1)}}<\omega _{p_l,q_l} \}.
$$
Then the set $D$ is open in $[0, \, 1]^d$ and $\overline{\theta}(\overline{\alpha})\in D$.

The numbers $\tau_*$, $\tau_{**}$ are defined by formulas \eqref{tau_st}, \eqref{tau_stst}. Then $0\le \tau_*<\lambda <\tau_{**}\le 1$. The vectors $\overline{\theta}_{(1)}$, $\overline{\theta}_{(2)}$ are defined by formula \eqref{th1th2_def}, the number $\tilde \lambda \in (0, \, 1)$ is defined by \eqref{1th11l2l}, and the numbers $\nu_{(1)}$, $\nu_{(2)}$, by formulas \eqref{theta_mukj}, \eqref{nu12_def}. As in the previous case,we get \eqref{nu_a_12}, \eqref{nu1lnu_pth} and \eqref{nu2lnu_pth}. In addition, from \eqref{rho_i1} it follows that
\begin{align}
\label{th1li2} \theta_{(1),l}> 2 > \theta_{(2),l}.
\end{align}

If $n\le \prod _{i=1}^{\mu_2}k_{\sigma(i)} \prod _{i=\mu_2+1}^d k_{\sigma(i)}^{2/q_{\sigma(i)}}$, inequalities \eqref{nu1lnu_pth} and \eqref{nu2lnu_pth} can be written, respectively, as follows:
\begin{align}
\label{nunlmu1}
\nu_{(1)}^{1-\tilde \lambda}\nu_{(2)}^{\tilde \lambda} \prod _{j=1}^{\mu_2} k_{\sigma(j)} ^{1/q_{\sigma(j)} - 1/\theta _{\sigma(j)}(\overline{\alpha})} \le \nu_{(1)}\prod _{j=1}^{\mu_2} k_{\sigma(j)} ^{1/q_{\sigma(j)} - 1/\theta _{(1),\sigma(j)}},
\end{align}
\begin{align}
\label{nunlmu2}
\nu_{(1)}^{1-\tilde \lambda}\nu_{(2)}^{\tilde \lambda} \prod _{j=1}^{\mu_2} k_{\sigma(j)} ^{1/q_{\sigma(j)} - 1/\theta _{\sigma(j)}(\overline{\alpha})} \le \nu_{(2)}\prod _{j=1}^{\mu_2} k_{\sigma(j)} ^{1/q_{\sigma(j)} - 1/\theta _{(2),\sigma(j)}};
\end{align}
hence,
$$
\frac{\nu_{(1)}}{\nu_{(2)}} = \prod _{j=1}^{\mu_2} k_{\sigma(j)} ^{1/\theta _{(1),\sigma(j)}-1/\theta _{(2), \sigma(j)}}.
$$
Therefore, the inequalities in \eqref{nunlmu1}, \eqref{nunlmu2} turn to equalities, and we get \eqref{phi12_eq}; this contradicts to generality of position.

If $\mu_2+1 \le t \le \nu_1$, $\prod _{i=1}^{t-1} k_{\sigma(i)} \prod _{i=t}^d k_{\sigma(i)} ^{2/q_{\sigma(i)}} < n\le \prod _{i=1}^t k_{\sigma(i)} \prod _{i=t+1}^d k_{\sigma(i)} ^{2/q_{\sigma(j)}}$, the inequalities \eqref{nu1lnu_pth} and \eqref{nu2lnu_pth} can be written, respectively, as \eqref{nu_t_eq111}, \eqref{nu_t_eq222}; this again implies \eqref{phi12_eq}, which contradicts to generality of position.

Let $\prod _{i=1}^{\nu_1} k_{\sigma(i)} \prod _{i=\nu_1+1}^d k_{\sigma(i)} ^{2/q_{\sigma(i)}}< n \le \prod _{i=1}^{\nu_1} k_{\sigma(i)}\cdot k_l \prod _{\nu_1+1\le i\le d, \, \sigma(i)\ne l} k_{\sigma(i)} ^{2/q_{\sigma(i)}}$. From \eqref{th1li2} and the definition of $D$ it follows that, for $\mu_2+1\le \nu_1$, the inequalities $1>\omega_{\theta_{(1),l},q_l}\ge \omega_{\theta_{(1),\sigma(\nu_1)},q_{\sigma(\nu_1)}}$ hold, and for $\mu_2+1>\nu_1$, the inequalities $2>\theta _{(1),l}\ge q_l$ hold. Hence, \eqref{nu1lnu_pth} and \eqref{nu2lnu_pth} can be written, respectively, as follows (see \eqref{phi_case20}, \eqref{phi_case3}, \eqref{th1li2}):
$$
\nu_{(1)}^{1-\tilde \lambda}\nu_{(2)}^{\tilde \lambda} \prod _{j=1}^{\nu_1} k_{\sigma(j)} ^{1/q_{\sigma(j)} - 1/\theta _{\sigma(j)}(\overline{\alpha})} n^{-1/2} \prod _{i=1}^{\nu_1} k_{\sigma(i)} ^{1/2} \prod _{i=\nu_1+1}^d k_{\sigma(i)} ^{1/q_{\sigma(i)}} \le
$$
$$
\le\nu_{(1)} \prod _{j=1}^{\nu_1} k_{\sigma(j)} ^{1/q_{\sigma(j)} - 1/\theta _{(1),\sigma(j)}} \Bigl( n^{1/2} \prod _{i=1}^{\nu_1} k_{\sigma(i)} ^{-1/2} \prod _{i=\nu_1+1}^d k_{\sigma(i)} ^{-1/q_{\sigma(i)}}\Bigr) ^{\frac{1/q_{l} - 1/ \theta _{(1),l}}{1/2-1/q_{l}}},
$$
$$
\nu_{(1)}^{1-\tilde \lambda}\nu_{(2)}^{\tilde \lambda} \prod _{j=1}^{\nu_1} k_{\sigma(j)} ^{1/q_{\sigma(j)} - 1/\theta _{\sigma(j)}(\overline{\alpha})} n^{-1/2} \prod _{i=1}^{\nu_1} k_{\sigma(i)} ^{1/2} \prod _{i=\nu_1+1}^d k_{\sigma(i)} ^{1/q_{\sigma(i)}} \le
$$
$$
\le \nu_{(2)} \prod _{j=1}^{\nu_1} k_{\sigma(j)} ^{1/q_{\sigma(j)} - 1/\theta _{(2),\sigma(j)}} n^{-1/2} \prod _{i=1}^{\nu_1} k_{\sigma(i)} ^{1/2} \prod _{i=\nu_1+1}^d k_{\sigma(i)} ^{1/q_{\sigma(i)}};
$$
this together with \eqref{1th11l2l} and the equality $\theta_l(\overline{\alpha})=2$ yields that
$$
\frac{\nu_{(1)}}{\nu_{(2)}} \ge \prod _{j=1}^{\nu_1} k_{\sigma(j)} ^{1/\theta_{(1),\sigma(j)} - 1/\theta _{(2),\sigma(j)}} \Bigl( n^{1/2} \prod _{i=1}^{\nu_1} k_{\sigma(i)} ^{-1/2} \prod _{i=\nu_1+1}^d k_{\sigma(i)} ^{-1/q_{\sigma(i)}}\Bigr) ^{\frac{1/\theta_{(1),l} - 1/ \theta _{(2),l}}{1/2-1/q_l}},
$$
$$
\frac{\nu_{(1)}}{\nu_{(2)}} \le \prod _{j=1}^{\nu_1} k_{\sigma(j)} ^{1/\theta_{(1),\sigma(j)} - 1/\theta _{(2),\sigma(j)}}.
$$
Applying \eqref{s_def3}, \eqref{nu1_nu2}, \eqref{segm_def_delta1}, \eqref{th1li2}, we get
\begin{align}
\label{est_predv4} s_l \le \Bigl(n^{1/2} \prod _{i=1}^{\nu_1} k_{\sigma(i)}^{-1/2} \prod _{i=\nu_1+1}^d k_{\sigma(i)} ^{-1/q_{\sigma(i)}}\Bigr)^{\frac{1}{1/2-1/q_l}}\prod _{\sigma(i)\in I_\omega} \left(\frac{k_{\sigma(i)}}{s_{\sigma(i)}}\right) ^{\frac{1/\theta_{(2),\sigma(i)} - 1/\theta_{(1),\sigma(i)}}{1/\theta_{(2),l}-1/\theta_{(1),l}}},
\end{align}
\begin{align}
\label{est_predv5} s_l \ge \prod _{\sigma(i)\in I_\omega} \left(\frac{k_{\sigma(i)}}{s_{\sigma(i)}}\right) ^{\frac{1/\theta_{(2),\sigma(i)} - 1/\theta_{(1),\sigma(i)}}{1/\theta_{(2),l}-1/\theta_{(1),l}}}.
\end{align}

If $n>\prod _{i=1}^{\nu_1} k_{\sigma(i)}\cdot k_l \prod _{\nu_1+1\le i\le d, \, \sigma(i)\ne l} k_{\sigma(i)} ^{2/q_{\sigma(i)}}$, then inequality \eqref{nu1lnu_pth} can be written as follows:
$$
\nu_{(1)}^{1-\tilde \lambda}\nu_{(2)}^{\tilde \lambda} \prod _{j=1}^{\nu_1} k_{\sigma(j)} ^{1/q_{\sigma(j)} - 1/\theta _{\sigma(j)}(\overline{\alpha})} n^{-1/2} \prod _{i=1}^{\nu_1} k_{\sigma(i)} ^{1/2} \prod _{i=\nu_1+1}^d k_{\sigma(i)} ^{1/q_{\sigma(i)}} \le
$$
$$
\le \nu_{(1)} \prod _{j=1}^{\nu_1} k_{\sigma(j)} ^{1/q_{\sigma(j)} - 1/\theta _{(1),\sigma(j)}} k_l^{1/q_l - 1/\theta_{(1), l}} n^{-1/2} \prod _{i=1}^{\nu_1} k_{\sigma(i)} ^{1/2} \cdot k_l^{1/2}\prod _{\nu_1+1\le i\le d,\, \sigma(i)\ne l} k_{\sigma(i)} ^{1/q_{\sigma(i)}};
$$
inequality \eqref{nu2lnu_pth} is the same as in the previous case. This yields \eqref{est_predv5} and the inequality
$$
\frac{\nu_{(1)}}{\nu_{(2)}} \ge \prod _{i=1}^{\nu_1} k_{\sigma(i)}^{1/\theta_{(1),\sigma(i)}-1/\theta_{(2),\sigma(i)}} \cdot k_l^{1/\theta_{(1),l}-1/\theta_{(2),l}};
$$
hence, by \eqref{s_def3}, \eqref{nu1_nu2}, \eqref{segm_def_delta1}, \eqref{th1li2}, we get
\begin{align}
\label{est_predv6} s_l \le k_l\prod _{\sigma(i)\in I_\omega} \left(\frac{k_{\sigma(i)}}{s_{\sigma(i)}}\right) ^{\frac{1/\theta_{(2),\sigma(i)} - 1/\theta_{(1),\sigma(i)}}{1/\theta_{(2),l}-1/\theta_{(1),l}}}.
\end{align}

Thus, we obtained
\begin{Sta}
\label{i2_sta} Let \eqref{psi_eq_nu_phi} hold. Then
\begin{enumerate}
\item If $n\le \prod _{i=1}^{\nu_1} k_{\sigma(i)} \prod _{i=\nu_1+1}^d k_{\sigma(i)} ^{2/q_{\sigma(i)}}$, we have $I_2=\varnothing$.

\item If $\prod _{i=1}^{\nu_1} k_{\sigma(i)} \prod _{i=\nu_1+1}^d k_{\sigma(i)} ^{2/q_{\sigma(i)}}< n \le \prod _{i=1}^{\nu_1} k_{\sigma(i)}\cdot k_l \prod _{\nu_1+1\le i\le d, \, \sigma(i)\ne l} k_{\sigma(i)} ^{2/q_{\sigma(i)}}$, $l\in I_2$, then \eqref{est_predv4}, \eqref{est_predv5} hold.

\item If $n>\prod _{i=1}^{\nu_1} k_{\sigma(i)}\cdot k_l \prod _{\nu_1+1\le i\le d, \, \sigma(i)\ne l} k_{\sigma(i)} ^{2/q_{\sigma(i)}}$, $l\in I_2$, then \eqref{est_predv5}, \eqref{est_predv6} hold.
\end{enumerate}
\end{Sta}

{\bf The case $l\in I_{\omega_r}$, $1\le r\le L$.} Then $\sigma^{-1}(l) \in \{t_r, \, t_{r+1}, \, \dots, \, t_r'\}$ (see \eqref{i_om_j_tj}). Consider the set of points $(1/p_1, \, \dots, \, 1/p_d)\in \Delta$ defined by the conditions
\begin{align}
\label{segm_def_delta2}
\begin{array}{c}
p_{\sigma(i)} = q_{\sigma(i)}, \; \mu_1+1\le i\le \mu_2; \\
p_{\sigma(i)} = 2, \; \nu_1+1\le i\le \nu_2; \\
\omega' _{p_{\sigma(t_s)}, q_{\sigma(t_s)}} = \omega' _{p_{\sigma(t_s+1)}, q_{\sigma(t_s+1)}} = \dots =\omega' _{p_{\sigma(t'_s)}, q_{\sigma(t'_s)}}, \; 1\le s\le L, \; s\ne r; \\ \omega'_{p_{\sigma(i)},q_{\sigma(i)}} = \omega'_{p_{\sigma(i')},q_{\sigma(i')}}, \quad i, \; i' \in \{t_r, \, t_{r+1}, \, \dots, \, t_r'\} \backslash \{\sigma^{-1}(l)\},
\end{array}
\end{align}
where $\omega'_{p,q}$ is from formula \eqref{om_prime}.

As in previous two cases, by generality of position, this set is a segment $[1/\overline{\rho}_{(1)}, \, 1/\overline{\rho}_{(2)}]$ with interior point $1/\overline{\theta}(\overline{\alpha})$. The number $\lambda \in (0, \, 1)$ is defined by \eqref{lam_def1}.

Let
$$
D = \{(1/p_1, \, \dots, \, 1/p_d):\; p_{\sigma(i)}> q_{\sigma(i)}, \; i\in \overline{1, \, \mu_1},
$$
$$
2<p_{\sigma(i)} < q_{\sigma(i)}, \; i\in \overline{\mu_2+1, \, \nu_1}, \;p_{\sigma(i)}<2, \; i\in \overline{\nu_2+1, \, d},
$$
$$
\omega_{p_{\sigma(t_r-1)},q_{\sigma(t_r-1)}}< \omega_{p_i,q_i}  < \omega _{p_{\sigma(t'_r+1)},q_{\sigma(t'_r+1)}}, \; i\in I_{\omega_r},
$$
$$
\omega_{p_{\sigma(\mu_2+1)}, q_{\sigma(\mu_2+1)}}
< \dots < \omega _{p_{\sigma(t_1-1)},q_{\sigma(t_1-1)}} < \omega _{p_{\sigma(t_1)},q_{\sigma(t_1)}}<$$$$<
\omega _{p_{\sigma(t'_1+1)},q_{\sigma(t'_1+1)}}< \dots < \omega _{p_{\sigma(t_2-1)},q_{\sigma(t_2-1)}}<\omega _{p_{\sigma(t_2)},q_{\sigma(t_2)}} < \omega _{p_{\sigma(t'_2+1)},q_{\sigma(t'_2+1)}} < $$$$<\dots < 
\omega _{p_{\sigma(t_r-1)},q_{\sigma(t_r-1)}}< \omega _{p_{\sigma(t'_r+1)},q_{\sigma(t'_r+1)}}<\dots<$$
$$
<\omega_{p_{\sigma(t_{r+1}-1)},q_{\sigma(t_{r+1}-1)}}<\omega_{p_{\sigma(t_{r+1})},q_{\sigma(t_{r+1})}} <\omega_{p_{\sigma(t'_{r+1}+1)},q_{\sigma(t'_{r+1}+1)}} < \dots
$$
$$
 < \omega _{p_{\sigma(t'_L+1)},q_{\sigma(t'_L+1)}}< \omega _{p_{\sigma(t'_L+2)},q_{\sigma(t'_L+2)}} < \dots < \omega _{p_{\sigma(\nu_2)}, q_{\sigma(\nu_2)}}\}.$$
Then the set $D$ is open in $[0, \, 1]^d$ and $\overline{\theta}(\overline{\alpha})\in D$.

We define the numbers $\tau_*$, $\tau_{**}$ by formulas \eqref{tau_st}, \eqref{tau_stst}. Then $0\le \tau_*<\lambda <\tau_{**}\le 1$. The vectors $\overline{\theta}_{(1)}$, $\overline{\theta}_{(2)}$ are defined by formula \eqref{th1th2_def}, the number $\tilde \lambda \in (0, \, 1)$ is defined by \eqref{1th11l2l}, and the numbers, $\nu_{(1)}$, $\nu_{(2)}$, by formulas \eqref{theta_mukj}, \eqref{nu12_def}.

As in previous two cases, we get \eqref{nu_a_12}, \eqref{nu1lnu_pth} and \eqref{nu2lnu_pth}. 

If $n\le \prod _{i=1}^{\mu_2}k_{\sigma(i)} \prod _{i=\mu_2+1}^d k_{\sigma(i)}^{2/q_{\sigma(i)}}$, inequalities \eqref{nu1lnu_pth} and \eqref{nu2lnu_pth} can be written as \eqref{nunlmu1} and \eqref{nunlmu2}, respectively. If $n> \prod _{i=1}^{\nu_1} k_{\sigma(i)} \prod _{i=\nu_1+1}^d k_{\sigma(i)} ^{2/q_{\sigma(i)}}$, as in the case $l\in I_q$, we get \eqref{nu_t_eq1118}, \eqref{nu_t_eq2228}. If $t\in \{\mu_2+1, \, \dots, \, t_r-1\} \cup \{t'_r+1, \, \dots, \, \nu_1\}$, $\prod _{i=1}^{t-1} k_{\sigma(i)} \prod _{i=t}^d k_{\sigma(i)} ^{2/q_{\sigma(i)}} < n\le \prod _{i=1}^t k_{\sigma(i)} \prod _{i=t+1}^d k_{\sigma(i)} ^{2/q_{\sigma(j)}}$, inequalities \eqref{nu1lnu_pth} and \eqref{nu2lnu_pth} can be written as \eqref{nu_t_eq111} and \eqref{nu_t_eq222}, respectively. In all these cases, \eqref{phi12_eq} holds; this contradicts to generality of position.

Thus, the following assertion holds:
\begin{Sta}
\label{iomega_sta} Let \eqref{psi_eq_nu_phi} hold. Then,
for $n\le \prod _{i=1}^{t_r-1} k_{\sigma(i)}\prod _{i=t_r}^d k_{\sigma(i)} ^{2/q_{\sigma(i)}}$ or $n> \prod _{i=1}^{t'_r} k_{\sigma(i)} \prod _{i=t'_r+1}^d k_{\sigma(i)} ^{2/q_{\sigma(i)}}$, we have $I_{\omega_r}=\varnothing$.
\end{Sta}

Comparing Propositions \ref{iq_sta}, \ref{i2_sta} and \ref{iomega_sta}, we get:

\begin{enumerate}
\item If $I_q \ne \varnothing$, then $I_\omega \sqcup I_2 = \varnothing$.

\item If $I_2 \ne \varnothing$, then $I_q \sqcup I_\omega = \varnothing$.

\item If $I_{\omega_r} \ne \varnothing$, then $I_q \sqcup (I_\omega \backslash I_{\omega_r}) \sqcup I_2 = \varnothing$ (therefore, $r=1$, $I_\omega = I_{\omega_1}$).

\item Let $I=I_q$, $l\in I_q$. Then, if $n\le \prod _{1\le i\le \mu_2, \, \sigma(i)\ne l} k_{\sigma(i)} \cdot k_l ^{2/q_l} \prod _{i=\mu_2+1}^d k_{\sigma(i)} ^{2/q_{\sigma(i)}}$, we get $1\le s_l\le k_l$; if $\prod _{1\le i\le \mu_2, \, \sigma(i)\ne l} k_{\sigma(i)} \cdot k_l ^{2/q_l} \prod _{i=\mu_2+1}^d k_{\sigma(i)} ^{2/q_{\sigma(i)}}< n \le \prod _{i=1}^{\mu_2} k_{\sigma(i)} \prod _{i=\mu_2+1}^d k_{\sigma(i)} ^{2/q_{\sigma(i)}}$, we get 
\begin{align}
\label{s_l_est_n12kl}
1\le \Bigl(n^{1/2} \prod _{1\le i\le \mu_2, \, \sigma(i)\ne l} k_{\sigma(i)}^{-1/2} \cdot k_l ^{-1/q_l} \prod _{i=\mu_2+1}^d k_{\sigma(i)} ^{-1/q_{\sigma(i)}}\Bigr)^{\frac{1}{1/2-1/q_l}}\le s_l\le k_l.
\end{align}

\item Let $I=I_2$, $l\in I_2$. Then, if $n>\prod _{i=1}^{\nu_1} k_{\sigma(i)}\cdot k_l \prod _{\nu_1+1\le i\le d, \, \sigma(i)\ne l} k_{\sigma(i)} ^{2/q_{\sigma(i)}}$, we get $1\le s_l\le k_l$; if $\prod _{i=1}^{\nu_1} k_{\sigma(i)} \prod _{i=\nu_1+1}^d k_{\sigma(i)} ^{2/q_{\sigma(i)}}< n \le \prod _{i=1}^{\nu_1} k_{\sigma(i)}\cdot k_l \prod _{\nu_1+1\le i\le d, \, \sigma(i)\ne l} k_{\sigma(i)} ^{2/q_{\sigma(i)}}$, we get
\begin{align}
\label{s_l_est_n121}
1\le s_l \le \Bigl(n^{1/2} \prod _{i=1}^{\nu_1} k_{\sigma(i)}^{-1/2} \prod _{i=\nu_1+1}^d k_{\sigma(i)} ^{-1/q_{\sigma(i)}}\Bigr)^{\frac{1}{1/2-1/q_l}} \le k_l.
\end{align}
\end{enumerate}

Hence the inequalities $1\le s_l\le k_l$ ($l\in I$) are proved in the cases $I= I_q$ and $I=I_2$.

It remains to consider the case $I=I_\omega$. Then $I_\omega = I_{\omega_1}$, $\# I_{\omega} \stackrel{\eqref{card_i_zm}}{=} m$, $t_1'=t_1+m-1$ (see \eqref{i_om_j_tj}),
\begin{align}
\label{n_gr} \prod _{i=1}^{t_1-1} k_{\sigma(i)} \prod _{i=t_1}^d k_{\sigma(i)}^{2/q_{\sigma(i)}} <n\le \prod _{i=1}^{t_1+m-1} k_{\sigma(i)} \prod _{i=t_1+m}^d k_{\sigma(i)}^{2/q_{\sigma(i)}}
\end{align}
(see Proposition \ref{iomega_sta}).

First we define $s_i$ for $i\in I_\omega$. To this end, we need some auxiliary constructions and assertions.

Recall that the numbers $\mu$, $\nu$ are defined by \eqref{mu_nu_def}.

Let
$$
D=\{(1/p_1, \, \dots, \, 1/p_d)\in \Delta:\; p_{\sigma(i)} \ge q_{\sigma(i)}, \; i\le \mu;
$$
$$
p_{\sigma(i)} \in [2, \, q_{\sigma(i)}], \; \mu+1\le i \le \nu; \; p_{\sigma(i)}\le 2, \; \nu+1\le i\le d;
$$
$$
\omega_{p_{\sigma(\mu+1)},q_{\sigma(\mu+1)}}\le \omega_{p_{\sigma(\mu+2)},q_{\sigma(\mu+2)}}\le \dots \le \omega_{p_{\sigma(t_1-1)},q_{\sigma(t_1-1)}}\le \omega_{p_{\sigma(t_1+m)},q_{\sigma(t_1+m)}} \le
$$
$$
\le \omega_{p_{\sigma(t_1+m+1)},q_{\sigma(t_1+m+1)}} \le \dots \le \omega_{p_{\sigma(\nu)},q_{\sigma(\nu)}},
$$
$$
\omega_{p_{\sigma(t_1-1)},q_{\sigma(t_1-1)}} \le \omega_{p_{\sigma(i)},q_{\sigma(i)}} \le \omega_{p_{\sigma(t_1+m)},q_{\sigma(t_1+m)}}, \; t_1\le i\le t_1+m-1\}.
$$
Then $1/\overline{\theta}(\overline{\alpha})$ belongs to relative interior of $D$.

Let $\{\pi_s\}_{s=1}^{s_*}$ be the set of all permutations of elements of $I_\omega$. Let
\begin{align}
\label{dsdef}
D_s = \{(1/p_1, \, \dots, \, 1/p_d)\in D:\; \omega_{p_{\pi_s(\sigma(t_1))},q_{\pi_s(\sigma(t_1))}} \le \dots \le \omega_{p_{\pi_s(\sigma(t_1+m-1))}, q_{\pi_s(\sigma(t_1+m-1))}}\}.
\end{align}
We obtain the partition of $D$ into convex polyhedra $D_s$, $1\le s\le s_*$. Notice that $\cap _{s=1}^{s_*} D_s=\{1/\overline{\theta}(\overline{\alpha})\}$ by the equality $\omega_{\theta_{\sigma(t_1)}(\overline{\alpha}), q_{\sigma(t_1)}} = \dots = \omega_{\theta_{\sigma(t_1+m-1)}(\overline{\alpha}), q_{\sigma(t_1+m-1)}}$ and generality of position (see assertion 2 of Definition \ref{gen_pos}).

Let
\begin{align}
\label{t0_pm_def}
T_+^0 = \{\sigma(i):\; 1\le i\le t_1-1\}, \quad T_-^0 = \{\sigma(i):\; t_1+m\le i\le d\}.
\end{align}
By \eqref{phi_case10}, \eqref{phi_case20}, \eqref{phi_case30}, \eqref{n_gr}, for each $s\in \{1, \, \dots, \, s_*\}$ there is a partition
\begin{align}
\label{i_razb_s} I=\{i_*(s)\} \sqcup I_+(s) \sqcup I_-(s)
\end{align}
such that for any $1/\overline{p}\in D_s$ the equality
\begin{align}
\label{phi_d_s} \Phi(\overline{p}, \, \overline{q}, \, \overline{k}, \, n) = \prod _{i\in T_0^+} k_i^{1/q_i-1/p_i} \prod _{i\in I} A_i(s)^{1/q_i-1/p_i}
\end{align}
holds, where
\begin{align}
\label{ais_def} A_i(s) = \begin{cases} k_i, & i\in I_+(s), \\ 1, & i\in I_-(s), \\ \Bigl(n^{1/2} \prod _{i\in T^+_0\sqcup I_+(s)} k_i^{-1/2} \prod _{i\in \{i_*(s)\} \sqcup T_-^0 \sqcup I_-(s)} k_i^{-1/q_i}\Bigr) ^{\frac{1}{1/2-1/q_{i_*(s)}}}, & i = i_*(s),\end{cases}
\end{align}
and
\begin{align}
\label{n_interval} \prod _{i\in T^+_0\sqcup I_+(s)} k_i \prod _{i\in \{i_*(s)\} \sqcup T_-^0 \sqcup I_-(s)} k_i^{2/q_i} < n \le \prod _{i\in T^+_0\sqcup I_+(s) \sqcup \{i_*(s)\} } k_i \prod _{i\in T_-^0 \sqcup I_-(s)} k_i^{2/q_i}.
\end{align}

We set
\begin{align}
\label{small_phi} \varphi(\overline{x}) = \log \Phi(1/\overline{x}, \, \overline{q}, \, \overline{k}, \, n), \quad \overline{x}\in D.
\end{align}
By \eqref{phi_d_s}, the restriction of $\varphi$ on the polyhedron $D_s$ is the affine function
\begin{align}
\label{phi_s_d_s}
\varphi_s(x_1, \, \dots, \, x_d) = \sum \limits _{i\in T^0_+} (1/q_i-x_i) \log k_i + \sum \limits _{i\in I} (1/q_i-x_i) \log A_i(s);
\end{align}
since the sets $D_s$ are closed, the function $\varphi$ is continuous on $D$.

\begin{Lem}
\label{conv_phi} The equality $\varphi = \max _{1\le s\le s_*}\varphi_s$ holds.
\end{Lem}
\begin{proof}
First we check that if the polyhedra $D_s$ and $D_{s'}$ ($s \ne s'$) intersect along a $m-2$-dimensional face, then $\varphi_s|_{D_s} \ge \varphi_{s'}|_{D_s}$.

From the definition of the sets $D_s$ and generality of position we get that there is a permutation $\pi \in S_d$ and a number $t\in \{t_1, \, \dots, \, t_1+m-2\}$ such that, for $(1/p_1, \, \dots, \, 1/p_d)$ which belongs to relative interior of $m-2$-dimensional face $D_s\cap D_{s'}$,
$$
\omega _{p_{\pi(1)},q_{\pi(1)}} = \dots = \omega _{p_{\pi(\mu)},q_{\pi(\mu)}} = 0 < \omega _{p_{\pi(\mu+1)},q_{\pi(\mu+1)}}< \dots < \omega _{p_{\pi(t-1)},q_{\pi(t-1)}}<
$$
$$
< \omega _{p_{\pi(t)},q_{\pi(t)}} = \omega _{p_{\pi(t+1)},q_{\pi(t+1)}}< \omega _{p_{\pi(t+2)},q_{\pi(t+2)}}< \dots < \omega _{p_{\pi(\nu)},q_{\pi(\nu)}} < 
$$
$$
<1 = \omega _{p_{\pi(\nu+1)},q_{\pi(\nu+1)}}= \dots = \omega _{p_{\pi(d)},q_{\pi(d)}}
$$
holds; in the interior of $D_s$ the equality $\omega _{p_{\pi(t)},q_{\pi(t)}} = \omega _{p_{\pi(t+1)},q_{\pi(t+1)}}$ is replaced by $\omega _{p_{\pi(t)},q_{\pi(t)}} < \omega _{p_{\pi(t+1)},q_{\pi(t+1)}}$, and in the interior of $D_{s'}$, by $\omega _{p_{\pi(t)},q_{\pi(t)}} > \omega _{p_{\pi(t+1)},q_{\pi(t+1)}}$.

By \eqref{phi_case10}, \eqref{phi_case20} and \eqref{phi_case30}, for $n\le \prod _{i=1}^{t-1} k_{\pi(i)} \prod _{i=t}^d k_{\pi(i)} ^{2/q_{\pi(i)}}$ and for $n> \prod _{i=1}^{t+1} k_{\pi(i)} \prod _{i=t+2}^d k_{\pi(i)} ^{2/q_{\pi(i)}}$ the function $\Phi(\overline{p}, \, \overline{q}, \, \overline{k}, \, n)$ is given by the same formula for $1/\overline{p}\in D_s$ and for $1/\overline{p}\in D_{s'}$; therefore, $\varphi_s$ and $\varphi_{s'}$ are the same.

Let
\begin{align}
\label{n_int1}
\prod _{i=1}^{t-1} k_{\pi(i)} \prod _{i=t}^d k_{\pi(i)} ^{2/q_{\pi(i)}} < n \le \prod _{i=1}^{t-1} k_{\pi(i)} \cdot \min \{k_{\pi(t)} k_{\pi(t+1)}^{2/q_{\pi(t+1)}}, \, k_{\pi(t+1)} k_{\pi(t)}^{2/q_{\pi(t)}}\}\prod _{i=t+2}^d k_{\pi(i)} ^{2/q_{\pi(i)}}.
\end{align}
Then, for $1/\overline{p}\in D_s$, we have
\begin{align}
\label{phi_ds1} \Phi(\overline{p}, \, \overline{q}, \, \overline{k}, \, n) = \prod _{i=1}^{t-1} k_{\pi(i)} ^{1/q_{\pi(i)}-1/p_{\pi(i)}} \Bigl(n^{1/2}\prod _{i=1}^{t-1} k_{\pi(i)}^{-1/2} \prod _{i=t}^d k_{\pi(i)} ^{-1/q_{\pi(i)}}\Bigr) ^{\frac{1/q_{\pi(t)}-1/p_{\pi(t)}}{1/2 - 1/q_{\pi(t)}}}=:E_1,
\end{align}
and for $1/\overline{p}\in D_{s'}$,
\begin{align}
\label{phi_ds2} \Phi(\overline{p}, \, \overline{q}, \, \overline{k}, \, n) = \prod _{i=1}^{t-1} k_{\pi(i)} ^{1/q_{\pi(i)}-1/p_{\pi(i)}} \Bigl(n^{1/2}\prod _{i=1}^{t-1} k_{\pi(i)}^{-1/2} \prod _{i=t}^d k_{\pi(i)} ^{-1/q_{\pi(i)}}\Bigr) ^{\frac{1/q_{\pi(t+1)}-1/p_{\pi(t+1)}}{1/2 - 1/q_{\pi(t+1)}}}=:E_2.
\end{align}
Compare the values of the right-hand sides of \eqref{phi_ds1} and \eqref{phi_ds2} for $1/\overline{p}\in D_s$. By \eqref{n_int1} and the inequality $\omega _{p_{\pi(t)},q_{\pi(t)}} \le \omega _{p_{\pi(t+1)},q_{\pi(t+1)}}$, we get that $E_1\ge E_2$ on $D_s$. Hence, $\varphi_s|_{D_s}\ge \varphi _{s'}|_{D_s}$.

Let
$$
\prod _{i=1}^{t-1} k_{\pi(i)} \cdot k_{\pi(t)} k_{\pi(t+1)}^{2/q_{\pi(t+1)}} \prod _{i=t+2}^d k_{\pi(i)} ^{2/q_{\pi(i)}} < n \le \prod _{i=1}^{t-1} k_{\pi(i)} \cdot k_{\pi(t+1)} k_{\pi(t)}^{2/q_{\pi(t)}} \prod _{i=t+2}^d k_{\pi(i)} ^{2/q_{\pi(i)}}.
$$
Then, for $1/\overline{p}\in D_s$,
\begin{align}
\label{phi_ds3} \begin{array}{c}\Phi(\overline{p}, \, \overline{q}, \, \overline{k}, \, n) = \prod _{i=1}^{t-1} k_{\pi(i)} ^{1/q_{\pi(i)}-1/p_{\pi(i)}} \cdot k_{\pi(t)}^{1/q_{\pi(t)} - 1/p_{\pi(t)}}\times \\ \times \Bigl(n^{1/2}\prod _{i=1}^{t-1} k_{\pi(i)}^{-1/2} \cdot k_{\pi(t)}^{-1/2} k_{\pi(t+1)}^{-1/q_{\pi(t+1)}}\prod _{i=t+2}^d k_{\pi(i)} ^{-1/q_{\pi(i)}}\Bigr) ^{\frac{1/q_{\pi(t+1)}-1/p_{\pi(t+1)}}{1/2 - 1/q_{\pi(t+1)}}}=:E_3,\end{array}
\end{align}
and for $1/\overline{p}\in D_{s'}$ equality \eqref{phi_ds2} holds. 

Compare the values of the right-hand sides of \eqref{phi_ds2} and \eqref{phi_ds3} for $1/\overline{p}\in D_s$. Notice that
$$
E_3 = E_2\cdot k_{\pi(t)} ^{1/q_{\pi(t)} - 1/p_{\pi(t)}-\frac{1/2-1/q_{\pi(t)}}{1/2-1/q_{\pi(t+1)}}\cdot (1/q_{\pi(t+1)} - 1/p_{\pi(t+1)})}.
$$
By the inequality $\omega _{p_{\pi(t)},q_{\pi(t)}} \le \omega _{p_{\pi(t+1)},q_{\pi(t+1)}}$ on $D_s$, we have $E_3\ge E_2$; again we obtain that $\varphi_s|_{D_s}\ge \varphi _{s'}|_{D_s}$.

The case
$$
\prod _{i=1}^{t-1} k_{\pi(i)} \cdot k_{\pi(t+1)} k_{\pi(t)}^{2/q_{\pi(t)}} \prod _{i=t+2}^d k_{\pi(i)} ^{2/q_{\pi(i)}} < n \le \prod _{i=1}^{t-1} k_{\pi(i)} \cdot k_{\pi(t)} k_{\pi(t+1)}^{2/q_{\pi(t+1)}} \prod _{i=t+2}^d k_{\pi(i)} ^{2/q_{\pi(i)}}
$$
is similar to the previous one.

Let
\begin{align}
\label{n_int2} 
\prod _{i=1}^{t-1} k_{\pi(i)} \cdot \max \{k_{\pi(t)} k_{\pi(t+1)}^{2/q_{\pi(t+1)}}, \, k_{\pi(t+1)} k_{\pi(t)}^{2/q_{\pi(t)}}\} \prod _{i=t+2}^d k_{\pi(i)} ^{2/q_{\pi(i)}} < n \le \prod _{i=1}^{t+1} k_{\pi(i)} \prod _{i=t+2}^d k_{\pi(i)} ^{2/q_{\pi(i)}}.
\end{align}
Then, for $1/\overline{p}\in D_s$, we have \eqref{phi_ds3}, and for $1/\overline{p}\in D_{s'}$ we get
\begin{align}
\label{phi_ds4} \begin{array}{c}\Phi(\overline{p}, \, \overline{q}, \, \overline{k}, \, n) = \prod _{i=1}^{t-1} k_{\pi(i)} ^{1/q_{\pi(i)}-1/p_{\pi(i)}} \cdot k_{\pi(t+1)}^{1/q_{\pi(t+1)} - 1/p_{\pi(t+1)}}\times \\ \times \Bigl(n^{1/2}\prod _{i=1}^{t-1} k_{\pi(i)}^{-1/2} \cdot k_{\pi(t+1)}^{-1/2} k_{\pi(t)}^{-1/q_{\pi(t)}}\prod _{i=t+2}^d k_{\pi(i)} ^{-1/q_{\pi(i)}}\Bigr) ^{\frac{1/q_{\pi(t)}-1/p_{\pi(t)}}{1/2 - 1/q_{\pi(t)}}}=:E_4.\end{array}
\end{align}
Compare the right-hand sides of \eqref{phi_ds3} and \eqref{phi_ds4} for $1/\overline{p}\in D_s$. Notice that
$$
E_3 = \prod _{i=1}^{t+1} k_{\pi(i)} ^{1/q_{\pi(i)}-1/p_{\pi(i)}} \Bigl(n^{1/2}\prod _{i=1}^{t+1} k_{\pi(i)}^{-1/2} \prod _{i=t+2}^d k_{\pi(i)} ^{-1/q_{\pi(i)}}\Bigr) ^{\frac{1/q_{\pi(t+1)}-1/p_{\pi(t+1)}}{1/2 - 1/q_{\pi(t+1)}}},
$$
$$
E_4 = \prod _{i=1}^{t+1} k_{\pi(i)} ^{1/q_{\pi(i)}-1/p_{\pi(i)}} \Bigl(n^{1/2}\prod _{i=1}^{t+1} k_{\pi(i)}^{-1/2} \prod _{i=t+2}^d k_{\pi(i)} ^{-1/q_{\pi(i)}}\Bigr) ^{\frac{1/q_{\pi(t)}-1/p_{\pi(t)}}{1/2 - 1/q_{\pi(t)}}}.
$$
Taking into account \eqref{n_int2}, we get $E_3\ge E_4$ and $\varphi_s|_{D_s}\ge \varphi _{s'}|_{D_s}$.

Now we prove that $\varphi = \max _{1\le s\le s_*} \varphi_s$. Suppose that, for some $s\in \{1, \, \dots, \, s_*\}$, there is a point $\overline{a}\in D_s$ and a number $s'\in \{1, \, \dots, \, s_*\} \backslash \{s\}$ such that $\varphi_s(\overline{a})< \varphi_{s'}(\overline{a})$. We may assume that $\overline{a}$ lies in the interior of $D_s$. Let $\overline{b}$ be an interior point of $D_{s'}$. We set $[\overline{a}(u), \, \overline{b}(u)] = [\overline{a}, \, \overline{b}]\cap D_u$, $1\le u\le s_*$. Making a sufficiently small translation of the points $\overline{a}$ and $\overline{b}$, we may assume that, for all $u$ such that $[\overline{a}(u), \, \overline{b}(u)]\ne \varnothing$, the points $\overline{a}(u)$ and $\overline{b}(u)$ belong to relative interior of an $m-2$-dimensional face of $D_u$ (if $\overline{a}(u)\ne \overline{a}$, $\overline{b}(u) \ne \overline{b}$), and the segment $[\overline{a}, \, \overline{b}]$ intersects with each of these faces at a unique point. Thus, the segment $[\overline{a}, \, \overline{b}]$ is made up of the successive segments $[\overline{a}_j, \, \overline{b}_j] = [\overline{a}, \, \overline{b}] \cap D_{s(j)}$, $1\le j\le j_*$, $s(1)=s$, $s(j_*)=s'$, and for all $j=1, \, \dots, \, j_*-1$ the point $\overline{b}_j=\overline{a}_{j+1}$ belongs to relative interior of an $m-2$-dimensional face of $D_{s(j)}$ and $D_{s(j+1)}$. We obtained a piecewise affine continuous function $\varphi|_{[\overline{a}, \, \overline{b}]}$ such that $\varphi|_{[\overline{a}_j, \, \overline{b}_j]}= \varphi_{s(j)}|_{[\overline{a}_j, \, \overline{b}_j]}$, and $\varphi_{s(j)}|_{[\overline{a}_j, \, \overline{b}_j]} \ge \varphi_{s(j\pm 1)}|_{[\overline{a}_j, \, \overline{b}_j]}$. This implies that $\varphi_{s'}(\overline{a})\le \varphi_{s}(\overline{a})$. We arrived to a contradiction.
\end{proof}

We set for $1\le s\le s_*$
$$
\Lambda_s = \Bigl\{ (\lambda_2, \, \dots, \, \lambda_m)\in \R^{m-1}:\; \Bigl(1-\sum \limits _{j=2}^m \lambda_j\Bigr)/ \overline{p}_{\alpha_1} + \sum \limits _{j=2}^m \lambda_j/\overline{p}_{\alpha_j}\in D_s\Bigr\},
$$
$\Lambda = \cup _{s=1}^{s_*}\Lambda_s$,
\begin{align}
\label{psi_s_sm} \begin{array}{c}\psi_s(\lambda_2, \, \dots, \, \lambda_m) = \sum \limits _{j=2}^m \lambda_j \log \nu_{\alpha_j} + \Bigl( 1- \sum \limits _{j=2}^m \lambda_j\Bigr) \log \nu_{\alpha_1} + \varphi_s\Bigl( \Bigl(1-\sum \limits _{j=2}^m \lambda_j\Bigr)/ \overline{p}_{\alpha_1} + \sum \limits _{j=2}^m \lambda_j/\overline{p}_{\alpha_j}\Bigr), \\
\psi(\lambda_2, \, \dots, \, \lambda_m) =\max _{1\le s\le s_*} \psi_s(\lambda_2, \, \dots, \, \lambda_m), \quad (\lambda_2, \, \dots, \, \lambda_m) \in \Lambda. \end{array}
\end{align}
Then $(\lambda_2(\overline{\alpha}), \, \dots, \, \lambda_m(\overline{\alpha}))$ is the interior point of $\Lambda$; in addition, it is the vertex of $\Lambda_s$ for each $s\in \{1, \, \dots, \, s_*\}$ by generality of position.

\begin{Lem}
\label{str_min_lam} The point $(\lambda_2(\overline{\alpha}), \, \dots, \, \lambda_m(\overline{\alpha}))$ is a strict minimum point of the function $\psi$ on $\Lambda$.
\end{Lem}
\begin{proof}
It suffices to prove that, for each $s=1, \, \dots, \, s_*$, the function $\psi$ strictly increases on each edge of the polyhedron $\Lambda_s$ going out of $(\lambda_2(\overline{\alpha}), \, \dots, \, \lambda_m(\overline{\alpha}))$; i.e., the strict minimum on this edge attains at $(\lambda_2(\overline{\alpha}), \, \dots, \, \lambda_m(\overline{\alpha}))$.

Let $(\mu_2, \, \dots, \, \mu_m)$ be the endpoint of one of such edges, and let the vector $\overline{\rho}=(\rho_1, \, \dots, \, \rho_d)$ be given by the equation
\begin{align}
\label{lem_rho_def}
1/\overline{\rho} = \Bigl(1-\sum \limits _{j=2}^m \mu_j\Bigr)/ \overline{p}_{\alpha_1} + \sum \limits _{j=2}^m \mu_j/\overline{p}_{\alpha_j}.
\end{align}

Let $I =\{i_1, \, \dots, \, i_m\}$. Since $1/\overline{\theta}(\overline{\alpha})$ lies in relative interior of $D$, from the definition of $D_s$ (see \eqref{dsdef}) it follows that there is $t\in \{2, \, \dots, \, m\}$ such that $1/\overline{\rho}$ lies in the plane $L_t$ given by the equations
\begin{align}
\label{lambda_t_def} \begin{array}{c} \omega'_{1/x_{i_1},q_{i_1}} =\omega'_{1/x_{i_2},q_{i_2}} = \dots = \omega'_{1/x_{i_{t-1}},q_{i_{t-1}}}, \\ \omega'_{1/x_{i_t},q_{i_t}}=\omega'_{1/x_{i_{t+1}},q_{i_{t+1}}} = \dots =\omega'_{1/x_{i_{m}},q_{i_{m}}}. \end{array}
\end{align}
Then $[1/\overline{\theta}(\overline{\alpha}), \, 1/\overline{\rho}] = D\cap L_t$. From the definition of $D$ it follows that for $\overline{x}=(x_1, \, \dots, \, x_d) \in [1/\overline{\theta}(\overline{\alpha}), \, 1/\overline{\rho}]$ we have $\omega'_{1/x_{i_l},q_{i_l}} \in [0, \, 1]$, $1\le l\le m$.

If $1/\overline{\rho}$ lies at the boundary of $\Delta$, then by generality of position (see assertion 2 of Definition \ref{gen_pos}) $1/\overline{\rho}$ belongs to relative interior of an $m-2$-dimensional face of $\Delta$. If $1/\overline{\rho}$ is the relative interior point of $\Delta$, then, besides equations \eqref{lambda_t_def}, an additional equation $\rho_j=2$, $\rho_j=q_j$ or $\omega_{\rho_j,q_j} = \omega _{\rho_{j'},q_{j'}}$ holds (see the definition of the set $D$). In both cases, by generality of position (see assertion 2 of Definition \ref{gen_pos}) we get $\overline{\rho}\in \Theta$; in particular, $\omega'_{\rho_{i_{t-1}},q_{i_{t-1}}}\in (0, \, 1)$ for $t\ge 3$ (if the additional equation is not $\rho_{i_1} = q_{i_1}$ or $\rho_{i_{t-1}} = 2$), $\omega'_{\rho_{i_t},q_{i_t}}\in (0, \, 1)$ for $t\le m-1$ (if the additional equation is not $\rho_{i_t} = q_{i_t}$ or $\rho_{i_m} = 2$), and if the additional equation to \eqref{lambda_t_def} is $\omega_{\rho_j,q_j} = \omega _{\rho_{j'},q_{j'}}$, then $\omega_{\rho_j,q_j}\in (0, \, 1)$. From \eqref{psi_def}, \eqref{psi_eq_nu_phi} and assertion 4 of Definition \ref{gen_pos} it follows that
$$
\nu_{\alpha_1}^{1-\sum \limits _{j=2}^m \lambda_j(\overline{\alpha})} \nu_{\alpha_2}^{\lambda_2(\overline{\alpha})} \dots \nu_{\alpha_m}^{\lambda_m(\overline{\alpha})} \Phi (\overline{\theta}(\overline{\alpha}), \, \overline{q}, \, \overline{k}, \, n) < \nu_{\alpha_1}^{1-\sum \limits _{j=2}^m \mu_j} \nu_{\alpha_2}^{\mu_2} \dots \nu_{\alpha_m}^{\mu_m} \Phi (\overline{\rho}, \, \overline{q}, \, \overline{k}, \, n).
$$
Taking the logarithm of both parts in this inequality and applying \eqref{small_phi}, \eqref{psi_s_sm}, \eqref{lem_rho_def} and Lemma \ref{conv_phi}, we get that $\psi(\lambda_2(\overline{\alpha}), \, \dots, \, \lambda_m(\overline{\alpha})) < \psi(\mu_2, \, \dots, \, \mu_m)$.
\end{proof}

Since $1/\overline{\theta}(\overline{\alpha})\in \cap _{s=1}^{s_*} D_s$, $\varphi|_{D_s} = \varphi_s|_{D_s}$, we have by \eqref{psi_s_sm}
\begin{align}
\label{psi_j_eq_th} \psi_1(\lambda_2(\overline{\alpha}), \, \dots, \, \lambda_m(\overline{\alpha})) = \dots = \psi_{s_*}(\lambda_2(\overline{\alpha}), \, \dots, \, \lambda_m(\overline{\alpha})).
\end{align}

Applying Dubovitskii--Milyutin's theorem and taking into account that $(\lambda_2(\overline{\alpha}), \, \dots, \, \lambda_m(\overline{\alpha}))\in {\rm int}\, \Lambda$ is the strict minimum point of the function $\psi$ on the set $\Lambda$, we get
\begin{align}
\label{dub_mil}
0\in {\rm int}\, {\rm conv}\, \Bigl(\cup _{s=1}^{s_*} \{\nabla\psi_s\}\Bigr).
\end{align}
By Carath\'{e}odory's theorem, there is a set $S\subset \{1, \, \dots, \, s_*\}$ such that $\# S\le m$, the points $\{\nabla \psi_s\}_{s\in S}$ are affinely independent and 
\begin{align}
\label{och}
0\in {\rm conv}\, \Bigl(\cup _{s\in S} \{\nabla\psi_s\}\Bigr).
\end{align}
By \eqref{dub_mil}, the system $\{\nabla \psi_s\}_{s\in S}$ can be complemented to an affinely independent system of points $\{\nabla \psi_s\}_{s\in S\sqcup S'}$, where $S' \subset \{1, \, \dots, \, s_*\}$ and $\# S + \# S' = m$.

We set $\{s(1), \, \dots, \, s(m)\} = S\sqcup S'$, 
\begin{align}
\label{f_max_def}
f= \max _{1\le t\le m} \psi_{s(t)}.
\end{align}
By \eqref{psi_j_eq_th}, \eqref{och} and Dubovitskii--Milyutin's theorem,
\begin{align}
\label{och_min} f(\lambda_2(\overline{\alpha}), \, \dots, \, \lambda_m(\overline{\alpha})) = \min _\Lambda f.
\end{align}
By the construction, $\{\nabla \psi_{s(t)}\}_{t=1}^m$ are affinely independent; hence $\{\nabla \psi_{s(t)}-\nabla \psi_{s(1)}\}_{t=2}^m$ are linearly independent.

By \eqref{phi_s_d_s}, \eqref{psi_s_sm}, we have
$$
\nabla\psi_{s(t)} = \Bigl(\log \nu_{\alpha_j}-\log \nu_{\alpha_1}-\sum \limits_{i\in T_+^0}(1/p_{\alpha_j,i}-1/p_{\alpha_1,i})\log k_i -\sum \limits _{i\in I}(1/p_{\alpha_j,i}-1/p_{\alpha_1,i})\log A_i(s(t))\Bigr)_{j=2}^m.
$$
Thus, we proved
\begin{Sta}
\label{a_vec_prop} The system of vectors $\{\overline{v}_t\}_{t=2}^m$, where
$$
\overline{v}_t= \Bigl(\sum \limits _{i\in I}(1/p_{\alpha_j,i}-1/p_{\alpha_1,i})(\log A_i(s(t))-\log A_i(s(1)))\Bigr)_{j=2}^m,
$$
is linearly independent.
\end{Sta}

\begin{Cor}
\label{a_matr_prop} The matrix $\Bigl(\log A_i(s(t))-\log A_i(s(1))\Bigr)_{2\le t\le m, \, i\in I}$ has range $m-1$.
\end{Cor}

Let
\begin{align}
\label{a_ti_def} A_{t,i} = A_i(s(t)), \quad i\in I.
\end{align}
From \eqref{ais_def}, \eqref{n_interval} it follows that
\begin{align}
\label{aij_cond}
1\le A_{j,\sigma(i)}\le k_{\sigma(i)}, \quad t_1\le i\le t_1+m-1.
\end{align}

We set
\begin{align}
\label{s_def4} s_{\sigma(i)} = \prod _{j=1}^m A_{j,\sigma(i)} ^{\tau_j}, \quad t_1\le i\le t_1+m-1,
\end{align}
where $\tau_j$ are defined by the equations 
\begin{align}
\label{tau_def_om} 
\begin{array}{c}
\prod _{i=1}^{t_1-1} k_{\sigma(i)} ^{1/p_{\alpha_2,\sigma(i)} - 1/p_{\alpha_1,\sigma(i)}} \prod _{i=t_1}^{t_1+m-1} s_{\sigma(i)} ^{1/p_{\alpha_2,\sigma(i)} - 1/p_{\alpha_1,\sigma(i)}} = \frac{\nu_{\alpha_2}}{\nu_{\alpha_1}}, \\ \dots \\
\prod _{i=1}^{t_1-1} k_{\sigma(i)} ^{1/p_{\alpha_m,\sigma(i)} - 1/p_{\alpha_1,\sigma(i)}} \prod _{i=t_1}^{t_1+m-1} s_{\sigma(i)} ^{1/p_{\alpha_m,\sigma(i)} - 1/p_{\alpha_1,\sigma(i)}} = \frac{\nu_{\alpha_m}}{\nu_{\alpha_1}}, \\
\sum \limits _{j=1}^m \tau_j = 1
\end{array}
\end{align}
(see \eqref{s_def2}, \eqref{si_def_ini}, \eqref{n_gr}).

Substituting \eqref{s_def4}, excluding $\tau_1$, taking the logarithms, and applying Proposition \ref{a_vec_prop}, we find that the system has a unique solution.

We show that $0\le \tau_j \le 1$; then, by \eqref{aij_cond}, \eqref{s_def4} and the equality $\sum \limits _{j=1}^m \tau_j = 1$, we have $1\le s_{\sigma(i)}\le k_{\sigma(i)}$, $t_1\le i\le t_1+m-1$.

Let
$$
\tilde \Lambda_j=\{(\lambda_2, \, \dots, \, \lambda_m)\in \Lambda:\;  f(\lambda_2, \, \dots, \, \lambda_m) = \psi _{s(j)}(\lambda_2, \, \dots, \, \lambda_m)\}, \; 1\le j\le m,
$$
where the function $f$ is given by formula \eqref{f_max_def}. Then from \eqref{psi_j_eq_th} it follows that $$\overline{\lambda}_*:= (\lambda_2 (\overline{\alpha}), \, \dots, \, \lambda_m (\overline{\alpha})) \in \cap _{j=1}^m \tilde \Lambda_j.$$
The sets $\tilde \Lambda_j$ are given by the system of inequalities
$$
\psi_{s(j)}(\lambda_2, \, \dots, \, \lambda_m)\ge \psi_{s(t)}(\lambda_2, \, \dots, \, \lambda_m), \; t\in \{1, \, \dots, \, m\} \backslash\{j\}.
$$
By affine independence of $\{\nabla \psi_{s(t)}\}_{t=1}^m$, we get that this system of inequalities has the nondegenerate matrix; hence $\overline{\lambda}_*$ is the vertex of $\tilde \Lambda_j$.

Let $1\le l\le m$. We show that $0\le \tau_l\le 1$. We take some $j_*\ne l$ and consider the edge $R_l$ of the polyhedron $\tilde \Lambda_{j_*}$ defined by the conditions
$$
\psi_{s(1)}(\lambda_2, \, \dots, \, \lambda_m) =\dots = \psi _{s(l-1)}(\lambda_2, \, \dots, \, \lambda_m) = \psi _{s(l+1)}(\lambda_2, \, \dots, \, \lambda_m) =\dots =$$$$= \psi _{s(m)}(\lambda_2, \, \dots, \, \lambda_m).
$$
Then
\begin{align}
\label{l_st_rl} \overline{\lambda}_*\in R_l, \quad f|_{R_l} = \psi _{s(t)}|_{R_l} \; \forall t\ne l.
\end{align}
We draw a straight line through $R_l$ and intersect it with $\Lambda$. We get the segment $[\overline{\mu}_{(1)}, \, \overline{\mu}_{(2)}]$, $\overline{\mu}_{(k)} = (\mu_{(k),2}, \, \dots, \, \mu_{(k),m})$, $k=1,\, 2$; $\overline{\mu}_{(1)}$ is the end of $R_l$, $\overline{\lambda}_*$ lies in the interior of $[\overline{\mu}_{(1)}, \, \overline{\mu}_{(2)}]$.

We have 
\begin{align}
\label{flm2}
f|_{[\overline{\lambda}_*, \, \overline{\mu}_{(2)}]} = \psi_{s(l)}|_{[\overline{\lambda}_*, \, \overline{\mu}_{(2)}]}.
\end{align}
Indeed, otherwise $f|_{[\overline{\lambda}_*, \, \overline{\mu}_{(2)}]} = \psi_{s(t)}|_{[\overline{\lambda}_*, \, \overline{\mu}_{(2)}]}$ for some $t\ne l$. Then by \eqref{l_st_rl} we get $f|_{[\overline{\mu}_{(1)}, \, \overline{\mu}_{(2)}]} = \psi _{s(t)}|_{[\overline{\mu}_{(1)}, \, \overline{\mu}_{(2)}]}$; i.e., $\overline{\lambda}_*$ is not an extremal point of $\tilde \Lambda_t$ and we arrive to a contradiction.

We define the points $1/\overline{\theta}_{(k)}$ and the numbers $\nu_{(k)}$, $k=1, \, 2$, by equations
$$
1/\overline{\theta}_{(k)} =\Bigl(1 -\sum \limits _{j=2} ^m \mu_{(k),j}\Bigr)/\overline{p}_{\alpha_1} + \sum \limits _{j=2} ^m \mu_{(k),j}/\overline{p}_{\alpha_j},
$$
$$
\nu_{(k)} = \nu_{\alpha_1}^{1 -\sum \limits _{j=2} ^m \mu_{(k),j}} \nu_{\alpha_2} ^{\mu_{(k),2}} \dots \nu_{\alpha_m}^{\mu_{(k),m}}.
$$
Then, for some $\lambda\in (0, \, 1)$, we have $\frac{1}{\overline{\theta}(\overline{\alpha})} = \frac{1-\lambda}{\overline{\theta}_{(1)}} + \frac{\lambda}{\overline{\theta}_{(2)}}$. Hence, by the affine independence of $\{1/\overline{p}_{\alpha_j}\}_{j=1}^m$, the equality $\nu_1 ^{\lambda_1(\overline{\alpha})} \dots \nu_m ^{\lambda_m(\overline{\alpha})} = \nu_{(1)} ^{1-\lambda}\nu_{(2)}^\lambda$ holds. 

Let $j_*\in \{1, \, \dots, \, m\}\backslash \{l\}$. From the inequalities $f(\overline{\lambda}_*) \stackrel{\eqref{och_min}}{\le} f(\overline{\mu}_{(k)})$, $k=1, \, 2$, the equalities $f(\overline{\lambda}_*)=\psi_{s(j_*)}(\overline{\lambda}_*)= \psi_{s(l)}(\overline{\lambda}_*)$, $f(\overline{\mu}_{(1)}) \stackrel{\eqref{l_st_rl}}{=} \psi _{s(j_*)}(\overline{\mu}_{(1)})$, $f(\overline{\mu}_{(2)}) \stackrel{\eqref{flm2}}{=} \psi _{s(l)}(\overline{\mu}_{(2)})$ and \eqref{t0_pm_def}, \eqref{phi_s_d_s}, \eqref{psi_s_sm} we get
$$
\nu_{(1)}^{1-\lambda} \nu_{(2)}^\lambda \prod _{i=1}^{t_1-1} k_{\sigma(i)} ^{1/q_{\sigma(i)} - 1/\theta_{\sigma(i)} (\overline{\alpha})} \prod _{i=t_1} ^{t_1+m-1} A_{j_*,\sigma(i)}^{1/q_{\sigma(i)} - 1/\theta_{\sigma(i)}(\overline{\alpha})} \le
$$
$$
\le \nu_{(1)} \prod _{i=1}^{t_1-1} k_{\sigma(i)} ^{1/q_{\sigma(i)} - 1/\theta_{(1),\sigma(i)}}\prod _{i=t_1} ^{t_1+m-1} A_{j_*,\sigma(i)}^{1/q_{\sigma(i)} - 1/\theta_{(1),\sigma(i)}},
$$
$$
\nu_{(1)}^{1-\lambda} \nu_{(2)}^\lambda \prod _{i=1}^{t_1-1} k_{\sigma(i)} ^{1/q_{\sigma(i)} - 1/\theta_{\sigma(i)} (\overline{\alpha})} \prod _{i=t_1} ^{t_1+m-1} A_{l,\sigma(i)}^{1/q_{\sigma(i)} - 1/\theta_{\sigma(i)}(\overline{\alpha})} \le
$$
$$
\le \nu_{(2)} \prod _{i=1}^{t_1-1} k_{\sigma(i)} ^{1/q_{\sigma(i)} - 1/\theta_{(2),\sigma(i)}} \prod _{i=t_1} ^{t_1+m-1} A_{l, \sigma(i)} ^{1/q_{\sigma(i)} - 1/\theta_{(2),\sigma(i)}}.
$$
This is equivalent to
\begin{align}
\label{m1_ineq}
\frac{\nu_{(1)}}{\nu_{(2)}} \ge \prod _{i=1}^{t_1-1} k_{\sigma(i)} ^{1/\theta_{(1),\sigma(i)} -1/\theta_{(2), \sigma(i)}} \prod _{i=t_1}^{t_1+m-1} A_{j_*,\sigma(i)}^{1/\theta_{(1),\sigma(i)} - 1/\theta_{(2),\sigma(i)}}=:M_1,
\end{align}
\begin{align}
\label{m2_ineq}
\frac{\nu_{(1)}}{\nu_{(2)}} \le \prod _{i=1}^{t_1-1} k_{\sigma(i)} ^{1/\theta_{(1),\sigma(i)} -1/\theta_{(2), \sigma(i)}} \prod _{i=t_1}^{t_1+m-1} A_{l, \sigma(i)} ^{1/\theta_{(1),\sigma(i)} - 1/\theta_{(2),\sigma(i)}}=:M_2.
\end{align}
If $M_1=M_2$, then the inequalities turn to equalities; hence $f|_{[\overline{\mu}_{(1)}, \, \overline{\mu}_{(2)}]}={\rm const}$; in particular, $f|_{[\overline{\mu}_{(1)}, \, \overline{\mu}_{(2)}]}=\psi _{s(l)}|_{[\overline{\mu}_{(1)}, \, \overline{\mu}_{(2)}]}$, and $\overline{\lambda}_*$ is not an extremal point of $\tilde \Lambda_l$.

Thus, $M_1\ne M_2$.

Further, from \eqref{l_st_rl} it follows that, for each $j\ne l$,
$$
\nu_{(1)}^{1-\lambda} \nu_{(2)}^\lambda \prod _{i=1}^{t_1-1} k_{\sigma(i)} ^{1/q_{\sigma(i)} - 1/\theta_{\sigma(i)} (\overline{\alpha})} \prod _{i=t_1} ^{t_1+m-1} A_{j,\sigma(i)}^{1/q_{\sigma(i)} - 1/\theta_{\sigma(i)}(\overline{\alpha})}=
$$
$$
=\nu_{(1)}^{1-\lambda} \nu_{(2)}^\lambda \prod _{i=1}^{t_1-1} k_{\sigma(i)} ^{1/q_{\sigma(i)} - 1/\theta_{\sigma(i)} (\overline{\alpha})} \prod _{i=t_1} ^{t_1+m-1} A_{j_*,\sigma(i)}^{1/q_{\sigma(i)} - 1/\theta_{\sigma(i)}(\overline{\alpha})},
$$
$$
\nu_{(1)} \prod _{i=1}^{t_1-1} k_{\sigma(i)} ^{1/q_{\sigma(i)} - 1/\theta_{(1),\sigma(i)}}\prod _{i=t_1} ^{t_1+m-1} A_{j,\sigma(i)}^{1/q_{\sigma(i)} - 1/\theta_{(1),\sigma(i)}}=
$$
$$
=\nu_{(1)} \prod _{i=1}^{t_1-1} k_{\sigma(i)} ^{1/q_{\sigma(i)} - 1/\theta_{(1),\sigma(i)}}\prod _{i=t_1} ^{t_1+m-1} A_{j_*,\sigma(i)}^{1/q_{\sigma(i)} - 1/\theta_{(1),\sigma(i)}};
$$
this implies that
$$
\prod _{i=t_1} ^{t_1+m-1} A_{j,\sigma(i)}^{1/q_{\sigma(i)} - 1/\theta_{\sigma(i)}(\overline{\alpha})} = 
\prod _{i=t_1} ^{t_1+m-1} A_{j_*,\sigma(i)}^{1/q_{\sigma(i)} - 1/\theta_{\sigma(i)}(\overline{\alpha})},
$$
$$
\prod _{i=t_1} ^{t_1+m-1} A_{j,\sigma(i)}^{1/q_{\sigma(i)} - 1/\theta_{(1),\sigma(i)}}= \prod _{i=t_1} ^{t_1+m-1} A_{j_*,\sigma(i)}^{1/q_{\sigma(i)} - 1/\theta_{(1),\sigma(i)}}.
$$
Dividing one equality by the other, we obtain
$$
\prod _{i=t_1} ^{t_1+m-1} A_{j,\sigma(i)}^{1/\theta_{(1),\sigma(i)} - 1/\theta_{\sigma(i)}(\overline{\alpha})} = 
\prod _{i=t_1} ^{t_1+m-1} A_{j_*,\sigma(i)}^{1/\theta_{(1),\sigma(i)} - 1/\theta_{\sigma(i)}(\overline{\alpha})};
$$
since $1/\overline{\theta}(\overline{\alpha}) \in [1/\overline{\theta}_{(1)}, \, 1/\overline{\theta}_{(2)}]$, this yields the equality
\begin{align}
\label{aj_prod} \prod _{i=t_1} ^{t_1+m-1} A_{j,\sigma(i)}^{1/\theta_{(1),\sigma(i)} - 1/\theta_{(2),\sigma(i)}} = 
\prod _{i=t_1} ^{t_1+m-1} A_{j_*,\sigma(i)}^{1/\theta_{(1),\sigma(i)} - 1/\theta_{(2),\sigma(i)}}, \quad j\ne l.
\end{align}

Applying \eqref{nu1_nu2}, \eqref{tau_def_om}, we get
$$
\frac{\nu_{(1)}}{\nu_{(2)}} = \prod _{i=1}^d s_{\sigma(i)}^{1/\theta_{(1),\sigma(i)} - 1/\theta_{(2),\sigma(i)}} \stackrel{\eqref{s_def2}, \eqref{n_gr}, \eqref{s_def4}}{=}$$$$=\prod _{i=1}^{t_1-1} k_{\sigma(i)} ^{1/\theta_{(1),\sigma(i)} -1/\theta_{(2), \sigma(i)}} \prod _{i=t_1}^{t_1+m-1}\prod _{1\le j\le m} A_{j,\sigma(i)} ^{\tau_j (1/\theta_{(1),\sigma(i)} - 1/\theta_{(2),\sigma(i)})} \stackrel{\eqref{aj_prod}}{=}
$$
$$
= \prod _{i=1}^{t_1-1} k_{\sigma(i)} ^{1/\theta_{(1),\sigma(i)} -1/\theta_{(2), \sigma(i)}}  \prod _{i=t_1}^{t_1+m-1}(A_{j_*,\sigma(i)} ^{(1-\tau_l) (1/\theta_{(1),\sigma(i)} - 1/\theta_{(2),\sigma(i)})} A_{l,\sigma(i)} ^{\tau_l (1/\theta_{(1),\sigma(i)} - 1/\theta_{(2),\sigma(i)})}) = $$$$=M_1^{1-\tau_l} M_2^{\tau_l}.
$$
This together with \eqref{m1_ineq}, \eqref{m2_ineq} yields that $M_1\le M_1^{1-\tau_l} M_2^{\tau_l}\le M_2$. Since $M_1\ne M_2$, we get $0\le \tau_l\le 1$.

{\bf Comparison of $n$ and $k_1^{\frac{2}{q_1}}\dots k_d ^{\frac{2}{q_d}}s_1^{1-\frac{2}{q_1}} \dots s_d ^{1- \frac{2}{q_d}}$.}

In what follows we denote by $\overline{2}\in \R^d$ the vector whose each coordinate is equal to 2, and set
$$\Delta_I=\{\overline{x}_I:\; \overline{x}\in \Delta\}$$
(recall that the notation $\overline{x}_I$ was before Definition \ref{gen_pos}).

In the case $I=I_\omega$ we have
\begin{align}
\label{i_om_n} n= k_1^{\frac{2}{q_1}}\dots k_d ^{\frac{2}{q_d}}s_1^{1-\frac{2}{q_1}} \dots s_d ^{1- \frac{2}{q_d}}.
\end{align}
Indeed, from \eqref{s_def2}, \eqref{n_gr}, \eqref{s_def4} we get
$$
k_1^{\frac{2}{q_1}}\dots k_d ^{\frac{2}{q_d}}s_1^{1-\frac{2}{q_1}} \dots s_d ^{1- \frac{2}{q_d}} = \prod _{i=1}^{t_1-1} k_{\sigma(i)} \prod _{i=t_1}^d k_{\sigma(i)}^{2/q_{\sigma(i)}} \prod _{i=t_1} ^{t_1+m-1} \prod _{j=1}^m A_{j,\sigma(i)}^{\tau_j\Bigl(1-\frac{2}{q_{\sigma(i)}}\Bigr)} =:E.
$$
From \eqref{ais_def} it follows that, for each $j\in \{1, \, \dots, \, m\}$,
$$
\prod _{i=t_1} ^{t_1+m-1} A_{j,\sigma(i)}^{1-\frac{2}{q_{\sigma(i)}}} = \prod _{\sigma(i)\in I_+(s(j))} k_{\sigma(i)}^{1-\frac{2}{q_{\sigma(i)}}}  \cdot n \prod _{i=1}^{t_1-1} k_{\sigma(i)}^{-1}\prod _{\sigma(i)\in I_+(s(j))} k_{\sigma(i)}^{-1} \times$$$$\times \prod _{\sigma(i)\in \{i_*(s(j))\}\sqcup I_-(s(j))} k_{\sigma(i)}^{-2/q_{\sigma(i)}} \prod _{i=t_1+m}^d k_{\sigma(i)}^{-2/q_{\sigma(i)}} \stackrel{\eqref{i_razb_s}}{=}
n \prod _{i=1}^{t_1-1} k_{\sigma(i)}^{-1} \prod _{i=t_1}^d k_{\sigma(i)}^{-2/q_{\sigma(i)}}.
$$
Hence,
$$
E = \prod _{i=1}^{t_1-1} k_{\sigma(i)} \prod _{i=t_1}^d k_{\sigma(i)}^{2/q_{\sigma(i)}} \cdot n \prod _{i=1}^{t_1-1} k_{\sigma(i)}^{-1} \prod _{i=t_1}^d k_{\sigma(i)}^{-2/q_{\sigma(i)}} = n.
$$
This completes the proof of \eqref{i_om_n}.

Notice that we obtained the equality
\begin{align}
\label{pr_aj_12q} \prod _{i=t_1} ^{t_1+m-1} A_{j,\sigma(i)}^{1-\frac{2}{q_{\sigma(i)}}} = n \prod _{i=1}^{t_1-1} k_{\sigma(i)}^{-1} \prod _{i=t_1}^d k_{\sigma(i)}^{-2/q_{\sigma(i)}}.
\end{align}

Let $I=I_q$. We show that
\begin{align}
\label{iq_ineq1} n \le k_1^{\frac{2}{q_1}}\dots k_d ^{\frac{2}{q_d}}s_1^{1-\frac{2}{q_1}} \dots s_d ^{1- \frac{2}{q_d}}.
\end{align}

Recall that in the case $I=I_q$ we have $$n\le \prod _{i=1}^{\mu_2} k_{\sigma(i)} \prod _{i=\mu_2+1}^d k_{\sigma(i)} ^{2/q_{\sigma(i)}}$$
(see Proposition \ref{iq_sta}).

If $n\le \prod _{i=1}^{\mu_1} k_{\sigma(i)} \prod _{i=\mu_1+1}^d k_{\sigma(i)} ^{2/q_{\sigma(i)}}$, then \eqref{iq_ineq1} follows from the conditions $s_{\sigma(i)} = k_{\sigma(i)}$, $1\le i\le \mu_1$ (see \eqref{s_def1}) and the inequalities $s_{\sigma(i)}\ge 1$, $\mu_1+1\le i\le d$.

Let
\begin{align}
\label{n_predel1}
\prod _{i=1}^{\mu_1} k_{\sigma(i)} \prod _{i=\mu_1+1}^d k_{\sigma(i)} ^{2/q_{\sigma(i)}} < n \le \prod _{i=1}^{\mu_2} k_{\sigma(i)} \prod _{i=\mu_2+1}^d k_{\sigma(i)} ^{2/q_{\sigma(i)}}.
\end{align}

If $\mu_2 = \mu_1+1$ (i.e., $I_q$ is a singleton), then \eqref{iq_ineq1} follows from the inequality $$ s_{\sigma(\mu_1+1)} \stackrel{\eqref{s_l_est_n12kl}}{\ge} \Bigl(n^{1/2} \prod _{1\le i\le \mu_1} k_{\sigma(i)}^{-1/2} \cdot k_{\sigma(\mu_1+1)} ^{-1/q_{\sigma(\mu_1+1)}} \prod _{i=\mu_2+1}^d k_{\sigma(i)} ^{-1/q_{\sigma(i)}}\Bigr)^{\frac{1}{1/2-1/q_{\sigma(\mu_1+1)}}}$$ and the conditions $s_{\sigma(i)} = k_{\sigma(i)}$, $1\le i\le \mu_1$, $s_{\sigma(i)}=1$, $\mu_2+1\le i\le d$ (see \eqref{s_def1}).

Let ${\rm card}\, I_q\ge 2$. Consider the segment $[(1/\overline{q})_I, \, (1/\overline{2})_I] \cap \Delta_I$. One endpoint of this interval is $(1/\overline{q})_I$, and the other one is $\sum \limits _{j=1}^m \frac{\varkappa_j}{(\overline{p}_{\alpha_j})_I}$, where $\varkappa_j\ge 0$, $1\le j\le m$, $\sum \limits _{j=1}^m \varkappa_j=1$. Let the vector $\overline{\theta}_*$ be defined by the equation $\frac{1}{\overline{\theta}_*} = \sum \limits _{j=1}^m \frac{\varkappa_j}{\overline{p}_{\alpha_j}}$. For $\tau\in [0, \, 1]$, we define the vector $\overline{\theta}^\tau= (\theta^\tau_1, \, \dots, \, \theta ^\tau_d)$ by the equation $\frac{1}{\overline{\theta}^\tau} = \frac{1-\tau}{\overline{\theta}(\overline{\alpha})} + \frac{\tau}{\overline{\theta}_*}$. Then $\frac{1}{\overline{\theta}^\tau}\in \Delta$ and
\begin{align}
\label{theta_tau_om}
\omega' _{\theta ^\tau _{\sigma(\mu_1+1)}, q_{\sigma (\mu_1+1)}} = \dots = \omega' _{\theta^\tau _{\sigma(\mu_2)}, q_{\sigma(\mu_2)}} \in [0, \, 1]
\end{align}
(since $(1/\overline{\theta}^\tau)_I\in [(1/\overline{q})_I, \, (1/\overline{2})_I]$).

Let
$$
D = \{1/\overline{p}\in [0, \, 1]^d:\; p_{\sigma(i)} > q_{\sigma(i)}, \; 1\le i\le \mu_1, \; 
\omega'_{p_{\sigma(\mu_2)},q_{\sigma(\mu_2)}} < \omega'_{p_{\sigma(i)},q_{\sigma(i)}}, \; \mu_2+1\le i\le d\},
$$
$$
\lambda = \sup \{\tau \in [0, \, 1]:\; 1/\overline{\theta}^\tau\in D\}.
$$
Then $\lambda>0$, $$\frac{1}{\overline{\theta} ^{\lambda}} = \sum \limits _{j=1}^m \frac{\mu'_j}{\overline{p}_{\alpha_j}}$$
for some $\mu'_j\ge 0$, $\sum \limits _{j=1}^m \mu'_j = 1$.

We show that $\overline{\theta} ^{\lambda} \in \Theta$. Indeed, if $\lambda=1$, the following two cases may hold:
\begin{enumerate}
\item $1/\overline{\theta} ^{\lambda}$ belongs to an $m-2$-dimensional face $\Delta'\subset\Delta$ and satisfies \eqref{theta_tau_om}; by generality of position (see asserion 2 of Definition \ref{gen_pos}), ${\rm aff}\, \Delta'$ and the plane given by equations \eqref{theta_tau_om} are complementary, and $1/\overline{\theta} ^{\lambda}$ lies in the interior of $\Delta'$, $\omega'_{\theta^\tau _{\sigma(\mu_1+1)}, q_{\sigma(\mu_1 + 1)}}\in (0, \, 1)$. Hence $\overline{\theta} ^{\lambda} \in \Theta$;
\item $(1/\overline{\theta}_*)_I=(1/\overline{2})_I$ lies in the interior of $\Delta$. By generality of position, we again get $\overline{\theta} ^{\lambda} \in \Theta$.
\end{enumerate}
If $\lambda<1$, then $1/\overline{\theta} ^{\lambda}$ lies in relative interior of $\Delta$, and, besides \eqref{theta_tau_om}, an additional equality $p_i=q_i$ or $\omega'_{p_i,q_i} = \omega'_{p_{\sigma(\mu_2)}, q_{\sigma(\mu_2)}}\in (0, \, 1)$ holds; the last inclusion holds since $0<\lambda<1$. By generality of position, these conditions together give the set from $\tilde{\cal Z}_m$. We again apply assertion 2 of Definition \ref{gen_pos} and get $\overline{\theta} ^{\lambda} \in \Theta$.

From \eqref{psi_def} and \eqref{psi_eq_nu_phi} it follows that
\begin{align}
\label{th_tau_ineq1}
\nu_{\alpha_1}^{\lambda_1(\overline{\alpha})} \dots \nu_{\alpha_m} ^{\lambda_m(\overline{\alpha})}\Phi(\overline{\theta}(\overline{\alpha}), \, \overline{q}, \overline{k}, \, n)\le \nu_{\alpha_1}^{\mu'_1} \dots \nu_{\alpha_m} ^{\mu'_m}\Phi(\overline{\theta}^\lambda, \, \overline{q}, \overline{k}, \, n).
\end{align}
This together with \eqref{phi_case10}, \eqref{phi_case30}, Remark \ref{t1t1_pr}, \eqref{nu1_nu2}, \eqref{n_predel1}, \eqref{theta_tau_om}, the definition of $D$ and equalities $\theta_{\sigma(i)}(\overline{\alpha}) = q_{\sigma(i)}$ $(\mu_1+1\le i\le \mu_2)$ yields that
$$
\prod _{i=1}^d s_i^{1/\theta_i(\overline{\alpha})-1/\theta_i^\lambda} \prod_{i=1}^{\mu_1} k_{\sigma(i)} ^{1/q_{\sigma(i)} - 1/\theta_{\sigma(i)}(\overline{\alpha})} \le $$$$ \le \prod_{i=1}^{\mu_1} k_{\sigma(i)} ^{1/q_{\sigma(i)} - 1/\theta^\lambda_{\sigma(i)}} \Bigl( n^{1/2} \prod _{i=1}^{\mu_1} k_{\sigma(i)} ^{-1/2} \prod _{i=\mu_1+1}^d k_{\sigma(i)} ^{-1/ q_{\sigma(i)}}\Bigr)^{\frac{1/q_{\sigma(\mu_1+1)}-1/\theta ^\lambda _{\sigma(\mu_1+1)}}{1/2-1/q_{\sigma(\mu_1+1)}}}.
$$
Applying \eqref{s_def1} and the equalities $\theta_{\sigma(i)}(\overline{\alpha}) = q_{\sigma(i)}$, $\mu_1+1\le i\le \mu_2$, we obtain that 
$$
\prod_{i=\mu_1+1}^{\mu_2} s_{\sigma(i)} ^{1/q_{\sigma(i)} -1/\theta^\lambda _{\sigma(i)}} \le \Bigl( n^{1/2} \prod _{i=1}^{\mu_1} k_{\sigma(i)} ^{-1/2} \prod _{i=\mu_1+1}^d k_{\sigma(i)} ^{-1/ q_{\sigma(i)}}\Bigr)^{\frac{1/q_{\sigma(\mu_1+1)}-1/\theta ^\lambda _{\sigma(\mu_1+1)}}{1/2-1/q_{\sigma(\mu_1+1)}}}.
$$
This together with \eqref{s_def1}, \eqref{theta_tau_om} implies \eqref{iq_ineq1}.

Let $I=I_2$. We show that
\begin{align}
\label{iq_ineq2} n \ge k_1^{\frac{2}{q_1}}\dots k_d ^{\frac{2}{q_d}}s_1^{1-\frac{2}{q_1}} \dots s_d ^{1- \frac{2}{q_d}}.
\end{align}

By Proposition \ref{i2_sta}, $n> \prod _{i=1}^{\nu_1} k_{\sigma(i)} \prod _{i=\nu_1+1}^d k_{\sigma(i)} ^{2/q_{\sigma(i)}}$.

If $n> \prod _{i=1}^{\nu_2} k_{\sigma(i)} \prod _{i=\nu_2+1}^d k_{\sigma(i)} ^{2/q_{\sigma(i)}}$, then \eqref{iq_ineq2} follows from the inequalities $s_i\le k_i$, $1\le i\le d$, and the equalities $s_{\sigma(i)}=1$, $\nu_2+1\le i \le d$ (see \eqref{s_def3}).

Let
\begin{align}
\label{n_predel2} \prod _{i=1}^{\nu_1} k_{\sigma(i)} \prod _{i=\nu_1+1}^d k_{\sigma(i)} ^{2/q_{\sigma(i)}}< n \le \prod _{i=1}^{\nu_2} k_{\sigma(i)} \prod _{i=\nu_2+1}^d k_{\sigma(i)} ^{2/q_{\sigma(i)}}.
\end{align}

If $\nu_2 = \nu_1+1$, i.e., $I_2$ is a singleton, then \eqref{iq_ineq2} follows from the inequality $$ s_{\sigma(\nu_1+1)} \stackrel{\eqref{s_l_est_n121}}{\le} \Bigl(n^{1/2} \prod _{1\le i\le \nu_1} k_{\sigma(i)}^{-1/2} \prod _{i=\nu_1+1}^d k_{\sigma(i)} ^{-1/q_{\sigma(i)}}\Bigr)^{\frac{1}{1/2-1/q_{\sigma(\nu_1+1)}}}$$ and the conditions $s_{\sigma(i)} = k_{\sigma(i)}$, $1\le i\le \nu_1$, $s_{\sigma(i)}=1$, $\nu_2\le i\le d$ (see \eqref{s_def3}).

Let ${\rm card}\, I_2\ge 2$. Consider the segment $[(1/\overline{q})_I, \, (1/\overline{2})_I] \cap \Delta_I$. One endpoint of this interval is $(1/\overline{2})_I$, and the other one is $\sum \limits _{j=1}^m \frac{\varkappa_j}{(\overline{p}_{\alpha_j})_I}$, where $\varkappa_j\ge 0$, $1\le j\le m$, $\sum \limits _{j=1}^m \varkappa_j=1$. Let the vector $\overline{\theta}_*$ be defined by the equation $1/\overline{\theta}_*= \sum \limits _{j=1}^m \frac{\varkappa_j}{\overline{p}_{\alpha_j}}$. Given $\tau\in [0, \, 1]$, we define the vector $\overline{\theta}^\tau= (\theta^\tau_1, \, \dots, \, \theta ^\tau_d)$ by the equation $\frac{1}{\overline{\theta}^\tau} = \frac{1-\tau}{\overline{\theta}(\overline{\alpha})} + \frac{\tau}{\overline{\theta}_*}$. Then $\frac{1}{\overline{\theta}^\tau}\in \Delta$ and
\begin{align}
\label{theta_tau_om1}
\omega' _{\theta ^\tau _{\sigma(\nu_1+1)}, q_{\sigma (\nu_1+1)}} = \dots = \omega' _{\theta^\tau _{\sigma(\nu_2)}, q_{\sigma(\nu_2)}} \in [0, \, 1].
\end{align}

Let
$$
D = \{1/\overline{p}\in [0, \, 1]^d:\; p_{\sigma(i)} < 2, \; \nu_2+1\le i\le d, \;
\omega'_{p_{\sigma(\nu_1+1)},q_{\sigma(\nu_1+1)}} > \omega'_{p_{\sigma(i)},q_{\sigma(i)}}, \; 1\le i\le \nu_1\},
$$
$$
\lambda = \sup \{\tau \in [0, \, 1]:\; 1/\overline{\theta}^\tau\in D\}.
$$
Then $\lambda>0$, $$\frac{1}{\overline{\theta}^\lambda} = \sum \limits _{j=1}^m \frac{\mu'_j}{\overline{p}_{\alpha_j}}$$
for some $\mu'_j\ge 0$, $\sum \limits _{j=1}^m \mu'_j = 1$. As in the case $I=I_q$, we get \eqref{th_tau_ineq1}. This together with Remark \ref{t1t1_pr}, \eqref{phi_case10}, \eqref{phi_case30}, \eqref{nu1_nu2}, \eqref{n_predel2}, \eqref{theta_tau_om1} and the definition of $D$ implies that
$$
\prod _{i=1}^d s_i^{1/\theta_i(\overline{\alpha})-1/\theta^\lambda _i} \prod _{i=1}^{\nu_1} k_{\sigma(i)} ^{1/q_{\sigma(i)} - 1/\theta_{\sigma(i)}(\overline{\alpha})}  n^{-1/2} \prod _{i=1}^{\nu_1} k_{\sigma(i)} ^{1/2} \prod _{i=\nu_1+1}^d k_{\sigma(i)} ^{1/ q_{\sigma(i)}}\le $$$$ \le \prod _{i=1}^{\nu_1} k_{\sigma(i)} ^{1/q_{\sigma(i)} - 1/\theta^\lambda_{\sigma(i)}} \Bigl( n^{1/2} \prod _{i=1}^{\nu_1} k_{\sigma(i)} ^{-1/2} \prod _{i=\nu_1+1}^d k_{\sigma(i)} ^{-1/ q_{\sigma(i)}}\Bigr)^{\frac{1/q_{\sigma(\nu_1+1)}-1/\theta^\lambda _{\sigma(\nu_1+1)}}{1/2-1/q_{\sigma(\nu_1+1)}}}.
$$
Hence, by \eqref{s_def3} and the equalities $\theta_{\sigma(i)}(\overline{\alpha})=2$, $\nu_1+1\le i\le \nu_2$, we have $$\prod_{i=\nu_1+1}^{\nu_2} s_{\sigma(i)} ^{1/2 -1/\theta^\lambda_{\sigma(i)}} \le \Bigl( n^{1/2} \prod _{i=1}^{\nu_1} k_{\sigma(i)} ^{-1/2} \prod _{i=\nu_1+1}^d k_{\sigma(i)} ^{-1/ q_{\sigma(i)}}\Bigr)^{\frac{1/2-1/\theta^\lambda_{\sigma(\nu_1+1)}}{1/2-1/q_{\sigma(\nu_1+1)}}};$$
this together with \eqref{s_def3}, \eqref{theta_tau_om1} yields \eqref{iq_ineq2}.

{\bf Checking the inclusion $W\subset 2^dM$.} Recall that the set $W$ is given by formula \eqref{w_mg2_def}. We put $u_{\sigma(i)} = \lceil s_{\sigma(i)} \rceil$ for $i\le \nu_1$, $u_{\sigma(i)} = \lfloor s_{\sigma(i)} \rfloor$ for $i> \nu_1$.

If the inclusion $W\subset 2^dM$ holds, then for $I=I_q$ we get
$$
d_n(M, \, l_{\overline{q}}^{\overline{k}}) \underset{d} {\gtrsim} d_n(W, \, l_{\overline{q}}^{\overline{k}}) \stackrel{\eqref{vk1kddn},\eqref{sigma_choice}, \eqref{w_mg2_def},\eqref{s_def1},\eqref{iq_ineq1}}{\gtrsim} \nu_{\alpha_1}^{\lambda_1(\overline{\alpha})} \dots \nu_{\alpha_m} ^{\lambda_m(\overline{\alpha})} \prod _{i=1}^{\mu_1}k_{\sigma(i)}^{1/q_{\sigma(i)} - 1/\theta_{\sigma(i)}(\overline{\alpha})} \stackrel{\eqref{phi_case1}, \eqref{sigma_choice}}{=}
$$
$$
= \nu_{\alpha_1}^{\lambda_1(\overline{\alpha})} \dots \nu_{\alpha_m} ^{\lambda_m(\overline{\alpha})} \Phi(\overline{\theta}(\overline{\alpha}), \, \overline{q}, \, \overline{k}, \, n).
$$
Here we used the inequality
\begin{align}
\label{p3sta2}
n \le \prod _{i=1}^{\mu_2} k_{\sigma(i)} \prod _{i=\mu_2+1}^d k_{\sigma(i)}^{2/q_{\sigma(i)}};
\end{align}
see assertion 3 of Proposition \ref{iq_sta}.

For $I=I_2$ we get
$$
d_n(M, \, l_{\overline{q}}^{\overline{k}}) \underset{d} {\gtrsim} d_n(W, \, l_{\overline{q}}^{\overline{k}}) \stackrel{\eqref{vk1kddn}, \eqref{sigma_choice}, \eqref{w_mg2_def}, \eqref{s_def3},\eqref{iq_ineq2}}{\gtrsim} \nu_{\alpha_1}^{\lambda_1(\overline{\alpha})} \dots \nu_{\alpha_m} ^{\lambda_m(\overline{\alpha})} n^{-1/2}\prod _{i=1}^{d}k_{\sigma(i)}^{1/q_{\sigma(i)}} \prod _{i=1}^{\nu_1} k_{\sigma(i)}^{1/2 - 1/\theta_{\sigma(i)}(\overline{\alpha})} \stackrel{\eqref{phi_case3}, \eqref{sigma_choice}}{=}
$$
$$
= \nu_{\alpha_1}^{\lambda_1(\overline{\alpha})} \dots \nu_{\alpha_m} ^{\lambda_m(\overline{\alpha})} \Phi(\overline{\theta}(\overline{\alpha}), \, \overline{q}, \, \overline{k}, \, n).
$$
Here we used the inequality
\begin{align}
\label{p3sta3}
n > \prod _{i=1}^{\nu_1} k_{\sigma(i)} \prod _{i=\nu_1+1}^d k_{\sigma(i)}^{2/q_{\sigma(i)}};
\end{align}
see assertion 1 of Proposition \ref{i2_sta}.

For $I=I_\omega$ we have
\begin{align}
\label{sta4n}
\prod _{i=1}^{t_1-1} k_{\sigma(i)} \prod _{i=t_1}^d k_{\sigma(i)}^{2/q_{\sigma(i)}}<n \le \prod _{i=1}^{t_1+m-1} k_{\sigma(i)} \prod _{i=t_1+m}^d k_{\sigma(i)}^{2/q_{\sigma(i)}}
\end{align}
(see Proposition \ref{iomega_sta}). We get
$$
d_n(M, \, l_{\overline{q}}^{\overline{k}}) \underset{d} {\gtrsim} d_n(W, \, l_{\overline{q}}^{\overline{k}}) \stackrel{\eqref{vk1kddn},\eqref{sigma_choice}, \eqref{i_om_str}, \eqref{w_mg2_def},\eqref{s_def2},\eqref{i_om_n}}{\gtrsim} $$$$ \gtrsim \nu_{\alpha_1}^{\lambda_1(\overline{\alpha})} \dots \nu_{\alpha_m} ^{\lambda_m(\overline{\alpha})} \prod _{i=1}^{t_1-1} k_{\sigma(i)}^{1/q_{\sigma(i)} - 1/\theta_{\sigma(i)}(\overline{\alpha})} \prod _{i=t_1}^{t_1+m-1} s_{\sigma(i)}^{1/q_{\sigma(i)} - 1/\theta_{\sigma(i)}(\overline{\alpha})}=:P.
$$
We show that
\begin{align}
\label{prod_si_q_th}
\prod _{i=t_1}^{t_1+m-1} s_{\sigma(i)}^{1/q_{\sigma(i)} - 1/\theta_{\sigma(i)}(\overline{\alpha})} = \Bigl(n^{1/2}\prod _{i=1}^{t_1-1} k_{\sigma(i)} ^{-1/2}\prod _{i=t_1}^d k_{\sigma(i)}^{-1/q_{\sigma(i)}}\Bigr) ^{\frac{1/q_{\sigma(t_1)} - 1/\theta_{\sigma(t_1)}(\overline{\alpha})}{1/2-1/q_{\sigma(t_1)}}}.
\end{align}
It suffices to check that
\begin{align}
\label{prod_it_ajsi}
\prod _{i=t_1}^{t_1+m-1}A_{j,\sigma(i)}^{1/q_{\sigma(i)} - 1/\theta_{\sigma(i)}(\overline{\alpha})} = \Bigl(n^{1/2}\prod _{i=1}^{t_1-1} k_{\sigma(i)} ^{-1/2}\prod _{i=t_1}^d k_{\sigma(i)}^{-1/q_{\sigma(i)}}\Bigr) ^{\frac{1/q_{\sigma(t_1)} - 1/\theta_{\sigma(t_1)}(\overline{\alpha})}{1/2-1/q_{\sigma(t_1)}}}
\end{align}
and apply \eqref{s_def4}.

Taking into account the equalities $\omega_{\theta_{\sigma(t_1)}(\overline{\alpha}), q_{\sigma(t_1)}} = \dots = \omega_{\theta_{\sigma(t_1+m-1)}(\overline{\alpha}), q_{\sigma(t_1+m-1)}}$, we get
$$
\prod _{i=t_1}^{t_1+m-1}A_{j,\sigma(i)}^{1/q_{\sigma(i)} - 1/\theta_{\sigma(i)}(\overline{\alpha})} = \prod _{i=t_1}^{t_1+m-1}A_{j,\sigma(i)}^{-\omega_{\theta_{\sigma(i)}(\overline{\alpha}), q_{\sigma(i)}}(1/2-1/q_{\sigma(i)})} = 
$$
$$
=\Bigl(\prod _{i=t_1}^{t_1+m-1}A_{j,\sigma(i)}^{1/2-1/q_{\sigma(i)}} \Bigr)^{-\omega_{\theta_{\sigma(t_1)}(\overline{\alpha}), q_{\sigma(t_1)}}} \stackrel{\eqref{pr_aj_12q}}{=} \Bigl(n^{1/2}\prod _{i=1}^{t_1-1} k_{\sigma(i)} ^{-1/2}\prod _{i=t_1}^d k_{\sigma(i)}^{-1/q_{\sigma(i)}}\Bigr)^{-\omega_{\theta_{\sigma(t_1)}(\overline{\alpha}), q_{\sigma(t_1)}}};
$$
this completes the proof of \eqref{prod_it_ajsi}.

From \eqref{prod_si_q_th} we get
\begin{align}
\label{phi_form_eq_om}
\begin{array}{c}
P=\nu_{\alpha_1}^{\lambda_1(\overline{\alpha})} \dots \nu_{\alpha_m} ^{\lambda_m(\overline{\alpha})} \prod _{i=1}^{t_1-1} k_{\sigma(i)}^{1/q_{\sigma(i)} - 1/\theta_{\sigma(i)}(\overline{\alpha})} \times \\ \times \Bigl(n^{1/2}\prod _{i=1}^{t_1-1} k_{\sigma(i)} ^{-1/2}\prod _{i=t_1}^d k_{\sigma(i)}^{-1/q_{\sigma(i)}}\Bigr) ^{\frac{1/q_{\sigma(t_1)} - 1/\theta_{\sigma(t_1)}(\overline{\alpha})}{1/2-1/q_{\sigma(t_1)}}} =
\\
= \nu_{\alpha_1}^{\lambda_1(\overline{\alpha})} \dots \nu_{\alpha_m} ^{\lambda_m(\overline{\alpha})} \Phi(\overline{\theta}(\overline{\alpha}), \, \overline{q}, \, \overline{k}, \, n);
\end{array}
\end{align}
the last equality follows from Remark \ref{t1t1_pr} and \eqref{sta4n}.

In order to prove the inclusion $W\subset 2^d M$, it suffices to check that, for any $\beta \in A$,
\begin{align}
\label{incl_ineq11} \nu_{\alpha_1}^{\lambda_1(\overline{\alpha})} \dots \nu_{\alpha_m} ^{\lambda_m(\overline{\alpha})} \prod _{i=1}^d s_i^{1/p_{\beta,i}-1/\theta_i(\overline{\alpha})} \le \nu_\beta.
\end{align}

If $\beta\in \{\alpha_j\}_{j=1}^m$, it follows from \eqref{nu1_nu2}.

Let $\beta \notin \{\alpha_j\}_{j=1}^m$.

Recall the notation $$\Delta = {\rm conv}\, \{1/\overline{p}_{\alpha_j}\}_{j=1}^m, \quad \Delta_I = \{\overline{x}_I:\; \overline{x}\in \Delta\}.$$

{\bf The case $I=I_q$.} Then we have \eqref{p3sta2} and $\# I = m-1$.

By Lemma \ref{conv_hull} and generality of position (see assertion 1 of Definition \ref{gen_pos}), there is $j\in \{1, \, \dots, \, m\}$ such that 
\begin{align}
\label{1qi_in_conv}
(1/\overline{q})_I \in {\rm conv}\, \{(1/\overline{p}_{\alpha_1})_I, \, \dots, \, (1/\overline{p}_{\alpha_{j-1}})_I, \, (1/\overline{p}_{\alpha_{j+1}})_I, \, \dots, \, (1/\overline{p}_{\alpha_m})_I, \, (1/\overline{p}_{\beta})_I\}.
\end{align}

Without loss of generality, $j=m$. We set $\overline{\gamma} = (\gamma_1, \, \dots, \, \gamma_{m-1}, \, \gamma_m) := (\alpha_1, \, \dots, \, \alpha_{m-1}, \, \beta)$, $\tilde \Delta = {\rm conv}\{1/\overline{p}_{\gamma_j}\}_{j=1}^m$.

Let $t_{\sigma(i)} = s_{\sigma(i)}$, $i\in \{1, \, \dots, \, \mu_1\}\cup \{\mu_2+1, \, \dots, \, d\}$ (see \eqref{s_def1}); for $i\in \{\mu_1+1, \, \dots, \, \mu_2\}$, we define the numbers $t_{\sigma(i)}$ by the equations
\begin{align}
\label{si_def_ini11} \frac{\nu_{\gamma_j}}{\nu_{\gamma_1}} = \prod _{i=1}^d t_i ^{1/p_{\gamma_j,i} - 1/p_{\gamma_1,i}}, \quad 2\le j\le m.
\end{align}
By generality of position (see assertion 1 of Definition \ref{gen_pos}), the vectors $\{(1/\overline{p}_{\gamma_j})_I\}_{j=1}^m$ are affinely independent; therefore the system of equations \eqref{si_def_ini11} has the unique solution.

Direct calculations show that the analogue of Proposition \ref{sta_nu12} holds: if $\mu_{j,k}\ge 0$, $1\le j\le m$, $k=1, \, 2$, $\sum \limits _{j=1}^m \mu_{j,k}=1$, the vectors $\overline{\theta}_k$ and the numbers $\nu_{(k)}$ are defined by the equations
\begin{align}
\label{1theta_k11}
\frac{1}{\overline{\theta}_k} =\sum \limits _{j=1}^m \frac{\mu_{j,k}}{\overline{p}_{\gamma_j}}, \quad \nu_{(k)} = \prod _{j=1}^m \nu_{\gamma_j}^{\mu_{j,k}}, \quad k=1, \, 2,
\end{align}
then
\begin{align}
\label{nu1_nu211} \frac{\nu_{(1)}}{\nu_{(2)}} =  \prod _{l=1}^d t_l^{1/\theta_{1,l}-1/\theta_{2,l}}.
\end{align}

Since $(1/\overline{q})_I\in \tilde \Delta$, there are numbers $\lambda_j(\overline{\gamma})\ge 0$ ($j=1, \, \dots, \, m$) such that $\sum \limits _{j=1}^m \lambda_j(\overline{\gamma})=1$ and $(1/\overline{q})_I = \sum \limits _{j=1}^m \lambda_j(\overline{\gamma})/ (\overline{p}_{\gamma_j})_I$. We define the vector $\overline{\theta}(\overline{\gamma})$ by the equation $1/\overline{\theta}(\overline{\gamma}) = \sum \limits _{j=1}^m \lambda_j(\overline{\gamma})/ \overline{p}_{\gamma_j}$. By generality of position (see assertion 2 of Definition \ref{gen_pos}), we have $\lambda_j(\overline{\gamma})> 0$ ($j=1, \, \dots, \, m$); otherwise, for some set $J\subset \{1, \, \dots, \, m\}$, $\# J<m$, the simplex ${\rm conv}\, \{1/\overline{p}_{\gamma_j}\}_{j\in J}$ and the plane $Z=\{(x_1, \, \dots, \, x_d):\; x_i = 1/q_i, \; i\in I\}$ intersect. In particular,
\begin{align}
\label{1qiniconv} (1/\overline{q})_I \notin {\rm conv}\, \{(1/\overline{p}_{\alpha_j})_I\}_{j=1}^{m-1}.
\end{align}
Also from assertion 2 of Definition \ref{gen_pos} it follows that $Z$ and ${\rm aff}\, \{1/\overline{p}_{\gamma_j}\}_{j=1}^m$ are complementary. Hence $\overline{\theta}(\overline{\gamma}) \in \Theta$.

Let
$$
D =\{1/\overline{p}\in [0, \, 1]^d:\; p_{\sigma(j)} > q_{\sigma(j)}, \, 1\le j\le \mu_1, \; p_{\sigma(j)} < q_{\sigma(j)}, \, \mu_2+1\le j\le d\}.
$$
Then $D$ is open in $[0, \, 1]^d$ and $1/\overline{\theta}(\overline{\alpha})\in D$. We set
$$
\tau = \sup \Bigl\{t\in [0, \, 1]:\; \frac{1-t}{\overline{\theta}(\overline{\alpha})} + \frac{t}{\overline{\theta}(\overline{\gamma})} \in D\Bigr\}, \quad \frac{1}{\overline{r}} = \frac{1-\tau}{\overline{\theta}(\overline{\alpha})} + \frac{\tau}{\overline{\theta}(\overline{\gamma})}.
$$
Then $\tau>0$. We show that $\overline{r} \in \Theta$. Indeed, if $\tau=1$, then $\overline{r} = \overline{\theta}(\overline{\gamma})\in \Theta$. If $0<\tau <1$, then $1/\overline{r}\in {\rm conv}\, \{1/\overline{p}_{\alpha_1}, \, \dots, \, 1/\overline{p}_{\alpha_m}, \, 1/\overline{p}_\beta\}$, the equalities $r_i=1/q_i$ hold for $i\in I$ and also for one $i\notin I$. We again use generality of position and get $\overline{r} \in \Theta$.

From \eqref{psi_def}, \eqref{psi_eq_nu_phi} it follows that
\begin{align}
\label{789}
\begin{array}{c}
\nu_{\alpha_1}^{\lambda_1(\overline{\alpha})} \dots \nu_{\alpha_m}^{\lambda_m(\overline{\alpha})}\Phi(\overline{\theta}(\overline{\alpha}),  \, \overline{q}, \, \overline{k}, \, n) \le \\ \le\nu_{\alpha_1}^{(1-\tau)\lambda_1(\overline{\alpha})+\tau\lambda_1(\overline{\gamma})} \dots \nu_{\alpha_{m-1}}^{(1-\tau)\lambda_{m-1}(\overline{\alpha}) + \tau \lambda_{m-1}(\overline{\gamma})} \nu_{\alpha_m}^{(1-\tau) \lambda_m(\overline{\alpha})} \nu_\beta ^{\tau \lambda_m(\overline{\gamma})} \Phi(\overline{r},  \, \overline{q}, \, \overline{k}, \, n).
\end{array}
\end{align}

From \eqref{p3sta2}, the definition of $D$, $\overline{r}$, $\overline{\theta}(\overline{\gamma})$ and \eqref{phi_case1}, \eqref{sigma_choice} it follows that \eqref{789} can be written as
\begin{align}
\label{phi_q_in}
\begin{array}{c}
\nu_{\alpha_1}^{\lambda_1(\overline{\alpha})} \dots \nu_{\alpha_m}^{\lambda_m(\overline{\alpha})} \prod _{i=1}^{\mu_1} k_{\sigma(i)}^{1/q_{\sigma(i)} - 1/\theta_{\sigma(i)}(\overline{\alpha})} \le 
\\
\le \prod _{j=1}^{m-1}\nu_{\alpha_j}^{(1-\tau)\lambda_j(\overline{\alpha})+\tau\lambda_j(\overline{\gamma})} \cdot \nu_{\alpha_m}^{(1-\tau) \lambda_m(\overline{\alpha})} \nu_{\beta}^{\tau \lambda_m(\overline{\gamma})} \prod _{i=1}^{\mu_1} k_{\sigma(i)}^{1/q_{\sigma(i)} - (1-\tau)/\theta_{\sigma(i)}(\overline{\alpha})-\tau/\theta_{\sigma(i)}(\overline{\gamma})};
\end{array}
\end{align}
hence
$$
\nu_{\alpha_1}^{\lambda_1(\overline{\alpha})} \dots \nu_{\alpha_m}^{\lambda_m(\overline{\alpha})} \prod _{i=1}^{\mu_1} k_{\sigma(i)}^{1/q_{\sigma(i)} - 1/\theta_{\sigma(i)}(\overline{\alpha})} \le 
$$
$$
\le \nu_{\gamma_1}^{\lambda_1(\overline{\gamma})} \dots \nu_{\gamma_{m-1}}^{\lambda_{m-1}(\overline{\gamma})} \nu_{\gamma_m}^{\lambda_m(\overline{\gamma})} \prod _{i=1}^{\mu_1} k_{\sigma(i)}^{1/q_{\sigma(i)} -1/\theta_{\sigma(i)}(\overline{\gamma})}.
$$
We divide both parts of the inequality by $\nu_{\alpha_j}$ ($1\le j\le m-1$) and apply \eqref{nu1_nu2}, \eqref{nu1_nu211}; this yields
$$
\prod _{i=1}^d s_{\sigma(i)} ^{1/\theta_{\sigma(i)}(\overline{\alpha}) -1/p_{\alpha_j,\sigma(i)}} \prod _{i=1}^{\mu_1} k_{\sigma(i)}^{1/q_{\sigma(i)} - 1/\theta_{\sigma(i)}(\overline{\alpha})}\le \prod _{i=1}^d t_{\sigma(i)} ^{1/\theta_{\sigma(i)}(\overline{\gamma}) -1/p_{\alpha_j,\sigma(i)}} \prod _{i=1}^{\mu_1} k_{\sigma(i)}^{1/q_{\sigma(i)} - 1/\theta_{\sigma(i)}(\overline{\gamma})}.
$$
Since $s_{\sigma(i)} = t_{\sigma(i)} = k_{\sigma(i)}$ for $i\le \mu_1$, $s_{\sigma(i)} = t_{\sigma(i)} =1$ for $i\ge \mu_2+1$, $\theta_{\sigma(i)}(\overline{\alpha}) = \theta_{\sigma(i)}(\overline{\gamma}) = q_{\sigma(i)}$ for $\mu_1+1\le i\le \mu_2$, we get
$$
\prod _{i\in I} s_i ^{1/q_i -1/p_{\alpha_j,i}} \le \prod _{i\in I} t_i ^{1/q_i -1/p_{\alpha_j,i}}.
$$
Let $\tau_j\ge 0$, $1\le j\le m-1$, $\sum \limits _{j=1} ^{m-1} \tau_j=1$. Then
$$
\prod _{i\in I} s_i ^{1/q_i -\sum \limits_{j=1}^{m-1} \tau_j/p_{\alpha_j,i}} \le \prod _{i\in I} t_i ^{1/q_i -\sum \limits_{j=1}^{m-1} \tau_j/p_{\alpha_j,i}}.
$$
By \eqref{1qi_in_conv}, there are $(\tau_1, \, \dots, \, \tau_{m-1})$ such that the vectors $\Bigl(\frac{1}{\overline{q}} - \sum \limits_{j=1}^{m-1} \frac{\tau_j}{\overline{p}_{\alpha_j}}\Bigr)_I$ and $\Bigl(\frac{1}{\overline{p}_\beta} - \sum \limits_{j=1}^{m-1} \frac{\tau_j}{\overline{p}_{\alpha_j}}\Bigr)_I$ are codirectional, and $\Bigl(\frac{1}{\overline{q}} - \sum \limits_{j=1}^{m-1} \frac{\tau_j}{\overline{p}_{\alpha_j}}\Bigr)_I\ne 0$ by \eqref{1qiniconv}. Hence
$$
\prod _{i\in I} s_i ^{1/p_{\beta,i} -\sum \limits_{j=1}^{m-1} \tau_j/p_{\alpha_j,i}} \le \prod _{i\in I} t_i ^{1/p_{\beta,i} -\sum \limits_{j=1}^{m-1} \tau_j/p_{\alpha_j,i}}.
$$
We again take into account the equality $s_i=t_i$ for $i\notin I$ and get
$$
\prod _{i=1}^d s_i ^{1/p_{\beta,i} -\sum \limits_{j=1}^{m-1} \tau_j/p_{\alpha_j,i}} \le \prod _{i=1}^d t_i ^{1/p_{\beta,i} -\sum \limits_{j=1}^{m-1} \tau_j/p_{\alpha_j,i}}.
$$
Hence
$$
\prod _{i=1}^d s_i^{1/\theta_i(\overline{\alpha})-\sum \limits _{j=1}^{m-1} \tau_j/p_{\alpha_j,i}} \prod _{i=1}^d s_i^{1/p_{\beta,i}-1/\theta_i(\overline{\alpha})} \le \prod _{i=1}^d t_i ^{1/p_{\beta,i} -\sum \limits_{j=1}^{m-1} \tau_j/p_{\alpha_j,i}}.
$$
Applying \eqref{nu1_nu2} and \eqref{nu1_nu211}, we obtain
$$
\frac{\nu_{\alpha_1}^{\lambda_1(\overline{\alpha})}\dots \nu_{\alpha_m}^{\lambda_m(\overline{\alpha})}}{\nu_{\alpha_1}^{\tau_1}\dots \nu_{\alpha_{m-1}}^{\tau_{m-1}}}\prod _{i=1}^d s_i^{1/p_{\beta,i}-1/\theta_i(\overline{\alpha})} \le \frac{\nu_\beta}{\nu_{\alpha_1}^{\tau_1}\dots \nu_{\alpha_{m-1}}^{\tau_{m-1}}};
$$
i.e., \eqref{incl_ineq11} holds.

{\bf The case $I=I_2$} is similar to the previous one. Recall that \eqref{p3sta3} holds. The equalities $\theta_i(\overline{\gamma}) = q_i$ for $i\in I$ are replaced by $\theta_i(\overline{\gamma}) = 2$. The set $D$ is given by the conditions
$$
D=\{1/\overline{p}\in [0, \, 1]^d:\; p_{\sigma(i)}<2, \; 1\le i\le \nu_1, \; p_{\sigma(i)}>2, \; \nu_2+1\le i\le d\}.
$$

The numbers $t_i$ are defined as follows: $t_{\sigma(i)} = s_{\sigma(i)}$, $i\in \{1, \, \dots, \, \nu_1\}\cup \{\nu_2+1, \, \dots, \, d\}$ (see \eqref{s_def3}); for $i\in \{\nu_1+1, \, \dots, \, \nu_2\}$ we define the numbers $t_{\sigma(i)}$ by \eqref{si_def_ini11}. Then again we get \eqref{nu1_nu211}.

Instead of \eqref{phi_q_in} we have (see \eqref{phi_case3}, \eqref{sigma_choice}, \eqref{p3sta3})
$$
\nu_{\alpha_1}^{\lambda_1(\overline{\alpha})} \dots \nu_{\alpha_m}^{\lambda_m(\overline{\alpha})} \prod _{i=1}^{\nu_1} k_{\sigma(i)}^{1/q_{\sigma(i)} - 1/\theta_{\sigma(i)}(\overline{\alpha})} n^{-1/2} \prod _{i=1}^{\nu_1} k_{\sigma(i)}^{1/2} \prod _{i=\nu_1+1}^d k_{\sigma(i)}^{1/q_{\sigma(i)}} \le 
$$$$
\le \nu_{\alpha_1}^{(1-\tau)\lambda_1(\overline{\alpha})+\tau\lambda_1(\overline{\gamma})} \dots \nu_{\alpha_{m-1}}^{(1-\tau)\lambda_{m-1}(\overline{\alpha}) + \tau \lambda_{m-1}(\overline{\gamma})} \nu_{\alpha_m}^{(1-\tau) \lambda_m(\overline{\alpha})} \nu_{\beta}^{\tau \lambda_m(\overline{\gamma})} \times$$$$ \times\prod _{i=1}^{\nu_1} k_{\sigma(i)}^{1/q_{\sigma(i)} - (1-\tau)/\theta_{\sigma(i)}(\overline{\alpha})-\tau/\theta_{\sigma(i)}(\overline{\gamma})}n^{-1/2} \prod _{i=1}^{\nu_1} k_{\sigma(i)}^{1/2} \prod _{i=\nu_1+1}^d k_{\sigma(i)}^{1/q_{\sigma(i)}};
$$
this implies
\begin{align}
\label{na1l1prod}
\begin{array}{c}
\nu_{\alpha_1}^{\lambda_1(\overline{\alpha})} \dots \nu_{\alpha_m}^{\lambda_m(\overline{\alpha})} \prod _{i=1}^{\nu_1} k_{\sigma(i)}^{1/q_{\sigma(i)} - 1/\theta_{\sigma(i)}(\overline{\alpha})} \le 
\\
\le \nu_{\gamma_1}^{\lambda_1(\overline{\gamma})} \dots \nu_{\gamma_{m-1}}^{\lambda_{m-1}(\overline{\gamma})} \nu_{\gamma_m}^{\lambda_m(\overline{\gamma})} \prod _{i=1}^{\nu_1} k_{\sigma(i)}^{1/q_{\sigma(i)} -1/\theta_{\sigma(i)}(\overline{\gamma})}.
\end{array}
\end{align}
We divide both parts of the inequality by $\nu_{\alpha_j}$ ($1\le j\le m-1$) and apply \eqref{nu1_nu2}, \eqref{nu1_nu211}; this yields
$$
\prod _{i=1}^d s_{\sigma(i)} ^{1/\theta_{\sigma(i)}(\overline{\alpha}) -1/p_{\alpha_j,\sigma(i)}} \prod _{i=1}^{\nu_1} k_{\sigma(i)}^{1/q_{\sigma(i)} - 1/\theta_{\sigma(i)}(\overline{\alpha})}\le \prod _{i=1}^d t_{\sigma(i)} ^{1/\theta_{\sigma(i)}(\overline{\gamma}) -1/p_{\alpha_j,\sigma(i)}} \prod _{i=1}^{\nu_1} k_{\sigma(i)}^{1/q_{\sigma(i)} - 1/\theta_{\sigma(i)}(\overline{\gamma})}.
$$
Since $s_{\sigma(i)} = t_{\sigma(i)} = k_{\sigma(i)}$ for $j\le \nu_1$, $s_{\sigma(i)} = t_{\sigma(i)} =1$ for $j\ge \nu_2+1$, $\theta_{\sigma(i)}(\overline{\alpha}) = \theta_{\sigma(i)}(\overline{\gamma}) = 2$ for $\nu_1+1\le i\le \nu_2$, we have
$$
\prod _{i\in I} s_i ^{1/2 -1/p_{\alpha_j,i}} \le \prod _{i\in I} t_i ^{1/2 -1/p_{\alpha_j,i}};
$$
the further arguments are as in the case $I=I_q$.

{\bf The case $I=I_\omega$.} In this case, \eqref{sta4n} holds. Recall that $I=\{\sigma(t_1), \, \sigma(t_1+1), \, \dots, \, \sigma(t_1+m-1)\}$.

Let $Z$ be the set from Definition \ref{nm_def1}, and the minimum in \eqref{psi_def} attains at $Z$. Then $1/\overline{\theta}(\overline{\alpha}) \in Z$,
\begin{align}
\label{z_omega}
\begin{array}{c}
\tilde Z := {\rm aff}\, Z = \{(x_1, \, \dots, \, x_d) \in \R^d:\; \omega'_{1/x_{\sigma(t_1)},q_{\sigma(t_1)}} = \omega'_{1/x_{\sigma(t_1+1)},q_{\sigma(t_1+1)}} = \\ \dots = \omega'_{1/x_{\sigma(t_1+m-1)},q_{\sigma(t_1+m-1)}}\}.
\end{array}
\end{align}

We denote $$\R^d_I = \{\overline{x}_I:\; \overline{x}\in \R^d\}, \; \tilde Z_I= \{\overline{x}_I:\; \overline{x}\in \tilde Z\}.$$
Then $\tilde Z_I$ is a straight line in $\R^d_I$ given by the equation
\begin{align}
\label{zi_2q} \tilde Z_I = \{(1/\overline{q})_I + \kappa ((1/\overline{2})_I-(1/\overline{q})_I):\; \kappa \in \R\}.
\end{align}

Let $L\subset (\R^d)_I$ be a hyperplane orthogonal to $\tilde Z_I$. For each $\overline{x}\in \R^d$ we denote by $(\overline{x})'_I$ the orthogonal projection $\overline{x}_I$ onto $L$ and set
$$\Delta'_I=\{(\overline{x})'_I:\; \overline{x}\in \Delta\}.$$
Notice that 
\begin{align}
\label{th_q_2_pr}
(1/\overline{\theta}(\overline{\alpha}))'_I = (1/\overline{2})'_I = (1/\overline{q})'_I.
\end{align}

By generality of position, for any different $\beta_1, \, \dots, \, \beta_m\in A$, the points $\{(1/\overline{p}_{\beta_j})_I\}_{j=1}^m$ are affinely independent (see assertion 1 of Definition \ref{gen_pos}), and the planes ${\rm aff}\,\{1/\overline{p}_{\beta_j}\}_{j=1}^m$ and ${\rm aff}\, Z$
are complementary (see assertion 2 of Definition \ref{gen_pos}). From \eqref{z_omega} it follows that ${\rm aff}\,\{(1/\overline{p}_{\beta_j})_I\}_{j=1}^m$ and $\tilde Z_I$ are complementary. Therefore, the points $\{(1/\overline{p}_{\beta_j})'_I\}_{j=1}^m$ are affinely independent.

Let $1/\overline{\theta}(\overline{\beta})$ be the intersection point of ${\rm aff}\,\{1/\overline{p}_{\beta_j}\}_{j=1}^m$ and ${\rm aff}\, Z$. Then $(1/\overline{\theta}(\overline{\beta}))_I$ is the intersection point of ${\rm aff}\,\{(1/\overline{p}_{\beta_j})_I\}_{j=1}^m$ and $\tilde Z_I$. Hence $(1/\overline{\theta}(\overline{\beta}))'_I = (1/\overline{\theta}(\overline{\alpha}))'_I$.
From assertion 2 of Definition \ref{gen_pos} it follows that $1/\overline{\theta}(\overline{\beta})$ cannot belong to relative boundary of the simplex ${\rm conv}\, \{1/\overline{p}_{\beta_j}\}_{j=1}^m$.
Therefore, $(1/\overline{\theta}(\overline{\beta}))'_I= (1/\overline{\theta}(\overline{\alpha}))'_I$ does not belong to the boundary of the simplex ${\rm conv}\, \{(1/\overline{p}_{\beta_j})'_I\}_{j=1}^m\subset L$; in particular, $(1/\overline{\theta}(\overline{\alpha}))'_I$ is the interior point of $\Delta'_I$.

Hence, for $\xi_j=(1/\overline{p}_{\alpha_j})'_I$ $(j=1, \, \dots, \, m)$, $\eta=(1/\overline{p}_\beta)'_I$ (where $\beta\in A\backslash \{\alpha_1, \, \dots, \, \alpha_m\}$) and $a=(1/\overline{\theta}(\overline{\alpha}))'_I$, the conditions of Lemma \ref{conv_hull} hold. Therefore, there is $j\in \{1, \, \dots, \, m\}$ such that 
$$
(1/\overline{\theta}(\overline{\alpha}))'_I \in {\rm conv}\, 
\{(1/\overline{p}_{\alpha_1})'_I, \, \dots, \, (1/\overline{p}_{\alpha_{j-1}})'_I, \, (1/\overline{p}_{\alpha_{j+1}})'_I, \, \dots, \, (1/\overline{p}_{\alpha_m})'_I, \, (1/\overline{p}_{\beta})'_I\}
$$
(and, as it was mentioned above, this point cannot belong to the boundary of this simplex). Without loss of generality, $j=m$. Let $\overline{\gamma} = (\gamma_1, \, \dots, \, \gamma_{m-1}, \, \gamma_m) = (\alpha_1, \, \dots, \, \alpha_{m-1}, \, \beta)$. We have $(1/\overline{\theta}(\overline{\alpha}))'_I = \sum \limits _{j=1}^m \kappa_j (1/\overline{p}_{\gamma_j})'_I$, where $\kappa_j> 0$, $\sum \limits _{j=1}^m \kappa_j=1$. We set $1/\overline{\theta}(\overline{\gamma}):=\sum \limits _{j=1}^m \kappa_j/\overline{p}_{\gamma_j}$, $\tilde \Delta = {\rm conv}\, \{1/\overline{p}_{\gamma_j}\}_{j=1}^m$. Since $(1/\overline{\theta}(\overline{\gamma}))'_I = (1/\overline{\theta}(\overline{\alpha}))'_I$, we have
\begin{align}
\label{th_g_z} 1/\overline{\theta}(\overline{\gamma}) \in \tilde Z.
\end{align}

Let
$$
D = \{1/\overline{p}\in [0, \, 1]^d:\; 0<\omega' _{p_{\sigma (t_1)}, q_{\sigma(t_1)}}<1,$$$$\omega' _{p_{\sigma (t_1)}, q_{\sigma(t_1)}}> \omega' _{p_{\sigma (i)}, q_{\sigma(i)}}, \; i<t_1, \; \omega' _{p_{\sigma (t_1)}, q_{\sigma(t_1)}}< \omega' _{p_{\sigma (i)}, q_{\sigma(i)}}, \; i>t_1+m-1\},
$$
$$
\tau = \sup \{\lambda\in [0, \, 1]:\; (1-\lambda)/\overline{\theta}(\overline{\alpha}) + \lambda/ \overline{\theta}(\overline{\gamma})\in D\}.
$$
Then $\tau>0$. We define the vector $\overline{r}=(r_1, \, \dots, \, r_d)$ by the equation
\begin{align}
\label{r_vect_defin} \frac{1}{\overline{r}} = \frac{1-\tau}{\overline{\theta}(\overline{\alpha})} + \frac{\tau}{\overline{\theta}(\overline{\gamma})}.
\end{align}
Then $1/\overline{r}\in \tilde Z$. Hence
\begin{align}
\label{r_om_eq} \omega'_{r_{\sigma(t_1)},q_{\sigma(t_1)}} = \omega'_{r_{\sigma(t_1+1)},q_{\sigma(t_1+1)}} = \dots = \omega'_{r_{\sigma(t_1+m-1)},q_{\sigma(t_1+m-1)}} \in [0, \, 1]
\end{align}
(the last inclusion follows from the definition of $D$).

We show that $\overline{r}\in \Theta$. Indeed, if $\tau = 1$, then $\overline{r}=\overline{\theta}(\overline{\gamma})$; here ${\rm aff}\, Z$ and ${\rm aff}\, \tilde \Delta$ are complementary and $1/\overline{\theta}(\overline{\gamma})$ lies in the interior of $\tilde \Delta$. We prove that $\omega'_{r_{\sigma(t_1)},q_{\sigma(t_1)}} \in (0, \, 1)$. Indeed, if $\omega' _{r_{\sigma(t_1)},q_{\sigma(t_1)}}$ is 0 or 1, then $r_{\sigma(t_1)}=q_{\sigma(t_1)}, \, \dots, \,  r_{\sigma(t_1+m-1)}=q_{\sigma(t_1+m-1)}$ or $r_{\sigma(t_1)}=2, \, \dots , \, r_{\sigma(t_1+m-1)}=2$; i.e., $1/\overline{\theta}(\overline{\gamma})$ belongs to some set from $\tilde{\cal Z}_{m+1}$; this contradicts to assertion 2 of Definition \ref{gen_pos}. Hence we proved the inclusion $\overline{r}\in \Theta$ for $\tau=1$.

Let $0<\tau <1$. Then $1/\overline{r}\in {\rm conv}\, \{1/\overline{p}_{\alpha_1}, \, \dots, \, 1/\overline{p}_{\alpha_m}, \, 1/\overline{p}_{\beta}\}$ and \eqref{r_om_eq} holds; in addition, one of the following conditions holds: 1) $r_{\sigma(t_1)}=q_{\sigma(t_1)}, \, \dots, \,  r_{\sigma(t_1+m-1)}=q_{\sigma(t_1+m-1)}$, 2) $r_{\sigma(t_1)}=2, \, \dots, \, r_{\sigma(t_1+m-1)}=2$, 3) besides \eqref{r_om_eq}, one additional equality $\omega'_{r_{\sigma(t_1)},q_{\sigma(t_1)}} = \omega'_{r_{\sigma(i)},q_{\sigma(i)}}$ holds, where $i\notin \{t_1, \, \dots, \, t_1+m-1\}$. Applying assertions 1, 2 of Definition \ref{gen_pos}, we get that $\overline{r}\in \Theta$.

This together with \eqref{psi_def}, \eqref{psi_eq_nu_phi} implies \eqref{789}.

From the definition of the set $D$, Remark \ref{t1t1_pr}, \eqref{phi_case1}, \eqref{phi_case3}, \eqref{sta4n} and \eqref{r_om_eq}, we get
\begin{align}
\label{phi_alpha_gamma}
\begin{array}{c}
\Phi(\overline{\theta}(\overline{\alpha}), \, \overline{q}, \, \overline{k}, \, n) = \prod _{i=1}^{t_1-1} k_{\sigma(i)} ^{1/q_{\sigma(i)} - 1/\theta_{\sigma(i)}(\overline{\alpha})} \Bigl(n^{1/2}\prod _{i=1}^{t_1-1} k_{\sigma(i)} ^{-1/2}\prod _{i=t_1}^d k_{\sigma(i)}^{-1/q_{\sigma(i)}}\Bigr) ^{\frac{1/q_{\sigma(t_1)} - 1/\theta_{\sigma(t_1)}(\overline{\alpha})}{1/2-1/q_{\sigma(t_1)}}},
\\
\Phi(\overline{r}, \, \overline{q}, \, \overline{k}, \, n) = \prod _{i=1}^{t_1-1} k_{\sigma(i)} ^{1/q_{\sigma(i)} - 1/r_{\sigma(i)}} \Bigl(n^{1/2}\prod _{i=1}^{t_1-1} k_{\sigma(i)} ^{-1/2}\prod _{i=t_1}^d k_{\sigma(i)}^{-1/q_{\sigma(i)}}\Bigr) ^{\frac{1/q_{\sigma(t_1)} - 1/r_{\sigma(t_1)}}{1/2-1/q_{\sigma(t_1)}}}.
\end{array}
\end{align}

We set
\begin{align}
\label{s_pr_i_nii}
s'_{\sigma(i)} =s_{\sigma(i)} \stackrel{\eqref{s_def2}}{=} k_{\sigma(i)} \text{ for }1\le i\le t_1-1, \; s'_{\sigma(i)} =s_{\sigma(i)} \stackrel{\eqref{s_def2}}{=} 1\text{ for }t_1+m\le i \le d.
\end{align}
For $t_1\le i\le t_1+m-1$ we set
\begin{align}
\label{s_pr_def_ai}
s'_{\sigma(i)} = \prod _{k=1}^m A_{k,\sigma(i)} ^{\tau_k'},
\end{align}
where $A_{k,\sigma(i)}$ are given by formula \eqref{a_ti_def}, and the numbers $\tau_k'$ are defined by the equations
\begin{align}
\label{tau_pr_syst}
\frac{\nu_{\gamma_j}}{\nu_{\gamma_1}} = \prod _{i=1} ^d (s'_i) ^{1/p_{\gamma_j,i} - 1/p_{\gamma_1,i}}, \; 2\le j \le m, \; \sum \limits _{k=1}^m \tau'_k=1.
\end{align}
Taking the logarithm of both parts of first $m-1$ equations and excluding $\tau_1'$, we get the system of linear equations for $\tau_2', \, \dots, \, \tau_m'$ with the matrix
$$
\Bigl(\sum \limits _{i=t_1}^{t_1+m-1} (1/p_{\gamma_j, \sigma(i)} - 1/p_{\gamma_1,\sigma(i)})(\log A_{k,\sigma(i)} - \log A_{1,\sigma(i)})\Bigr) _{2\le k\le m, \, 2\le j\le m}.
$$
By generality of position (see assertion 3 of Definition \ref{gen_pos}) and Corollary \ref{a_matr_prop}, this matrix is nondegenerated; hence $\tau_2', \, \dots, \, \tau_m'$ are well-defined.

The analogue of Proposition \ref{sta_nu12} holds: if $\mu_{j,k}\ge 0$, $1\le j\le m$, $k=1, \, 2$, $\sum \limits _{j=1}^m \mu_{j,k}=1$, the vectors $\overline{\theta}_k = (\theta_{k,1}, \, \dots, \, \theta_{k,d})$ and the numbers $\nu_{(k)}$ are defined by the equations
\begin{align}
\label{1theta_k22}
\frac{1}{\overline{\theta}_k} =\sum \limits _{j=1}^m \frac{\mu_{j,k}}{\overline{p}_{\gamma_j}}, \quad \nu_{(k)} = \prod _{j=1}^m \nu_{\gamma_j}^{\mu_{j,k}}, \quad k=1, \, 2,
\end{align}
then
\begin{align}
\label{nu1_nu222} \frac{\nu_{(1)}}{\nu_{(2)}} =  \prod _{l=1}^d (s'_l)^{1/\theta_{1,l}-1/\theta_{2,l}}.
\end{align}

From \eqref{789} and \eqref{phi_alpha_gamma} we get
$$
\nu_{\alpha_1}^{\lambda_1(\overline{\alpha})} \dots \nu_{\alpha_m}^{\lambda_m(\overline{\alpha})} \prod _{i=1}^{t_1-1} k_{\sigma(i)} ^{1/q_{\sigma(i)} - 1/\theta_{\sigma(i)}(\overline{\alpha})} \Bigl(n^{1/2}\prod _{i=1}^{t_1-1} k_{\sigma(i)} ^{-1/2}\prod _{i=t_1}^d k_{\sigma(i)}^{-1/q_{\sigma(i)}}\Bigr) ^{\frac{1/q_{\sigma(t_1)} - 1/\theta_{\sigma(t_1)}(\overline{\alpha})}{1/2-1/q_{\sigma(t_1)}}}\le
$$
$$
\le \nu_{\alpha_1}^{(1-\tau)\lambda_1(\overline{\alpha}) + \tau \lambda_1(\overline{\gamma})} \dots \nu_{\alpha_{m-1}}^{(1-\tau)\lambda_{m-1}(\overline{\alpha}) + \tau \lambda_{m-1}(\overline{\gamma})}\nu_{\alpha_m}^{(1-\tau)\lambda_m(\overline{\alpha})} \nu_\beta^{\tau \lambda_m(\overline{\gamma})}\times $$$$ \times\prod _{i=1}^{t_1-1} k_{\sigma(i)} ^{1/q_{\sigma(i)} - 1/r_{\sigma(i)}}  \Bigl(n^{1/2}\prod _{i=1}^{t_1-1} k_{\sigma(i)} ^{-1/2}\prod _{i=t_1}^d k_{\sigma(i)}^{-1/q_{\sigma(i)}}\Bigr) ^{\frac{1/q_{\sigma(t_1)} - 1/r_{\sigma(t_1)}}{1/2-1/q_{\sigma(t_1)}}};
$$
this together with \eqref{r_vect_defin} yields
$$
\nu_{\alpha_1}^{\lambda_1(\overline{\alpha})} \dots \nu_{\alpha_m}^{\lambda_m(\overline{\alpha})} \prod _{i=1}^{t_1-1} k_{\sigma(i)} ^{1/q_{\sigma(i)} - 1/\theta_{\sigma(i)}(\overline{\alpha})} \Bigl(n^{1/2}\prod _{i=1}^{t_1-1} k_{\sigma(i)} ^{-1/2}\prod _{i=t_1}^d k_{\sigma(i)}^{-1/q_{\sigma(i)}}\Bigr) ^{\frac{1/q_{\sigma(t_1)} - 1/\theta_{\sigma(t_1)}(\overline{\alpha})}{1/2-1/q_{\sigma(t_1)}}}\le
$$
$$
\le \nu_{\gamma_1}^{\lambda_1(\overline{\gamma})} \dots \nu_{\gamma_m}^{\lambda_m(\overline{\gamma})} \prod _{i=1}^{t_1-1} k_{\sigma(i)} ^{1/q_{\sigma(i)} - 1/\theta_{\sigma(i)}(\overline{\gamma})}  \Bigl(n^{1/2}\prod _{i=1}^{t_1-1} k_{\sigma(i)} ^{-1/2}\prod _{i=t_1}^d k_{\sigma(i)}^{-1/q_{\sigma(i)}}\Bigr) ^{\frac{1/q_{\sigma(t_1)} - 1/\theta_{\sigma(t_1)}(\overline{\gamma})}{1/2-1/q_{\sigma(t_1)}}}.
$$
From \eqref{nu1_nu2} and \eqref{nu1_nu222} it follows that for each $j\in \{1, \, \dots, \, m-1\}$ we have
$$
\prod _{i=1}^d s_{\sigma(i)} ^{1/\theta_\sigma(i) (\overline{\alpha}) -1/p_{\alpha_j,\sigma(i)}} \prod _{i=1}^{t_1-1} k_{\sigma(i)} ^{1/q_{\sigma(i)} - 1/\theta_{\sigma(i)}(\overline{\alpha})} \times$$$$ \times\Bigl(n^{1/2}\prod _{i=1}^{t_1-1} k_{\sigma(i)} ^{-1/2}\prod _{i=t_1}^d k_{\sigma(i)}^{-1/q_{\sigma(i)}}\Bigr) ^{\frac{1/q_{\sigma(t_1)} - 1/\theta_{\sigma(t_1)}(\overline{\alpha})}{1/2-1/q_{\sigma(t_1)}}}\le
$$
$$
\le \prod _{i=1}^d (s'_{\sigma(i)}) ^{1/\theta_{\sigma(i)}(\overline{\gamma}) -1/p_{\alpha_j,\sigma(i)}} \prod _{i=1}^{t_1-1} k_{\sigma(i)} ^{1/q_{\sigma(i)} - 1/\theta_{\sigma(i)}(\overline{\gamma})} \times $$$$ \times \Bigl(n^{1/2}\prod _{i=1}^{t_1-1} k_{\sigma(i)} ^{-1/2}\prod _{i=t_1}^d k_{\sigma(i)}^{-1/q_{\sigma(i)}}\Bigr) ^{\frac{1/q_{\sigma(t_1)} - 1/\theta_{\sigma(t_1)}(\overline{\gamma})}{1/2-1/q_{\sigma(t_1)}}}.
$$
From \eqref{s_def2} and \eqref{s_pr_i_nii} we get
$$
\prod _{i=1}^{t_1-1} k_{\sigma(i)} ^{1/q_{\sigma(i)} - 1/p_{\alpha_{j},\sigma(i)}} \prod _{i=t_1}^{t_1+m-1} s_{\sigma(i)} ^{1/q_{\sigma(i)} - 1/p_{\alpha_{j},\sigma(i)}} \prod _{i=t_1}^{t_1+m-1} s_{\sigma(i)} ^{1/\theta_{\sigma(i)}(\overline{\alpha})-1/q_{\sigma(i)}} \times
$$
$$
\times \Bigl(n^{1/2}\prod _{i=1}^{t_1-1} k_{\sigma(i)} ^{-1/2}\prod _{i=t_1}^d k_{\sigma(i)}^{-1/q_{\sigma(i)}}\Bigr) ^{\frac{1/q_{\sigma(t_1)} - 1/\theta_{\sigma(t_1)}(\overline{\alpha})}{1/2-1/q_{\sigma(t_1)}}} \le
$$
$$
\le \prod _{i=1}^{t_1-1} k_{\sigma(i)} ^{1/q_{\sigma(i)} - 1/p_{\alpha_{j},\sigma(i)}} \prod _{i=t_1}^{t_1+m-1} (s'_{\sigma(i)}) ^{1/q_{\sigma(i)} - 1/p_{\alpha_{j},\sigma(i)}} \prod _{i=t_1}^{t_1+m-1} (s'_{\sigma(i)}) ^{1/\theta_{\sigma(i)}(\overline{\gamma})-1/q_{\sigma(i)}} \times
$$
$$
\times \Bigl(n^{1/2}\prod _{i=1}^{t_1-1} k_{\sigma(i)} ^{-1/2}\prod _{i=t_1}^d k_{\sigma(i)}^{-1/q_{\sigma(i)}}\Bigr) ^{\frac{1/q_{\sigma(t_1)} - 1/\theta_{\sigma(t_1)}(\overline{\gamma})}{1/2-1/q_{\sigma(t_1)}}}.
$$

Recall that \eqref{prod_si_q_th} holds. Similarly from \eqref{pr_aj_12q}, \eqref{z_omega}, \eqref{th_g_z}, \eqref{s_pr_def_ai} we get
$$
\prod _{i=t_1}^{t_1+m-1} (s'_{\sigma(i)}) ^{1/\theta_{\sigma(i)}(\overline{\gamma}) -1/q_{\sigma(i)}} = 
 \Bigl(n^{1/2}\prod _{i=1}^{t_1-1} k_{\sigma(i)} ^{-1/2}\prod _{i=t_1}^d k_{\sigma(i)}^{-1/q_{\sigma(i)}}\Bigr) ^{\frac{1/\theta_{\sigma(t_1)}(\overline{\gamma})-1/q_{\sigma(t_1)}}{1/2-1/q_{\sigma(t_1)}}}.
$$
Hence
\begin{align}
\label{pr_t_sqpaj}
\prod _{i=t_1}^{t_1+m-1} s_{\sigma(i)} ^{1/q_{\sigma(i)} - 1/p_{\alpha_{j},\sigma(i)}} \le \prod _{i=t_1}^{t_1+m-1} (s'_{\sigma(i)}) ^{1/q_{\sigma(i)} - 1/p_{\alpha_{j},\sigma(i)}}.
\end{align}
We show that 
\begin{align}
\label{pr_t_s2paj}
\prod _{i=t_1}^{t_1+m-1} s_{\sigma(i)} ^{1/\theta_{\sigma(i)}(\overline{\gamma}) - 1/p_{\alpha_{j},\sigma(i)}} \le \prod _{i=t_1}^{t_1+m-1} (s'_{\sigma(i)}) ^{1/\theta_{\sigma(i)}(\overline{\gamma}) - 1/p_{\alpha_{j},\sigma(i)}}.
\end{align}
Indeed, from \eqref{zi_2q} and \eqref{th_g_z} it follows that $(1/\overline{\theta}(\overline{\gamma}))_I = (1/\overline{q})_I + \kappa ((1/\overline{2})_I- (1/\overline{q})_I)$ for some $\kappa \in \R$. Hence, by \eqref{pr_t_sqpaj},
it suffices to check the equalities
$$
\prod _{i=t_1}^{t_1+m-1} s_{\sigma(i)} ^{1/2-1/q_{\sigma(i)}}=\prod _{i=t_1}^{t_1+m-1} (s'_{\sigma(i)}) ^{1/2-1/q_{\sigma(i)}} = n^{1/2} \prod _{i=1} ^{t_1-1} k_{\sigma(i)} ^{-1/2} \prod _{i=t_1} ^{d} k_{\sigma(i)} ^{-1/q_{\sigma(i)}}.
$$
It follows from \eqref{s_def4}, \eqref{pr_aj_12q} and \eqref{s_pr_def_ai}.

From \eqref{pr_t_s2paj} we get that for all $\varkappa_j\ge 0$ ($1\le j\le m-1$) such that $\sum \limits _{j=1}^{m-1} \varkappa _j=1$, the following inequality holds:
\begin{align}
\label{ggggg}
\prod _{i=t_1}^{t_1+m-1} s_{\sigma(i)} ^{1/\theta_{\sigma(i)}(\overline{\gamma}) - \sum \limits_{j=1}^{m-1}\varkappa_j/p_{\alpha_{j},\sigma(i)}} \le \prod _{i=t_1}^{t_1+m-1} (s'_{\sigma(i)}) ^{1/\theta_{\sigma(i)}(\overline{\gamma}) - \sum \limits_{j=1}^{m-1}\varkappa_j/p_{\alpha_{j},\sigma(i)}}.
\end{align}

We prove that there are $\varkappa_j\ge 0$ ($1\le j\le m-1$) such that $\sum \limits _{j=1}^{m-1} \varkappa _j=1$ and the vectors $\Bigl(1/\overline{\theta}(\overline{\gamma}) - \sum \limits_{j=1}^{m-1}\varkappa_j/\overline{p}_{\alpha_j}\Bigr)_I$ and $\Bigl(1/\overline{p}_\beta - \sum \limits_{j=1}^{m-1}\varkappa_j/\overline{p}_{\alpha_j}\Bigr)_I$ are codirectional.

Indeed, by the definition of $\overline{\gamma}$ and $\overline{\theta}(\overline{\gamma})$, we have $(1/\overline{\theta}(\overline{\gamma}))'_I =(1/\overline{\theta}(\overline{\alpha}))'_I \in {\rm int}\, \tilde \Delta'_I$ (recall that $\tilde \Delta = {\rm conv} \{1/\overline{p}_{\gamma_j}\}_{j=1}^m$, $(\gamma_1, \, \dots, \, \gamma_{m-1}, \, \gamma_m) = (\alpha_1, \, \dots, \, \alpha_{m-1}, \, \beta)$). Let $\Delta_m = {\rm conv}\, \{1/\overline{p}_{\alpha_j}\}_{j=1}^{m-1}$. Consider the ray emanating from $(1/\overline{p}_{\beta})'_I$ and passing through $(1/\overline{\theta}(\overline{\gamma}))'_I$. It intersects with $(\Delta_m)'_I := \{(\overline{x})'_I:\; \overline{x}\in \Delta_m\}$ at a point $\overline{\xi}$. The numbers $\varkappa_j$ ($0\le j\le m$) are defined from the equation $\overline{\xi} = \sum \limits _{j=1}^{m-1} \varkappa_j (1/\overline{p}_{\alpha_j})'_I$.

We have $\overline{\xi}= (\overline{a})'_I$ for $\overline{a} := \sum \limits _{j=1}^{m-1} \varkappa_j/\overline{p}_{\alpha_j}\in \Delta_m$. Then
$$
(1/\overline{\theta}(\overline{\gamma}))_I\in [(1/\overline{p}_{\beta})_I, \, (\overline{a})_I],
$$
and $(1/\overline{\theta}(\overline{\gamma}))_I \ne (\overline{a})_I$.

Hence, from \eqref{ggggg} we get
$$
\prod _{i=t_1}^{t_1+m-1} s_{\sigma(i)} ^{1/p_{\beta,\sigma(i)}- \sum \limits_{j=1}^{m-1}\varkappa_j/p_{\alpha_{j},\sigma(i)}} \le \prod _{i=t_1}^{t_1+m-1} (s'_{\sigma(i)}) ^{1/ p_{\beta,\sigma(i)}- \sum \limits_{j=1}^{m-1}\varkappa_j/p_{\alpha_{j},\sigma(i)}};
$$
this together with \eqref{s_pr_i_nii} implies that
$$
\prod _{i=1}^{d} s_{\sigma(i)} ^{1/p_{\beta,\sigma(i)}- \sum \limits_{j=1}^{m-1}\varkappa_j/p_{\alpha_{j},\sigma(i)}} \le \prod _{i=1}^{d} (s'_{\sigma(i)}) ^{1/ p_{\beta,\sigma(i)}- \sum \limits_{j=1}^{m-1}\varkappa_j/p_{\alpha_{j},\sigma(i)}}.
$$
Therefore,
$$
\prod _{i=1}^{d} s_{\sigma(i)} ^{1/\theta_{\sigma(i)}(\overline{\alpha}) - \sum \limits_{j=1}^{m-1}\varkappa_j/p_{\alpha_{j},\sigma(i)}}\prod _{i=1}^{d} s_{\sigma(i)} ^{1/p_{\beta,\sigma(i)}- 1/\theta_{\sigma(i)}(\overline{\alpha})} \le $$$$\le \prod _{i=1}^{d} (s'_{\sigma(i)}) ^{1/ p_{\beta,\sigma(i)}- \sum \limits_{j=1}^{m-1}\varkappa_j/p_{\alpha_{j},\sigma(i)}}.
$$
Applying \eqref{nu1_nu2} and \eqref{nu1_nu222}, we get
$$
\frac{\nu_{\alpha_1}^{\lambda_1(\overline{\alpha})}\dots \nu_{\alpha_m}^{\lambda_m(\overline{\alpha})}}{\nu_{\alpha_1}^{\varkappa_1} \dots \nu_{\alpha_{m-1}}^{\varkappa_{m-1}}} \prod _{i=1}^{d} s_{\sigma(i)} ^{1/p_{\beta,\sigma(i)}- 1/\theta_{\sigma(i)}(\overline{\alpha})} \le \frac{\nu_\beta}{\nu_{\alpha_1}^{\varkappa_1} \dots \nu_{\alpha_{m-1}}^{\varkappa_{m-1}}};
$$
i.e., \eqref{incl_ineq11} holds.

\section{The estimate in the general case}

{\bf Step 1.} First we prove that the assertion of Theorem \ref{main} holds if the set $A$ is finite, $q_i>2$, $i=1, \, \dots, \, d$, and $\{(\nu_\alpha, \, \overline{p}_\alpha)\}_{\alpha\in A}$ are arbitrary.

{\bf Substep 1.1.} Let $\{\overline{p}_{\alpha}\}_{\alpha\in A}$ satisfy conditions 1, 2, 3 of Definition \ref{gen_pos}. Consider the right-hand side of \eqref{dn_low_est_mod}. If the sets $Z$, $Z'\in \tilde{\cal Z}_m$ are different, $\overline{\alpha}\in \tilde{\cal N}_m$ and the numbers $\lambda_j(\overline{\alpha}, \, Z)$, $\lambda_j(\overline{\alpha}, \, Z')$ are well-defined, then $(\lambda_1(\overline{\alpha}, \, Z), \, \dots, \, \lambda_m(\overline{\alpha}, \, Z))$ and $(\lambda_1(\overline{\alpha}, \, Z'), \, \dots, \, \lambda_m(\overline{\alpha}, \, Z'))$ are different: otherwise, ${\rm conv}\{1/\overline{p}_{\alpha_j}\}_{j=1}^m$ intersects with $Z\cap Z'$, which contradicts to assertion 2 of Definition \ref{gen_pos}.

So, by slightly changing the numbers $\nu_\alpha$ we can satisfy condition 4 of Definition \ref{gen_pos}; with this proviso, the values of both parts in the order equality \eqref{main_eq} will change by at most a factor 2.

\vskip 0.3cm

Now  let $\{(\nu_\alpha, \, \overline{p}_\alpha)\}_{\alpha\in A}$ be arbitrary. For fixed $\nu_\alpha$, we will change the points $\overline{p}_{\alpha}$ so as to eventually satisfy conditions 1, 2, 3 of Definition \ref{gen_pos}. At each step, only one point is translated, and the distance of this translation can be made arbitrarily small. (The number of such translastions is bounded from above by the number depending only on $\# A$ and $d$.)

{\bf Substep 1.2.} First we translate the points so that the new set of points will satisfy condition 1. We make induction on $m$.

Let $m=2$, $I\subset \{1, \, \dots, \, d\}$, $\# I = 1$, $\alpha_1$, $\alpha_2\in A$, $\alpha_1\ne \alpha_2$. Then $I=\{i_*\}$. The condition of affine dependence of $\{(1/\overline{p}_{\alpha_1})_I, \, (1/\overline{p}_{\alpha_2})_I\}$ is equivalent to the equality $1/p_{\alpha_1,i_*} = 1/p_{\alpha_2,i_*}$. It suffices to  change $1/p_{\alpha_2,i_*}$, and the translation may be arbitrarily small. Going through all singletons $I$ and all pairs of different indices from $A$, we obtain a new set of points, for which assertion 1 of Definition \ref{gen_pos} holds for $m=2$.

Let a new point set be obtained, for which assertion 1 holds for all $m\le l-1$. We make the induction step from $l-1$ to $l$.

Let $3\le l\le d+1$, $\# I = l-1$, let $\{\alpha_j\}_{j=1}^l\subset A$ be different points, and let $\{(1/\overline{p}_{\alpha_j})_I\}_{j=1}^l$ be affinely dependent. Consider a set $I' \subset I$, $\# I' = l-2$. By the induction hypothesis, the points $\{(1/\overline{p}_{\alpha_j})_{I'}\}_{j=1}^{l-1}$ are affinely independent. Hence, $\{(1/\overline{p}_{\alpha_j})_I\}_{j=1}^{l-1}$ are affinely independent. Therefore, we can translate $(1/\overline{p}_{\alpha_l})_I$ by an arbitrarily small vector, so that the new point set will be affinely independent.

{\bf Substep 1.3.} We now make successive small translations of the points $1/\overline{p}_\alpha$ so as to satisfy condition 2 of Definition \ref{gen_pos}. Let $Z$ be a plane given by condition \eqref{pl_tr_simpl}, and let $\beta_1, \, \dots, \, \beta_m\in A$ be different indices. Then $$Z = \xi_0 + {\rm span}\, \{\overline{v}_j\}_{j=1}^{d-m+1},$$
where $\{\overline{v}_j\}_{j=1}^{d-m+1}$ is the system of linearly independent vectors. Let $\overline{w}_j= 1/\overline{p}_{\beta_j}-1/\overline{p}_{\beta_1}$, $2\le j\le m$. Changing if necessary the numeration, we may without loss of generality assume that for some $j_0\in \{2, \, \dots, \, m+1\}$ the system $E:= \{\overline{v}_j\}_{j=1}^{d-m+1} \cup \{\overline{w}_j\}_{j=2}^{j_0-1}$ is linearly independent, and $E\cup \{\overline{w}_j\}$ is linearly dependent for each $j\in \{j_0, \, \dots, \, m\}$. If $j_0\le m$, then by an arbitrarily small translations  of $1/\overline{p}_{\beta_{j_0}}$ we can achieve that $E\cup \{\overline{w}_{j_0}\}$ would be linearly independent. After that we similarly change $1/\overline{p}_{\beta_{j_0+1}}$, and so on.

Making successive translations for each $m\in \{2, \, \dots, \, d+1\}$, for each $d+m-1$-dimensional plane $Z$ defined by condition \eqref{pl_tr_simpl}, and for each tuple of different indices $\beta_1, \, \dots, \, \beta_m\in A$, we satisfy the first part of item 2 of Definition \ref{gen_pos}.

Let now $k<m$, $\beta_1, \, \dots, \, \beta_k\in A$. If $Z$ and ${\rm conv}\, \{1/\overline{p}_{\beta_j}\}_{j=1}^k$ intersect, then $Z\cup {\rm conv}\, \{1/\overline{p}_{\beta_j}\}_{j=1}^k\subset V$, where $V$ is a $d-1$-dimensional plane. Hence $V$ divides $\R^d$ into two open halfspaces $\Pi_+$ and $\Pi_-$. Performing in succession arbitrarily small translations of  the points $1/\overline{p}_{\beta_j}$, we will make sure that $1/\overline{p}_{\beta_j}\in \Pi_+$ for all $j=1, \, \dots, \, k$; in this case, $Z\cap {\rm conv}\, \{1/\overline{p}_{\beta_j}\}_{j=1}^k$ will become empty. Next, we make such translations for all $m\in \{2, \, \dots, \, d+1\}$, $k<m$, $Z$ and $\{1/\overline{p}_{\beta_j}\}_{j=1}^k$, after which
we satisfy item 2 of Definition \ref{gen_pos}.

{\bf Substep 1.4.} We will now move the points so as to satisfy item 3 of Definition \ref{gen_pos}. Let ${\cal B} = (B_{j,i})_{2\le j\le m, \, i\in I} \in \hat {\cal M} _{m,I}$, and let $\beta_1, \, \dots, \, \beta_m\in A$ be different indices. Item 1 of Definition \ref{gen_pos} is already satisfied (see Substep 1.2); hence the points $\{(1/\overline{p}_{\beta_j})_I\} _{j=1}^m$ are affinely independent. Therefore, the vectors $\overline{v}_j:= (1/\overline{p}_{\beta_j}-1/\overline{p}_{\beta_1})_I$, $2\le j\le m$, are linearly independent. Recall that $\# I=m$. The matrix ${\cal B}$ has range $m-1$; hence, there are $i_*\in I$ and numbers $\lambda_t$, $t\in I':= I\backslash \{i_*\}$, such that $B_{j,i_*} = \sum \limits _{t\in I'} \lambda _t B_{j,t}$, $2\le j\le m$, and the matrix $(B_{j,t})_{2\le j\le m, \, t\in I'}$ is nondegenerate. 

Let $\overline{v}_k = (v_{k,i})_{i\in I}$.
Then
$$
\sum \limits _{i\in I} v_{k,i} B_{j,i} = \sum \limits _{i\in I'} \Bigl(v_{k,i} +\lambda_i v_{k,i_*}\Bigr) B_{j,i}.
$$
The matrices $W:=\Bigl(v_{k,i} + \lambda_i v_{k,i_*}\Bigr)_{i\in I',2\le k\le m}$ and ${\cal B}'=(B_{j,i})_{2\le j\le m, \, i\in I'}$ are square, and ${\cal B}'$ is nondegenerate. So, by small translations of the points $1/\overline{p}_{\beta_j}$ we will make sure that the matrix $W$ would be nondegenerate, i.e.,
the vectors $$\overline{w}_k:=\Bigl(v_{k,i} +\lambda_i v_{k,i_*}\Bigr)_{i\in I'}, \quad 2\le k\le m,$$ would be linearly independent.

Consider the operator $Q$ associating with the vector $\overline{u} = (u_i)_{i\in I}$ the vector whose $i$th coordinate is $u_i+ \lambda_i u_{i_*}$ for $i\in I'$, and the  coordinate with number $i_*$ is $0$. Then $Q$ is the projection onto the subspace $\{(u_i)_{i\in I}:\; u_{i_*}=0\}$ along the one-dimensional subspace $\Lambda = {\rm span}\, \{(\mu_i)_{i\in I}\}$, where $\mu_i=\lambda_i$ for $i\in I'$, $\mu_{i_*}=-1$.

So, our aim is to make the vectors $\{Q\overline{v}_k\}_{2\le k\le m}$ linearly dependent by successive small translations of the points $1/\overline{p}_{\beta_j}$.
This will be the case if the only common point of $\Lambda$ and $L:={\rm span}\, \{\overline{v}_k\}_{2\le k\le m}$ is the origin.

Let $\Lambda$ and $L$ have a nontrivial intersection. Then $\Lambda \subset L$.

We set $Z_j={\rm span}\, \{\overline{v}_k\}_{2\le k\le m, \, k\ne j}$, $j=2, \, \dots, \, m$. Then $\cap _{j=2}^m Z_j=\{0\}$; hence there is $j_*\in \{2, \, \dots, \, m\}$ such that $Z_{j_*} \cap \Lambda = \{0\}$.

For any $\varepsilon>0$ there is a point $1/\overline{p}_{\beta_{j_*}}^\varepsilon$ such that $|1/\overline{p}_{\beta_{j_*}}^\varepsilon - 1/\overline{p}_{\beta_{j_*}}| < \varepsilon$ and $\overline{v}_{j_*}^\varepsilon := 1/\overline{p}_{\beta_{j_*}}^\varepsilon - 1/\overline{p}_{\beta_1}\notin L$. 

Let $L^\varepsilon = {\rm span}\, (Z_{j_*} \cup \{\overline{v}_{j_*}^\varepsilon\})$. Then $L\cap L^\varepsilon = Z_{j_*}$. Hence $\Lambda \cap L^\varepsilon = \Lambda \cap L\cap L^\varepsilon= \Lambda \cap Z_{j_*}=\{0\}$.

By successive arbitrarily small translations over all matrices ${\cal B}\in \hat{\cal M}_{m,I}$, $2\le m\le d+1$, $\# I=m$, and systems of points $\{1/\overline{p}_{\beta_j}\}_{j=1}^m$, we satisfy condition 3.

{\bf Substep 1.5.} We prove that estimate \eqref{main_eq} holds for arbitrary $\{(\nu_\alpha, \, \overline{p}_\alpha)\}_{\alpha \in A}$.

\begin{Lem}
\label{sdvig} Let $2\le m\le d+1$, $\{1/\overline{p}_j\}_{j=1}^{m-1} \subset [0, \, 1]^d$, $1/\overline{p}_m\in [0, \, 1]^d$, $1/\overline{p}_m^N\in [0, \, 1]^d$, $N\in \N$, $1/\overline{p}_m^N \underset{N\to \infty}{\to} 1/\overline{p}_m$, and let $\Lambda \subset \R^d$ be an affine plane of dimension $d-m+1$. Let $\Delta_N = {\rm conv}\, \{1/\overline{p}_1, \, \dots, \, 1/\overline{p}_{m-1}, \, 1/\overline{p}_m^N\}$, $\Delta = {\rm conv}\, \{1/\overline{p}_j\}_{j=1}^m$. Suppose that for each $N$ the simplex $\Delta_N$ is nondegenerate and intersects with $\Lambda$ an its interior point, and, in addition, the intersection point is unique. Then the simplex $\Delta$
is nondegenerate and intersects with $\Lambda$ at unique point.
\end{Lem}

\begin{proof}
Since $\Delta_N\cap \Lambda \ne \varnothing$ for each $N\in \N$, we have $\Delta \cap \Lambda \ne \varnothing$.

If $\Delta$ is degenerate or intersects with $\Lambda$ at two or more points, then $\Lambda \cup \Delta \subset V$, where $V\subset \R^d$ is a $d-1$-dimensional plane.

Notice that ${\rm conv}\, \{1/\overline{p}_j\}_{j=1}^{m-1}$ does not intersect with $\Lambda$; otherwise, $\Delta_N$ intersects with $\Lambda$ at its relative boundary point, which contradicts to conditions of the Lemma.

If $1/\overline{p}_m^N \notin V$, then $\Delta_N \cap V = {\rm conv}\, \{1/\overline{p}_j\}_{j=1}^{m-1}$; hence $\Delta_N \cap \Lambda=\varnothing$; we arrive to a contradiction.

Therefore, $1/\overline{p}_m^N \in V$; hence $\Delta_N \subset V$. We get that the $m-1$-dimensional simplex $\Delta_N$ and $d-m+1$-dimensional plane $\Lambda$ are contained in a $d-1$-dimensional plane and intersect at the relative interior point of $\Delta_N$. Thus, the set of intersection points has dimension at least 1; we again arrive to a contradiction.
\end{proof}

Denote by $\Psi_0(\{\overline{p}_\alpha\}_{\alpha \in A}, \, \overline{q}, \, \overline{k}, \, n)$ the right-hand side of \eqref{main_eq}.

First we notice that the function $[0, \, 1]^d\ni \overline{x}\stackrel{\varphi}{\mapsto} \Phi(1/\overline{x}, \, \overline{q}, \, \overline{k}, \, n)$ is continuous; indeed, $[0, \, 1]^d$ can be covered by a finite family of closed subsets $D_j$ $(1\le j\le j_0)$ with the following property: for each $j\in \{1, \, \dots, \, j_0\}$, the function $\varphi|_{D_j}$ is continuous (see formulas \eqref{phi_case10}, \eqref{phi_case20}, \eqref{phi_case30}). 

Let $\alpha \in A$, $\overline{p}_{\alpha}^N \in [1, \, \infty]^d$, $1/\overline{p}_{\alpha}^N \underset{N\to \infty}{\to} 1/\overline{p}_{\alpha}$. Given $\beta \ne \alpha$, we set $\overline{p}^N_{\beta} = \overline{p}_{\beta}$. The sets ${\cal N}_m^N$, the numbers $\lambda_j^N(\overline{\alpha}, \, Z)$ and the vectors $\overline{\theta}^N(\overline{\alpha}, \, Z)$ are defined according to Definition \ref{nm_def} for $\{\overline{p}^N_{\beta}\}_{\beta \in A}$. Then
\begin{align}
\label{psi_0_pnb}
\Psi_0(\{\overline{p}^N_{\beta}\}_{\beta\in A}, \, \overline{q}, \, \overline{k}, \, n) \underset{N\to \infty}{\to} \Psi_0(\{\overline{p}_{\beta}\}_{\beta\in A}, \, \overline{q}, \, \overline{k}, \, n).
\end{align}
The arguments are similar to \cite{vas_plq} (see also\cite{vas_mix_sev}), taking into account Lemma \ref{sdvig}. From \eqref{psi_0_pnb} and the above Substeps 1.2--1.4  it follows that we can make successive translations of the points $1/\overline{p}_\beta$ so as to eventually satisfy conditions 1--3 of Definition \ref{gen_pos}, and $\Psi_0(\{\overline{p}_{\beta}\}_{\beta\in A}, \, \overline{q}, \, \overline{k}, \, n)$ and $d_n(M, \, l^{\overline{k}}_{\overline{q}})$ will change by at most a factor 2.

{\bf Step 2.} We now change to the case of finite $A$ and arbitrary $\{(\nu_\alpha, \, \overline{p}_{\alpha})\}_{\alpha \in A}$, $2\le q_i<\infty$ ($i=1, \, \dots, \, d$); i.e., $q_i$ can be equal to 2.

Let $q_i=2$, $i\in J$, $q_i>2$, $i\notin J$. It suffices to consider the case when $\# J\le d-1$; in the case $q_1=\dots=q_d=2$, the estimates of the widths are obtained in \cite{vas_plq}.

Let $2<q^N_i<\infty$, $i\in J$, $q^N_i\underset{N\to \infty}{\to} 2$. For $i\notin J$, we set $q^N_i=q_i$ and denote $\overline{q}^N =(q_1^N, \, \dots, \, q_d^N)$.

We have already proved that
$$
d_n(M, \, l^{\overline{k}}_{\overline{q}^N}) \underset{d}{\gtrsim} \Psi_0(\{\overline{p}_\alpha\} _{\alpha\in A}, \, \overline{q}^N, \, \overline{k}, \, n).
$$
From properties of the Kolmogorov widths it follows that
$$
d_n(M, \, l^{\overline{k}}_{\overline{q}^N}) \underset{N\to \infty}{\to} d_n(M, \, l^{\overline{k}}_{\overline{q}}).
$$
It remains to prove that
\begin{align}
\label{psi_q_n_lim}
\Psi_0(\{\overline{p}_\alpha\} _{\alpha\in A}, \, \overline{q}^N, \, \overline{k}, \, n) \underset{N\to \infty}{\to} \Psi_0(\{\overline{p}_\alpha\} _{\alpha\in A}, \, \overline{q}, \, \overline{k}, \, n).
\end{align}

First we check that, for any $\overline{p} \in [1, \, \infty]^d$,
\begin{align}
\label{phi_q_n_lim}
\Phi(\overline{p}, \, \overline{q}^N, \, \overline{k}, \, n) \underset{N\to \infty}{\to} \Phi(\overline{p}, \, \overline{q}, \, \overline{k}, \, n).
\end{align}

Indeed, let $\sigma$ be a permutation satisfying \eqref{upor}. Then there are numbers $0\le \mu\le \nu\le d$ such that
\begin{enumerate}
\item $p_{\sigma(t)} > q_{\sigma(t)}$ or $p_{\sigma(t)} = q_{\sigma(t)} > 2$, $1\le t\le \mu$ (then $\omega _{p_{\sigma(t)}, q _{\sigma(t)}}=0$);

\item $0< \omega _{p_{\sigma(\mu+1)}, q _{\sigma(\mu+1)}} \le \omega _{p_{\sigma(\mu+2)}, q _{\sigma(\mu+2)}} \le \dots \le \omega _{p_{\sigma(\nu)}, q _{\sigma(\nu)}}<1$ (it implies that $q_{\sigma(t)}>2$ for $\mu +1 \le t\le \nu$);

\item $p_{\sigma(t)}\le 2$, $\nu+1\le t\le d$.
\end{enumerate}
If $t\le \mu$, then for sufficiently large $N$ we have $p_{\sigma(t)} > q^N_{\sigma(t)}$ or $p_{\sigma(t)} = q^N_{\sigma(t)}>2$; hence $\omega _{p_{\sigma(t)}, q^N _{\sigma(t)}}=0$. If $\mu+1\le t\le \nu$, then $q^N _{\sigma(t)} = q _{\sigma(t)}$; hence $0< \omega _{p_{\sigma(\mu+1)}, q^N _{\sigma(\mu+1)}} \le \omega _{p_{\sigma(\mu+2)}, q^N _{\sigma(\mu+2)}} \le \dots \le \omega _{p_{\sigma(\nu)}, q^N _{\sigma(\nu)}}<1$. If $t\ge \nu + 1$, then $p_{\sigma(t)}\le 2$ and $\omega _{p_{\sigma(t)}, q^N _{\sigma(t)}}=1$. Hence the analogue of \eqref{upor} holds, where $\overline{q}$ is replaced by $\overline{q}^N$.

Let $n< \prod _{i=1}^\mu k_{\sigma(i)} \prod _{i=\mu+1} ^d k_{\sigma(i)} ^{2/q_{\sigma(i)}}$. Then, for sufficiently large $N$, we have $n< \prod _{i=1}^\mu k_{\sigma(i)} \prod _{i=\mu+1} ^d k_{\sigma(i)} ^{2/q ^N_{\sigma(i)}}$; therefore, by \eqref{phi_case10} we have
$$
\Phi(\overline{p}, \, \overline{q}^N, \, \overline{k}, \, n) = \prod _{i=1}^\mu k_{\sigma(i)} ^{1/q^N_{\sigma(i)} - 1/p_{\sigma(i)}} \underset{N\to \infty}{\to}\prod _{i=1}^\mu k_{\sigma(i)} ^{1/q _{\sigma(i)} - 1/p_{\sigma(i)}} = \Phi(\overline{p}, \, \overline{q}, \, \overline{k}, \, n).
$$

Let $\mu+1\le t\le \nu$, $\prod _{i=1}^{t-1} k_{\sigma(i)} \prod _{i=t} ^d k_{\sigma(i)} ^{2/q_{\sigma(i)}} \le n < \prod _{i=1}^{t} k_{\sigma(i)} \prod _{i=t+1} ^d k_{\sigma(i)} ^{2/q_{\sigma(i)}}$. Then $q_{\sigma(t)}>2$, and for sufficiently large $N$ we have $$\prod _{i=1}^{t-1} k_{\sigma(i)} \prod _{i=t} ^d k_{\sigma(i)} ^{2/q ^N _{\sigma(i)}} \le n < \prod _{i=1}^{t} k_{\sigma(i)} \prod _{i=t+1} ^d k_{\sigma(i)} ^{2/q ^N _{\sigma(i)}}.$$ Hence, by \eqref{phi_case20}, we get
$$
\Phi(\overline{p}, \, \overline{q}^N, \, \overline{k}, \, n) = \prod _{i=1}^{t-1} k_{\sigma(i)} ^{1/q^N_{\sigma(i)} - 1/p_{\sigma(i)}} \Bigl( n^{-1/2} \prod _{i=1}^{t-1} k_{\sigma(i)}^{1/2} \prod _{i=t} ^d k_{\sigma(i)} ^{1/q ^N _{\sigma(i)}}\Bigr) ^{\frac{1/p_{\sigma(t)} - 1/q^{N}_{\sigma(t)}}{1/2 - 1/q^N _{\sigma(t)}}}\underset{N\to \infty}{\to} $$$$ \to \prod _{i=1}^{t-1} k_{\sigma(i)} ^{1/q_{\sigma(i)} - 1/p_{\sigma(i)}} \Bigl( n^{-1/2} \prod _{i=1}^{t-1} k_{\sigma(i)}^{1/2} \prod _{i=t} ^d k_{\sigma(i)} ^{1/q _{\sigma(i)}}\Bigr) ^{\frac{1/p_{\sigma(t)} - 1/q_{\sigma(t)}}{1/2 - 1/q _{\sigma(t)}}} = \Phi(\overline{p}, \, \overline{q}, \, \overline{k}, \, n).
$$

Let $n\ge \prod _{i=1}^{\nu} k_{\sigma(i)} \prod _{i= \nu +1} ^d k_{\sigma(i)} ^{2/q_{\sigma(i)}}$. Then $n\ge \prod _{i=1}^{\nu} k_{\sigma(i)} \prod _{i= \nu +1} ^d k_{\sigma(i)} ^{2/q ^N_{\sigma(i)}}$. Hence, by \eqref{phi_case30}, we have
$$
\Phi(\overline{p}, \, \overline{q}^N, \, \overline{k}, \, n) = \prod _{i=1}^{\nu} k_{\sigma(i)} ^{1/q^N_{\sigma(i)} - 1/p_{\sigma(i)}} n^{-1/2} \prod _{i=1}^{\nu} k_{\sigma(i)}^{1/2} \prod _{i= \nu+1} ^d k_{\sigma(i)} ^{1/q ^N _{\sigma(i)}} \underset{N\to \infty}{\to} $$$$ \to \prod _{i=1}^{\nu} k_{\sigma(i)} ^{1/q _{\sigma(i)} - 1/p_{\sigma(i)}} n^{-1/2} \prod _{i=1}^{\nu} k_{\sigma(i)}^{1/2} \prod _{i= \nu+1} ^d k_{\sigma(i)} ^{1/q  _{\sigma(i)}} = \Phi(\overline{p}, \, \overline{q}, \, \overline{k}, \, n).
$$

This completes the proof of \eqref{phi_q_n_lim}.

We define ${\cal Z}_m^N$, ${\cal N}_m^N$, $\overline{\theta} ^N(\overline{\alpha}, \, Z^N)$, $\lambda_j^N (\overline{\alpha}, \, Z^N)$ similarly to ${\cal Z}_m$, ${\cal N}_m$, $\overline{\theta} (\overline{\alpha}, \, Z)$, $\lambda_j (\overline{\alpha}, \, Z)$, replacing $\overline{q}$ by $\overline{q}^N$. The equations for $Z^N$ from parts 1, 2 of Definition \ref{zm_def} are the same as the equations for $Z$, with $\overline{q}$ replaced by $\overline{q}^N$. The equations from part 3 of Definition \ref{zm_def} are replaced by the equations $\omega'_{1/x_i,q_i^N}= \omega'_{1/x_l,q_l^N}$, $i$, $l\in I$.

Let us prove \eqref{psi_q_n_lim}. It is the consequence of \eqref{phi_q_n_lim} and the following assertions:
\begin{enumerate}
\item Let $\overline{\alpha} \in {\cal N}_m$. Then for sufficiently large $N$ we have $\alpha \in {\cal N}_m^N$ with the set $Z^N\in {\cal Z}_m^N$; in addition, $\overline{\theta} ^N(\overline{\alpha}, \, Z^N) \underset{N \to \infty}{\to} \overline{\theta}(\overline{\alpha}, \, Z)$, $\lambda_j ^N(\overline{\alpha}, \, Z^N) \underset{N \to \infty}{\to} \lambda_j(\overline{\alpha}, \, Z)$.

\item Let $\overline{\alpha} \notin {\cal N}_m$. Suppose that there is a subsequence $\{N_k\}_{k\in \N}$ such that $\overline{\alpha} \in {\cal N}_m^{N_k}$ (changing once again to subsequences, we may assume that the sets $Z^{N_k}\in {\cal Z}_m^{N_k}$ are defined by systems of equations that differ only by $\overline{q}^{N_k}$). Then there are $\overline{\beta} =(\alpha_{i_1}, \, \dots, \, \alpha_{i_l}) \in {\cal N}_l$ for some $l\le m$ and $Z'\in {\cal Z}_l$ such that $\overline{\theta} ^{N_k}(\overline{\alpha}, \, Z^{N_k}) \underset{k \to \infty}{\to} \overline{\theta}(\overline{\beta}, \, Z')$, $\lambda_j ^{N_k} (\overline{\alpha}, \, Z^{N_k}) \underset{k \to \infty}{\to} \lambda_j(\overline{\beta}, \, Z')$.
\end{enumerate}
In order to prove assertion 2, we consider the case when $Z^{N_k}$ are given by equations from part 3 of Definition \ref{zm_def}. Then $\#I=m$. We check that the intersection of the sets $\Delta_I := {\rm conv}\, \{(1/\overline{p}_{\alpha_j})_I\} _{j=1}^m$ and $[(1/\overline{q})_I, \, (1/\overline{2})_I]$ is nonempty, and their affine hulls intersect at the unique point.

The existence of the intersection point can be proved by passing to limit, since $\Delta_I \cap [(1/\overline{q}^{N_k})_I, \, (1/\overline{2})_I] \ne \varnothing$.

We prove the uniquness of the intersection point of the affine hulls. It suffices to consider the case when $\{i\in I:\; q_i>2\} \ne \varnothing$. Let $L={\rm aff}\, \Delta_I$. Since $\overline{\alpha}\in {\cal N}_m^{N_k}$, the equations defining $\lambda_j^{N_k}(\overline{\alpha}, \, Z^{N_k})$ have the unique solution. This implies that the system of vectors $\{(1/\overline{p}_{\alpha_j})_I\} _{j=1}^m$ is affinely independent. Therefore, $\dim L = m-1$. We get that either $L$ and ${\rm aff}\, [(1/\overline{q})_I, \, (1/\overline{2})_I]$ are complementary, or $[(1/\overline{q})_I, \, (1/\overline{2})_I] \subset L$.

Since $L$ and ${\rm aff}\, [(1/\overline{q}^N)_I, \, (1/\overline{2})_I]$ are complementary and their intersection point lies in the interior of $[(1/\overline{q}^N)_I, \, (1/\overline{2})_I]$, we have $(1/\overline{2})_I \notin L$; hence the case $[(1/\overline{q})_I, \, (1/\overline{2})_I] \subset L$ is impossible.

{\bf Step 3.} A transition from a finite to an arbitrary set $A$ proceeds as in \cite[\S 5]{vas_mix_sev}.

\end{document}